\documentclass[11pt,reqno]{amsart}

\usepackage{enumerate,amsmath,amsthm,amsfonts,amssymb,xy,combelow, mathrsfs,graphicx,paralist,fancyvrb,ytableau,braket,xcolor,verbatim,tikz,tikz-cd,mathtools,enumitem,comment,stmaryrd}
\usepackage{quiver,mathabx,adjustbox}

\usepackage[bookmarks, colorlinks=true, linkcolor=blue, citecolor=blue, urlcolor=blue]{hyperref}

\tikzset{
commutative diagrams/.cd,
arrow style=tikz,
diagrams={>=latex}}

\allowdisplaybreaks

\input xy
\xyoption{all}

\DeclareMathOperator{\mor}{\mathsf{Mor}}
\DeclareMathOperator{\quot}{\mathsf{Quot}}
\DeclareMathOperator{\Hquot}{\mathsf{HQ}}

\DeclarePairedDelimiter{\parens}{\lparen}{\rparen}

\renewcommand{\tilde}[1]{\widetilde{#1}}
\renewcommand{\bar}[1]{\overline{#1}}

\newtheorem{theorem}{Theorem}[section]
\newtheorem{lemma}[theorem]{Lemma}
\newtheorem{corollary}[theorem]{Corollary}
\newtheorem{proposition}[theorem]{Proposition}

\newtheorem*{theorem*}{Theorem}
\usepackage{comment}

\theoremstyle{definition}
\newtheorem{definition}[theorem]{Definition}
\newtheorem{remark}[theorem]{Remark}

\newtheorem{problem}[theorem]{Problem}

\newtheoremstyle{TheoremNum}
        {7pt}{7pt}              
        {\itshape}                      
        {}                              
        {\bfseries}                     
        {.}                             
        { }                             
        {\thmname{#1}\thmnote{ \bfseries #3}}
    \theoremstyle{TheoremNum}

\newcommand{\BA}{\mathbb{A}}
\newcommand{\BC}{\mathbb{C}}
\newcommand{\BE}{\mathbb{E}}
\newcommand{\BL}{\mathbb{L}}
\newcommand{\BN}{\mathbb{N}}
\newcommand{\BP}{\mathbb{P}}

\newcommand{\BZ}{\mathbb{Z}}
\newcommand{\Ba}{\mathbf{a}}
\newcommand{\Bd}{\mathbf{d}}
\newcommand{\Be}{\mathbf{e}}
\newcommand{\Bm}{\mathbf{m}}
\newcommand{\Bq}{\mathbf{q}}
\newcommand{\Br}{\mathbf{r}}
\newcommand{\Bt}{\mathbf{t}}
\newcommand{\Bw}{\mathbf{w}}
\newcommand{\Bx}{\mathbf{x}}
\newcommand{\Bz}{\mathbf{z}}

\newcommand{\rL}{\mathrm{L}}
\newcommand{\rR}{\mathrm{R}}

\newcommand{\CA}{\mathcal{A}}
\newcommand{\CV}{\mathcal{V}}
\newcommand{\CE}{\mathcal{E}}
\newcommand{\CF}{\mathcal{F}}

\newcommand{\CI}{\mathcal{I}}
\newcommand{\CK}{\mathcal{K}}
\newcommand{\CL}{\mathcal{L}}

\newcommand{\CO}{\mathcal{O}}

\newcommand{\CX}{\mathcal{X}}

\newcommand{\Bun}{\mathfrak{Bun}}
\newcommand{\uBun}{{\Bun}}
\newcommand{\pr}{\mathrm{pr}}
\newcommand{\fE}{\mathfrak{E}}
\newcommand{\SpecSym}{\mathrm{Spec\,Sym}}

\newcommand{\Nor}{N^{\vir}_{\sigma,D}}
\DeclareMathOperator*{\Res}{Res}
\DeclareMathOperator{\Pic}{Pic}
\newcommand{\Tan}{T^{\vir}}

\newcommand{\fix}{\mathrm{X}}
\newcommand{\rank}{\mathrm{rank}}

\newcommand{\vir}{\mathrm{vir}}

\newcommand{\vdim}{\mathrm{vdim}}
\newcommand{\Hom}{\operatorname{\mathcal{H}\! \mathit{om}}}
\newcommand{\fund}{\mathrm{fund}}
\newcommand{\Bdelta}{\boldsymbol{\delta}}
\newcommand{\Bzero}{\boldsymbol{0}}
\newcommand{\Boldeta}{\boldsymbol{\eta}}

\newcommand{\bb}{\mathbb}
\newcommand{\ca}{\mathcal}
\newcommand{\fr}{\mathfrak}

\newcommand{\Gr}{\mathrm{Gr}}
\newcommand{\Fl}{\mathrm{Fl}}

\newcommand{\pt}{\mathrm{pt}}

\title{Intersection Theory of Hyperquot Schemes on curves}
    \author[R.~Ontani]{Riccardo Ontani}
	\address{Department of Mathematics, Imperial College London, London SW7 2AZ,
United Kingdom}
    \email{r.ontani@imperial.ac.uk}
    \author[S.~Sinha]{Shubham Sinha}
	\address{International Centre for Theoretical Physics, Strada Costiera 11, 34151 Trieste, Italy}
     \email{ssinha1@ictp.it}
	\author[W.~Xu]{Weihong Xu}
	\address{The Division of Physics,
Mathematics and Astronomy, California Institute of Technology, Pasadena, CA 91125, U.S.A.}
    \email{weihong@caltech.edu}

\begin{document}
\begin{abstract}
   We study the virtual intersection theory of Hyperquot schemes parameterizing sequences of quotient sheaves of a vector bundle on a smooth projective curve. Our results generalize the Vafa--Intriligator formula for Quot schemes and provide a closed formula for virtual counts of maps from the curve to a partial flag variety. 
\end{abstract}
\maketitle

\tableofcontents

\section{Introduction}
Let $C$ be a smooth projective curve and $V$ be a vector bundle of rank $n$ over $C$. The Quot scheme $\quot_d(C, r, V)$ is a projective scheme that parametrizes short exact sequences of sheaves on $C$:
\[
0 \to E \xhookrightarrow{} V \twoheadrightarrow F \to 0 \quad \text{such that } 
\begin{cases}
    \rank\, F = n - r, \\
    \deg F = d.
\end{cases}
\]
When \(V\) is trivial, the Quot scheme compactifies the space of degree $d$ maps from
$C$ to the Grassmannian $\Gr(r,n)$, and it was used to prove the Vafa--Intriligator formula for counts of such maps satisfying some incidence conditions; see \cite{stromme,bertram,bertram-deskalopoulos-wentworth,mo}.

The Hyperquot scheme is a natural generalization of the Quot scheme. For fixed tuples $$\Br = (r_1, r_2, \dots, r_k)\quad  \text{and}\quad \Bd = (d_1, d_2, \dots, d_k)$$ of integers, the Hyperquot scheme $\Hquot_\Bd(C, \Br, V)$, or \(\Hquot_\Bd\) for short, parametrizes sequences of sheaves of the form
\[
E_1\xhookrightarrow{}E_2\xhookrightarrow{}\cdots E_k\xhookrightarrow{}V \twoheadrightarrow F_1 \twoheadrightarrow F_2 \twoheadrightarrow \cdots \twoheadrightarrow F_k \quad \text{such that } 
\begin{cases}
    \rank\, F_j = n - r_j, \\
    \deg F_j = d_j,
\end{cases}
\]
where $E_j$ is the kernel of the surjection $V\twoheadrightarrow F_j$ for each $j$.

In this article, we establish a formula for virtual intersection numbers on the Hyperquot scheme, generalizing the work of Bertram~\cite{bertram} and Marian--Oprea~\cite{mo}. In the special case where $V$ is trivial and $r_1, r_2, \dots, r_k$ are distinct, the Hyperquot scheme compactifies the space of multidegree \(\Bd\) maps from $C$ to the partial flag variety $\Fl(\Br,n)$, and our formula provides a virtual count of such maps satisfying certain incidence conditions. Furthermore, when $C = \mathbb{P}^1$, the Hyperquot scheme is a smooth variety and the virtual counts in consideration are actual counts \cite{kim1996gromov,cf2} that played an important role in the study of the {small} quantum cohomology of $\Fl(\Br,n)$; see, for example, \cite{kim,bertram2,cf, cf2,chen}.

\subsection{Statement of the main formula} 
The Hyperquot scheme $\Hquot_\Bd$ is not smooth or irreducible in general. 
By \cite{behrend_fantechi}, one can nevertheless 
construct a virtual fundamental class, which is a well-behaved replacement for the fundamental class. This was constructed for the Quot scheme in \cite{mo} and for the Hyperquot scheme in \cite{ckm,mr}, yielding
\[
[\Hquot_\Bd]^{\vir} \in H_{2\vdim}(\Hquot_\Bd, \mathbb{Z})
\]
whose virtual dimension is
$
\vdim = \sum_{j=1}^{k} \chi(E_j, F_j) \;-\; \sum_{j=1}^{k-1} \chi(E_j, F_{j+1}).
$
Consider the universal sequence of sheaves 
\[
0\to \CE_{1}\xhookrightarrow{} \cdots \xhookrightarrow{} \CE_{{k}} \xhookrightarrow{} \pi_C^*V \twoheadrightarrow \CF_{1}\twoheadrightarrow \cdots \twoheadrightarrow \CF_{k}\to 0
\]
on $\Hquot_{\Bd}\times C$, where $\pi_C$ denotes the projection map to  $C$. We will compute the integrals over $[\Hquot_\Bd]^{\vir}$ of polynomials in the Chern classes 
 \begin{equation}\label{eq:a_i,j}
    c_i(\CE_{j|p}^\vee)\in H^{2i}(\Hquot_\Bd,\BZ),
\end{equation}
where $\CE_{j|p}$ is the rank $r_j$ locally free sheaf obtained by restricting $\CE_{j}$ to $\Hquot_\Bd\times \{p\}$ for a point $p\in C$.
Our main result expresses these intersection numbers as a sum over the solutions of an explicit system of polynomial equations, which we now describe in detail.

Consider $k$ tuples of variables $\Bz_1,\Bz_2,\dots,\Bz_k$ where 
    $\Bz_{j}\coloneqq (z_{1,j},z_{2,j},\dots,z_{r_j,j}),$ and set $\Bz_0=\{\},\ \Bz_{k+1} = (0, 0, \dots, 0)\in \BC^n,$ as illustrated below.
 \[
		\begin{tikzpicture}
		\draw (-0.5,0) -- (7,0) ;
		\draw (0,0.2) node {$\textbf{z}_0$};
		\draw (1,0.2) node {$\textbf{z}_1$};
		\draw (2,0.2) node {$\textbf{z}_2$};
		\draw (3,0.2) node {$\cdots$};
		\draw (4,0.2) node {$\cdots$};
		\draw (5,0.2) node {$\textbf{z}_k$};
		\draw (6,0.2) node {$\textbf{z}_{k+1}$};

		\draw (3,-0.4) node {$\cdots$};
		\draw (4,-0.4) node {$\cdots$};
		
		\draw (3,-2) node {$\ddots$};
		\draw (4,-2.5) node {$\ddots$};
		
		\draw (1,-0.4) node {\small $z_{1,1}$};
		\draw (1,-0.8) node {\small $\vdots$};
		\draw (1,-1.2) node {\small $z_{r_1,1}$};

		\draw (2,-0.4) node {\small $z_{1,2}$};
		\draw (2,-0.8) node {\small $\vdots$};
		\draw (2,-1.2) node {\small $\vdots$};
		\draw (2,-1.7) node {\small $z_{r_2,2}$};
		
		\draw (5,-0.4) node {\small $z_{1,k}$};
		\draw (5,-0.8) node {\small $\vdots$};
		\draw (5,-1.2) node {\small $\vdots$};
		\draw (5,-1.7) node {\small $\vdots$};
		\draw (5,-2.1) node {\small $\vdots$};
		\draw (5,-3) node {\small $z_{r_k,k}$};
		
		\draw (6,-0.4) node {\small $0$};
		\draw (6,-0.8) node {\small $\vdots$};
		\draw (6,-1.2) node {\small $\vdots$};
		\draw (6,-1.7) node {\small $\vdots$};
		\draw (6,-2.1) node {\small $\vdots$};
		\draw (6,-2.5) node {\small $\vdots$};
		\draw (6,-2.9) node {\small $\vdots$};
		\draw (6,-3.5) node {\small $0$};
	\end{tikzpicture}
	\]
 For each $1\le j\le k$, define the polynomial 
 \begin{align}\label{eq:Bethe_P(X)}
		P_j(X)\coloneqq \prod_{\alpha\in \Bz_{j+1}}(X-\alpha)+(-1)^{r_{j}-r_{j-1}}q_j\prod_{\alpha\in \Bz_{j-1}}(X-\alpha).
\end{align}
 \begin{definition}\label{def:Bethe_ansatz}
 Let $(q_1,q_2,\dots,q_k)$ be a $k$-tuple of complex numbers.
 We define the system of polynomial equations\footnote{These equations have appeared in several other contexts; see Section~\ref{sec:related} for more details. }
 \begin{align}\label{eq:Bethe_system_intro}
 P_j(z_{i,j})=0\quad \text{for each } 1\le j\le k, \text{ and } 1\le i\le r_j. 
 \end{align}
 We say that a solution $(\Bz_1,\Bz_2\dots,\Bz_k)=(\boldsymbol{\zeta}_1,\boldsymbol{\zeta}_2\dots,\boldsymbol{\zeta}_k)$ to the above system is \emph{non-degenerate} if each tuple $\boldsymbol{\zeta}_j = (\zeta_{1,j},\zeta_{2,j},\dots,\zeta_{r_j,j})$ has no repetition. 
\end{definition}
 We are now ready to state our main result; see Theorem~\ref{thm:VI_Hyper_quot_equivariant} for an equivariant version. 
\begin{theorem}\label{thm:intro_VI_formula}
	 Fix a tuple $\Br=(r_1,\dots,r_k)$ and a vector bundle $V$ on $C$ of rank $n$ with $\deg V= 0$. Then for natural numbers \(m_{i,j}\) and a generic $(q_1, \dots, q_k) \in (\BC^\ast)^k$, we have
	\begin{align}\label{eq:Intersection_Hquot(C,V)}
    \sum_{\Bd\in \BN^k} q_1^{d_1}\cdots q_k^{d_k}
\int_{[\Hquot_{\Bd}]^{\vir}}  
\prod_{j=1}^{k}\prod_{i=1}^{r_j} c_i(\CE_{j|p}^\vee)^{m_{i,j}}
    =\sum_{\boldsymbol{\zeta}_1,\dots,\boldsymbol{\zeta}_k}  \prod_{j=1}^{k}\frac{1}{r_j!}\prod_{i=1}^{r_j}e_i(\boldsymbol{\zeta}_j)^{m_{i,j}}\cdot J^{g-1},
	\end{align}
    where the sum is taken over all non-degenerate solutions $(\boldsymbol{\zeta}_1,\dots,\boldsymbol{\zeta}_k)$ of the equations in~\eqref{eq:Bethe_system_intro} and $e_i(\boldsymbol{\zeta}_j)$ denotes the $i$-th elementary symmetric polynomial in the $r_j$ variables \(\boldsymbol{\zeta}_j\). 
     The factor $J$ is given by
     \[ J\coloneqq J(\boldsymbol{\zeta}_1,\dots,\boldsymbol{\zeta}_k)\coloneqq \left( \prod_{\ell=1}^{k} {\frac{1}{\Delta(\boldsymbol{\zeta}_\ell)}}\right)\cdot
\det\left(\frac{\partial P_j(z_{i,j})}{\partial z_{i',j'}}\right)\Bigg|_{(\Bz_1,\Bz_2,\dots,\Bz_k) = (\boldsymbol{\zeta}_1,\boldsymbol{\zeta}_2,\dots,\boldsymbol{\zeta}_k)}
 ,\]
where $ \Delta(X_1,\dots,X_r)\coloneqq\prod_{a\ne b}(X_a-X_b)$ and the second factor is the Jacobian of \eqref{eq:Bethe_system_intro}. 
\end{theorem} 
\begin{remark}
The expression in \eqref{eq:Intersection_Hquot(C,V)} is a polynomial in the variables $q_1, q_2, \dots, q_k$, depending only on the genus $g$, the ranks $\Br$ and the exponents $m_{i,j}$. In practice, this polynomiality allows one to substitute generic complex values for $q_1, q_2, \dots, q_k$ to solve \eqref{eq:Bethe_system_intro} numerically and then use interpolation to recover the full polynomial. 
\end{remark}

\begin{remark} 
    Theorem \ref{thm:intro_VI_formula} is stated for bundles $V$ of degree zero, but a completely analogous formula holds when the degree is nonzero. In that case, $V$ can be reduced to a degree zero bundle by twisting with some line bundle and applying a sequence of elementary modifications. In Proposition \ref{prop:elementary_mod}, we show a simple identity that relates the virtual class of $\Hquot_\Bd(C, \Br, V)$ to the one of $\Hquot_\Bd(C, \Br, \tilde V)$, where $\tilde V$ is an elementary modification of $V$, extending the analogous result proven in \cite[Section~2]{marian25} for Quot schemes. 
    \end{remark}
    \begin{remark}
    We also prove a compatibility result of virtual classes in Proposition \ref{prop:compatibility_twisting}, analogous to \cite[Theorem~2]{mo}, keeping $V$ fixed but changing the degree tuple $\Bd$. This is obtained by twisting the $\ell$-th subbundle by $\CO_C(-\pt)$ to obtain an embedding $\iota:\Hquot_{\Bd}\to \Hquot_{\Bd+\Boldeta_\ell}$ (where $\Boldeta_\ell=(r_1,\dots,r_\ell,0,\dots,0)$) and showing
  \begin{align*}
        \iota_* [\Hquot_{\Bd}]^{\vir} = e(\CE_{\ell}^\vee\otimes \CE_{\ell+1})_{\vert p}\cap [\Hquot_{\Bd+\Boldeta_\ell}]^{\vir}.
    \end{align*}
    The consequence of this compatibility at the level of the intersection numbers can be directly observed in formula \eqref{eq:Intersection_Hquot(C,V)}.
\end{remark}

\subsection{Specializations}\label{sec:spec} 
We now state the \(k=1\) case of our result in the equivariant setting; see Section~\ref{sec:equiv_localization} for more details.  
Let the torus $T = (\mathbb{C}^*)^n$ act on $\quot_d(C, r, \CO_C^{\oplus n})$ by scaling the fibers of $\CO_C^{\oplus n}$, and let \(\varepsilon_1,\dots, \varepsilon_n\in H_T^*(\pt)\) be the corresponding equivariant parameters.
\begin{corollary}
For non-negative integers \(m_1,\dots, m_r\) and generic values of \((\varepsilon_1,\dots,\varepsilon_n;q)\),
we have 
\begin{align}\label{eq:VI_equivariant}
    \sum_{d\ge 0}q^d\int_{[\quot_d(C, r, \mathcal{O}_C^{\oplus n})]^{\vir}} \prod_{i=1}^{r} c_i^{T}(\CE_{|p}^\vee)^{m_i}
= \sum_{\boldsymbol{\zeta}} \prod_{i=1}^{r}e_i(\boldsymbol{\zeta})^{m_{i}}\cdot J^{g-1},
	\end{align}
    where the sum is taken over all unordered \(r\)-tuples \(\boldsymbol{\zeta} = (\zeta_1, \dots, \zeta_r)\) of distinct roots of the equation 
    \[
(X-\varepsilon_1)\cdots(X-\varepsilon_n) + (-1)^{ r} q=0,
    \]
 and the factor $J$ is given by
     \[ J\coloneqq
     \prod_{i=1}^r\sum_{\ell=1}^n\frac{\prod_{s\leq n,\ s\neq \ell}\parens*{\zeta_i-\varepsilon_s}}{\prod_{t\leq r,\ t\neq i}\parens*{\zeta_i-\zeta_t}}.    
 \] 
\end{corollary}
\begin{remark}\label{rem:ST_formula}
Specializing further to the non-equivariant setting (i.e., letting all \(\varepsilon_s\) go to \(0\)), we recover the Vafa--Intriligator formula 
of \cite[Theorem 3]{mo}.
    An analogous Vafa--Intriligator formula for fixed-domain Gromov--Witten invariants on the Grassmannian (which are defined using the moduli spaces of stable maps) was proven in \cite{siebert.tian:on} and the two formulas match. 
\end{remark}
\begin{remark}
   In \cite{mo}, the virtual intersections of more general classes arising from the Künneth decomposition of $c_i(\CE^\vee)$ on $\quot_d(C, r, \mathcal{O}_C^{\oplus n})\times C$ were also considered. One could study the analogous problem for Hyperquot schemes. We do not pursue this direction in the present paper; however, our methods may prove useful for such investigations.
\end{remark}

The \textit{Hyperquot schemes of points} $\Hquot_{\mathbf{d}}(C, V)$ parameterizing successive zero-dimensional quotients of $V$ of degrees $\mathbf{d} = (d_1, d_2, \dots, d_k)$, 
are smooth projective varieties; in our notation, they have ranks $\mathbf{r} = (n, n, \dots, n)$. They play an important role in the study of the cohomology \cite{marian_negut1} and the derived category \cite{Krug,marian_negut2} of the Quot schemes of points on \(C\). Theorem~\ref{thm:intro_VI_formula} implies the following simple formula (see Corollary~\ref{thm:punctual} for the Chern class version):
\begin{corollary}\label{cor:Segra_integral}
Given a rank $n$ vector bundle $V$ on $C$, we have
    \begin{align*}
    \sum_{\Bd\in \BN^k}
q_1^{d_1}\cdots q_k^{d_k} \int_{\Hquot_{\Bd}(C,V)}  \prod_{j=1}^{k}s_{t_j}(\CE_{j|p}) =\prod_{j=1}^{k}\frac{1}{1-t_j^n\alpha_j}, 
\end{align*}
where $s_{t_j}(\CE_{j|p})$ are Segre polynomials and \(\alpha_j{= \sum_{a = j}^k \prod_{b=1}^a q_b}\).
\end{corollary}
\begin{remark}
     The Segre integrals and Euler characteristics of more general tautological classes on the Quot schemes of points on~$C$ were computed in \cite{oprea_pandharipande,Oprea-Shubham}. Our methods may be useful for evaluating analogous invariants on Hyperquot schemes of points. 
\end{remark}

 We find another simple formula
 for ranks $\Br = (1,n-1)$. We explicitly solve the system of equations~\eqref{eq:Bethe_system_intro}, and Theorem~\ref{thm:intro_VI_formula} specializes to the following expression involving binomial coefficients.
\begin{corollary}\label{cor:intro_r=(1,n-1)}
    Fix $\mathbf{r}=(1,n-1)$ and a rank $n$ bundle $V$ on $C$ of degree zero. For a fixed tuple of natural numbers $\ell,m_1,m_2,\dots,m_{n-1}$, we have
   \begin{align*}
\sum_{\Bd\in\BN^{2}}q_1^{d_1}q_2^{d_2}\int_{[\Hquot_\Bd]^{\vir}}
c_1(\CE_{1|p}^\vee)^{\ell}\prod_{i=1}^{n-1} c_i(\CE_{2|p}^\vee)^{m_i}
= 
n^{g}(n-1)^{g} \sum_{j\in\mathbb{Z}}
\binom{d-\bar{g}-m_{n-1}}{jn-\ell-m_{n-1}+ \bar{g}} q_1^{\bar{g} + j} q_{2}^{d-\bar{g}-j},
   \end{align*}
   where \(\overline{g}\coloneqq g-1\)  and   \[
    d\coloneqq \frac{\ell+\sum_{i=1}^{n-1}im_i+(2n-3)\bar{g}}{n-1}.
    \]
    Here we set the binomial coefficients $\binom{p}{q}=0$ unless $p,q\in \BN$ and $q\le p$.
\end{corollary}

\subsection{Related work and future directions}\label{sec:related}
When the ranks \(r_1,\dots,r_k\) are distinct, let $\mor_{\Bd}(C,\Fl(\Br,n))$ be the moduli space parametrizing morphisms $f$ from $C$ to the partial flag variety $\Fl(\Br,n)$ of multidegree $\Bd=(d_1,\dots,d_k)$, that is,
$$f_*[C] = \Bd\in H_2(\Fl(\Br,n), \BZ)\cong \BZ^{k}.$$ The basis for the homology group above is given, via the universal coefficient isomorphism $H_2(\Fl, \BZ) \cong H^2(\Fl, \BZ)$, by the first Chern classes of the dual bundles to the universal subbundles $\CA_1,\dots,\CA_k$ on $\Fl(\Br,n)$. The Hyperquot scheme $\Hquot_{\Bd}(C,\Br, \CO_C^{\oplus n})$ compactifies $\mor_{\Bd}(C,\Fl(\Br,n))$, which suggests the following:
\begin{problem}\label{ques:enumerativity}
    Find conditions under which the virtual intersection numbers on the Hyperquot scheme $\Hquot_\Bd(C,\Br,\CO_C^{\oplus n})$ in Theorem~\ref{thm:intro_VI_formula} are enumerative, that is,
\begin{equation}\label{eq:enum}
    \int_{[\Hquot_{\Bd}]^{\vir}}  
\prod_{\ell=1}^{t} c_{i_\ell}(\CE_{j_{\ell}|p}^\vee) = \#
\left\{
f:C\to \Fl(\Br,n) \mid f_*[C]  = 
\Bd,\ f(p_{\ell})\in Y_{i_\ell,j_\ell}\text{ for all }\ell
\right\}, 
\end{equation}
where $p_1,\dots,p_t\in C$ are distinct points, and 
$Y_{i_1,j_1}, \dots,Y_{i_t,j_t}$ are the special Schubert subvarieties of 
$\Fl(\Br,n)$ Poincar\'e dual to the classes 
$c_{i_1}(\CA_{j_1}^\vee), \dots,c_{i_t}(\CA_{j_t}^\vee)$, 
placed in general position so that the right-hand side is finite.
\end{problem}
If $C=\BP^1$, the Hyperquot scheme in Problem~\ref{ques:enumerativity} is smooth and irreducible, and the enumerativity of the intersection numbers was established in \cite{kim1996gromov,cf2} for all multidegrees~$\Bd$. The enumerativity of the virtual counts in the Grassmannian case (i.e., for $\quot_d(C,r,\CO_{C}^{\oplus n})$) was proved by Bertram~\cite{bertram} for sufficiently large degrees~$d$. More recently, this approach was also used in~\cite{marian2025short} to count maps to Fano hypersurfaces in Grassmannians. We plan to pursue Problem~\ref{ques:enumerativity} in our upcoming work.\\

Let us summarize, to the best of our knowledge, the occurrences of our system of equations~\eqref{eq:Bethe_system_intro} in the literature.
In physics, they arise as the vacuum equations of a two-dimensional gauged 
linear sigma model; equivalently, they define the critical locus of the 
twisted effective superpotential; see \cite[Section 4.6.2]{gu2023quantum} and \cite{kim_oh_ueda_yoshida}. Our equations also appear as relations describing the small quantum cohomology of the Grassmannian in \cite{siebert.tian:on} and of the partial flag variety in \cite{GuKa}.
In the Grassmannian case, these are also the Bethe ansatz equations of \cite[(4.17)]{Gorbounov:2014}, which arise naturally in integrable systems.

In particular, we can view formula \eqref{eq:Intersection_Hquot(C,V)} as a sum over the geometric generic fiber of the spectrum of the small quantum cohomology of the partial flag variety \(\Fl(\mathbf{r},n)\) over the affine space of quantum parameters \(
q_1,\dots,q_k\). Moreover, the properties of the solutions to the system of equations we derive in Section~\ref{sec:Equations} imply that the quantum cohomology of \(\Fl(\mathbf{r},n)\) is generically semisimple. Our results may have other applications in the contexts mentioned above.

A different approach to counting the maps in \eqref{eq:enum} is via fixed-domain Gromov--Witten invariants. When \(k=1\), the approach is known to yield the same intersection numbers as the Quot scheme approach in all degrees, as discussed in Remark \ref{rem:ST_formula}, regardless of their enumerativity. Fixed-domain Gromov--Witten invariants are studied for certain nonsingular projective varieties in 
\cite{chaput.manivel.perrin, buch2021tevelev,Cela_Lian,Cela}. It would be interesting to find an answer to the following:
\begin{problem}\label{prob:gw}
    When \(V\) is trivial and the ranks are distinct, does our virtual intersection number \(\int_{[\Hquot_{\Bd}]^{\vir}}  
\prod_{\ell=1}^{t} c_{i_\ell}(\CE_{j_{\ell}|p}^\vee)\) always agree with the multidegree \(\Bd\) fixed-domain Gromov--Witten invariant on $\Fl(\Br,n)$ associated to $c_{i_1}(\CA_{j_1}^\vee), \dots,c_{i_t}(\CA_{j_t}^\vee)$? 
\end{problem}

By \cite{Siebert} and the unpublished work \cite{bp2022vafa}, it is expected that fixed-domain Gromov--Witten invariants on smooth projective Fano varieties can be computed in the small quantum cohomology ring and that, for generalized flag varieties \(G/P\), they are always non-negative when the insertions are Schubert classes. This suggests that whenever these virtual intersection numbers \(\int_{[\Hquot_{\Bd}]^{\vir}}  
\prod_{\ell=1}^{t} c_{i_\ell}(\CE_{j_{\ell}|p}^\vee)\) agree with fixed-domain Gromov--Witten invariants, they must be non-negative. This positivity is also reflected in Corollaries~\ref{cor:Segra_integral} (where the ranks are not distinct and \(V\) is arbitrary) and~\ref{cor:intro_r=(1,n-1)}. We speculate that positivity holds in general and pose the following:

\begin{problem}
     Prove that each virtual integral in \eqref{eq:Intersection_Hquot(C,V)} is non-negative and find positive combinatorial formulas for these invariants.
\end{problem}

As mentioned earlier, when the ranks are distinct, the Hyperquot scheme $\Hquot_\Bd(C,\Br,\CO_C^{\oplus n})$ compactifies the space of maps $\mor_{\Bd}(C,\Fl(\Br,n))$. More precisely, $\Hquot_\Bd(C,\Br,\CO_C^{\oplus n})$ is known to be the moduli space of $0^+$-stable graph quasimaps of multidegree $\Bd$ from $C$ to the partial flag variety $\Fl(\Br,n)$ with no marked points \cite{ckm}, where the target $\Fl(\Br,n)$ is given in its standard presentation as a geometric invariant theory (GIT) quotient.

In this sense, our work fits into the literature computing quasimap invariants of Fano quotients of linear spaces.
When the target is a Fano toric variety $V/\!/T$ (presented as a quotient of a linear space $V$ by the action of a torus $T$), \cite{SzenesVergne} proved a Vafa--Intriligator formula for the generating series of genus zero quasimap invariants.
In the case of positive quotients of linear spaces by actions of (non-abelian) reductive groups, a similar formula was found by the first author \cite{ontani2025vafa} via abelianization, which relates the invariants of $V/\!/G$ to those of $V/\!/T$. When the target is \(\Fl(\Br,n)\), this formula agrees with the genus zero specialization of (\ref{eq:Intersection_Hquot(C,V)}), which we prove by very different methods.\\

A $K$-theoretic analog of the Vafa--Intriligator formula \eqref{eq:VI_equivariant} for Euler characteristics of products of exterior powers (more generally, Schur functors) of the tautological bundle $\CE_{p}$ on Quot schemes was obtained in \cite{SinhaZhang2} via torus localization. It was applied in \cite{SinhaZhang1} to study the quantum $K$-theory of Grassmannians, producing new proofs and formulas for genus-zero $K$-theoretic Gromov--Witten invariants. In the genus-zero setting, individual sheaf cohomology groups of certain Schur functors of $\CE_{p}$ were computed in \cite{Gautam_Lin_Sinha}. It would be interesting to prove a $K$-theoretic analog of Theorem~\ref{thm:intro_VI_formula} and relate it to the quantum $K$-theory of the partial flag variety. 

\subsection{Sketch of proof}
We use deformation invariance of virtual integrals (Proposition~\ref{prop:reduction_split_bundles}) to reduce to the case where \begin{equation*}
    V=M_1\oplus M_2\oplus\cdots\oplus M_{n}
\end{equation*}
and each  $M_i\to C$ is a line bundle of degree zero. We then consider the action of \(\parens*{\BC^*}^n\) scaling each \(M_j\) and apply torus localization to reduce the calculation to summing over contributions from the components of the fixed locus, which are products of Hilbert schemes on \(C\). We handle the summation carefully by utilizing the multivariate Lagrange--B\"urmann formula. This allows us to prove, in genus zero, that the generating polynomial of equivariant virtual integrals is a sum over formal power series solutions to the system of equations (see Theorem~\ref{thm:VI_Hyper_quot_equivariant_formal}). We prove in Proposition~\ref{prop:genus0eval} that this generating polynomial can be evaluated by first specializing \(\Bq\) and the equivariant parameters to generic values, and then summing over non-degenerate solutions; the proof applies Gr\"obner degeneration to the system of equations~\eqref{eq:Bethe_system_intro}. In higher genus, we need some intricate intersection-theoretic computations on products of Hilbert schemes on $C$, which we handle in Section~\ref{sec:symmetric_product_of_curves}.

\subsection*{Acknowledgments}
We are grateful to Anders Buch, Alessio Cela, Barbara Fantechi, Ajay Gautam, Thomas Graber, Christian Korff, Alina Marian, Leonardo Mihalcea, Dragos Oprea, Andrea Ricolfi and Richard Thomas for useful discussions.

\section{Perfect obstruction theory and virtual fundamental class}\label{sec:virtual_class}
In this section, we recall some foundational results about the virtual intersection theory on Hyperquot schemes.
Let $B$ be a scheme and consider a vector bundle $V$ on $B\times C$ over $B$. The relative Hyperquot scheme 
\begin{align}\label{eq:hquot_scheme}
    f:\Hquot_\Bd(C,\Br,V)\to B
\end{align} 
can be constructed as a tower of relative Quot schemes as we outline below in the two-step case. Let $\Br=(r_1,r_2)$ and $\Bd=(d_1,d_2)$. The first step is to consider the relative Quot scheme 
$\quot_{d_2}(C,r_2,V)\to B.
$ 
Let $\CE_2$ be the universal subsheaf over $\quot_{d_2}(C,r_2,V)\times C$, and consider the relative Quot scheme 
$
\quot_{d_1-d_2}(C,r_1,\CE_2)\to \quot_{d_2}(C,r_2,V). 
$
It is easy to check that the composition
\begin{equation}\label{eq:tower_construction}
    \quot_{d_1-d_2}(C,r_1,\CE_2)\to \quot_{d_2}(C,r_2,V)\to B
\end{equation}
is the relative Hyperquot scheme $\Hquot_\Bd(C,\Br,V)$ over $B$. This construction can be easily iterated for flags of any length.

\subsection{Obstruction theory}\label{sec:obstruction_theory}
On the product $\Hquot_\Bd(C, \Br, V) \times C$, we have the universal subsheaves
\begin{align*}
    0 =: \mathcal{E}_0 \hookrightarrow \mathcal{E}_1 \hookrightarrow \dots \hookrightarrow \mathcal{E}_k \hookrightarrow \mathcal{E}_{k+1}\coloneqq  V,
\end{align*}
where each $\mathcal{E}_i$ is a vector bundle, and the projection $\pi: \Hquot_\Bd(C, \Br, V) \times C \rightarrow \Hquot_\Bd(C, \Br, V)$. As described in \cite{ckm} and \cite[Theorem 3.2]{mr}, there is a $f$-relative perfect obstruction theory $\BE\to \BL_{f}$, where $\BL_f$ denotes the $f$-relative cotangent complex, for the Hyperquot scheme (\ref{eq:hquot_scheme}) with
\begin{align}\label{monavari_ricolfi_pot}
    \bb{E} \coloneqq  \mathrm{Cone} \left[ \bigoplus_{i=1}^k \Hom_\pi(\mathcal{E}_i, \mathcal{E}_i) \xrightarrow{\gamma} \bigoplus_{i=1}^k \Hom_\pi(\mathcal{E}_i, \mathcal{E}_{i+1}) \right]^\vee
\end{align}
where $\Hom_\pi$ is the composition of $\rR\pi_\ast$ and $\Hom$ in the derived category. The morphism $\gamma$ is described as follows. By post-composition, the subsheaves define morphisms of sheaves $\alpha_i: \Hom(\mathcal{E}_i, \mathcal{E}_i) \rightarrow \Hom(\mathcal{E}_i, \mathcal{E}_{i+1})$ for every $i \in [k]$. Analogously, there are morphisms $\beta_i: \Hom(\mathcal{E}_i, \mathcal{E}_i) \rightarrow \Hom(\mathcal{E}_{i-1}, \mathcal{E}_{i})$ given by pre-composition. 
The arrow is given by
\begin{align*}
    \Hom_\pi(\mathcal{E}_i, \mathcal{E}_i) \oplus \Hom_\pi(\mathcal{E}_{i+1}, \mathcal{E}_{i+1}) \xrightarrow{\alpha_i - \beta_{i+1}} \Hom_\pi(\mathcal{E}_i, \mathcal{E}_{i+1})
\end{align*}
for every $i$.

\begin{remark}
    Note that \cite[Theorem 3.2]{mr} is stated in the absolute setting, that is, when $B=\{\mathrm{pt}\}$, but the exact same argument gives a relative perfect obstruction theory when $B$ is an arbitrary scheme. In the relative setting, the perfect obstruction theory for the Quot scheme was constructed by \cite{gillam_deformation, kuhn_atiyah}. {One can also construct a perfect obstruction theory for $\Hquot_\Bd(C,\Br,V)$ by viewing it as an iterated Quot scheme and applying \cite[Remark~4.6]{manolache_virtual_pull}.}
\end{remark}
\subsubsection*{Sketch of construction (see \cite[Theorem 3.2]{mr} for details)}
For fixed $\Br = (r_1, \dots, r_k)$, $\Bd= (d_1, \dots, d_k)$, and a vector bundle $V$ on $B\times C$, we can consider the moduli Artin stack $\Bun_i$ of vector bundles of rank $r_i$ and relative degree $\deg V-d_i$ on $C$. Since these stacks are smooth, the projection
\begin{align*}
    \pr_B: B \times \prod_{i=1}^{k}\Bun_i\rightarrow B
\end{align*}
is smooth and the identity $\mathrm{Id}: \BL_{\pr_B}\rightarrow\BL_{\pr_B}$ of the relative cotangent complex of $\pr_B$ defines a relative perfect obstruction theory. From this, thanks to \cite[Remark 4.6]{manolache_virtual_pull}, we just have to describe a relative obstruction theory for the morphism
\begin{align*}
    \Hquot_\Bd(C,\Br,V) \rightarrow B \times \uBun \quad : \quad [E_1 \hookrightarrow \dots \hookrightarrow E_k \hookrightarrow V\vert_{\lbrace b\rbrace \times C}] \mapsto (b, [E_1], \dots, [E_k])
\end{align*}
where $\uBun \coloneqq  \prod_{i=1}^k\Bun_i$. This is easily achieved by factoring this map as a closed embedding followed by a smooth morphism, as we now recall. Let $\fE_i$ be the universal bundle on $\Bun_i \times C$ and consider the sheaf $F \coloneqq  \bigoplus_{i=1}^k F_i\coloneqq  \bigoplus_{i=1}^k\Hom(\fE_{i+1},\fE_i \otimes \omega_\pi)$ on $C \times B \times \uBun$, where
\begin{align*}
    \pi : C \times B \times \uBun \rightarrow B \times \uBun
\end{align*}
is the projection along the curve and $\omega_\pi$ is its relative cotangent bundle. These can be resolved as
\begin{align}\label{eq:resolution}
    0 \rightarrow K_i \rightarrow M_i \rightarrow F_i \rightarrow 0
\end{align}
 so that ${K_{i}}$ and ${M_{i}}$ are locally free, satisfy $\pi_\ast K_i = \pi_\ast M_i =0$ and both $\rR^1\pi_\ast K_i$ and $\rR^1\pi_\ast M_i$ are locally free sheaves (c.f. \cite[Lemma 2.8]{sinha} or \cite[Lemma 1.10]{scattareggia2018perfect}). Once we set $K\coloneqq  \bigoplus_{i=1}^k K_i$ and $M\coloneqq  \bigoplus_{i=1}^k M_i$, this implies that the structure morphism
 \begin{align*}
     p : \SpecSym(\rR^1\pi_\ast M) \rightarrow B \times \uBun
 \end{align*}
 defines the total space of a vector bundle on $B \times \uBun$, the inclusion $K \rightarrow M$ defines a section $\kappa$ of the locally free sheaf $(p^\ast \rR^1\pi_\ast K)^\vee$ and the corresponding zero locus
\begin{align}\label{eq:factorisation}
    Z(\kappa) \xhookrightarrow{i} \SpecSym(\rR^1\pi_\ast M) \xrightarrow{p} B \times \uBun  
\end{align}
Note that $Z(\kappa)$ parametrizes chains of maps $E_1 \rightarrow \dots \rightarrow E_k \rightarrow \CO_C^{\oplus n}$, and the Hyperquot scheme sits inside $Z(\kappa)$ as the open locus of injective maps. The factorization (\ref{eq:factorisation}) induces the relative perfect obstruction theory
\[\begin{tikzcd}
        \BE \arrow[d, "\phi"] & = & {[p^\ast \rR^1\pi_\ast K \vert_{Z(\kappa)}} \arrow[r, "\mathrm{d}\kappa^\vee"] \arrow[d, "\iota_\kappa"] & {\Omega_{\SpecSym(\rR^1\pi_\ast M)} \vert_{Z(\kappa)}{]}} \arrow[d, equal]\\
        \tau^{[-1,0]}\BL_{p\circ i} & = & {[}\CI/\CI^2 \arrow[r, "\mathrm{d}"] & {\Omega_{\SpecSym(\rR^1\pi_\ast M)} \vert_{Z(\kappa)}{],}}
    \end{tikzcd}\]
    where $\CI$ is the ideal sheaf of the closed embedding $i$.
\subsection{Virtual fundamental class} If $B = \pt$ is a point, via the construction of \cite{behrend_fantechi}, the perfect obstruction theory recalled above induces a \textit{virtual fundamental class} on the moduli space, namely a Chow class 
\begin{align*}
    [\Hquot_\Bd(C, \Br, V)]^\vir \in \mathrm{A}_{\mathrm{vdim}}(\Hquot_\Bd(C, \Br, V))
\end{align*}
of \textit{virtual dimension} equal to the expected dimension from deformation theory
\begin{align}\label{eq:virtual_dim}
    \mathrm{vdim} = \chi(\BE) = (1-g)\sum_{i=1}^k r_i\bigl(r_{i+1} - r_i\bigr)
+ \deg V(r_1 - n)
+ \sum_{i=1}^k d_i \bigl(r_{i+1} - r_{i-1}\bigr)
\end{align}
with $r_0\coloneqq 0$ and $r_{k+1}\coloneqq n$. We remark that $\sum_{i=1}^k r_i\bigl(r_{i+1} - r_i\bigr) = \dim(\Fl(r_1, \dots, r_k,n))$.

\begin{remark}
It is known that $\quot_d(C,r,V)$ is irreducible \cite{bertram-deskalopoulos-wentworth,popa_stable_maps} of the expected dimension for all sufficiently large $d$ relative to $g,r$ and $V$. For Hyperquot schemes, irreducibility has been studied in \cite{rasul:sabastian}. For example, when $k=2$, and $d_1 \gg d_2\gg0$ or $0\ll d_1\ll d_2$, the Hyperquot schemes are irreducible of the expected dimension. Consequently, $[\Hquot_\Bd(C, \Br, V)]^{\vir} = [\Hquot_\Bd(C, \Br, V)]$, and the virtual intersection numbers computed in Theorem~\ref{thm:intro_VI_formula} coincide with the classical intersection numbers.
\end{remark}

The presence of the relative perfect obstruction theory implies the deformation invariance of the virtual intersection numbers via the well-known result \cite[Proposition 3.9]{manolache_virtual_push}. In our case, it implies:
\begin{lemma}\label{lem:deformation_invariance}
Let $B$ be an irreducible projective scheme,
$V$ be a vector bundle on $B \times C$, and fix a point $p \in C$.
Let $f : \Hquot_\Bd(C,\Br,V) \rightarrow B$ be the relative Hyperquot scheme endowed with the $f$-relative perfect obstruction theory \eqref{monavari_ricolfi_pot}, and let $\CE_1, \dots, \CE_k$ be the corresponding universal subsheaves. For any point $b \in B$ denote by
\begin{align*}
    i_b : \Hquot_\Bd(C,\Br,V_b) \hookrightarrow \Hquot_\Bd(C,\Br,V)
\end{align*}
the inclusion of the fiber over $b$, where $V_b \coloneqq  V\vert_{\lbrace b\rbrace \times C}$. If $\alpha$ is any cohomology class on $\Hquot_\Bd(C,\Br,V)$
obtained from the Chern classes of the bundles $\CE_{i\vert p}^\vee$, the function
\begin{align*}
    B \longrightarrow \mathbb{Z}, \qquad
b \longmapsto \int_{[\Hquot_\Bd(C,\Br,V_b)]^{\vir}} i_b^*\alpha
\end{align*}
is constant on $B$.
\end{lemma}
From this we deduce the following useful fact, which shows that we may replace $V$
by a split vector bundle without changing the virtual intersection numbers considered
in this paper.
\begin{proposition}\label{prop:reduction_split_bundles}
Let $V$ be a vector bundle on $C$ of rank $n$, and fix
$\Br=(r_1,\dots,r_k)$ and $\Bd=(d_1,\dots,d_k)$.
Then there exists a split vector bundle $\tilde V \simeq \bigoplus_{i=1}^n M_i$ of the same degree and rank as $V$ such that
\[
\int_{[\Hquot_{\Bd}(C,\Br,V)]^{\vir}} \prod_{i,j} c_i(\CE_{j\vert p}^\vee)_{}^{m_{i,j}}
\;=\;
\int_{[\Hquot_{\Bd}(C,\Br,\tilde V)]^{\vir}} \prod_{i,j} c_i(\tilde\CE_{j\vert p}^\vee)^{m_{i,j}}
\]
for every $p \in C$ and $m_{i,j}\geq 0$, where
\begin{align*}
    0\hookrightarrow \CE_1 \hookrightarrow \dots \hookrightarrow \CE_k \hookrightarrow V \quad \text{and} \quad 0\hookrightarrow \tilde \CE_1 \hookrightarrow \dots \hookrightarrow \tilde \CE_k \hookrightarrow \tilde V
\end{align*}
are the universal subsheaves on $\Hquot_{\Bd}(C,\Br,V) \times C$ and $\Hquot_{\Bd}(C,\Br,\tilde V) \times C$ respectively. Moreover, we may choose the splitting such that \(0\leq \deg(M_i)-\deg(M_j)\leq 1\) for all \(i\leq j\).
\end{proposition}
\begin{proof}
First of all, note that there is a line bundle $L$ and a positive integer $N$ such that $V$ can be realized as a subsheaf of $H\coloneqq  L^{\oplus N}$. Consider the Quot scheme $B\coloneqq  \quot_{\deg H-\deg V}(C,n,H)$ which parametrizes subsheaves of $H$ that have rank $n$ and degree \(\deg V\), so that $b=[V\to H]\in B$. We may choose $L$ of sufficiently large degree, so that the Quot scheme $B$ is irreducible by \cite{popa_stable_maps} and $n\deg(L)>\deg(V)$. The universal subsheaf $\CV$ on $B \times C$ restricts to $\CV|_b\cong V$. On the other hand, we show that $B$ contains a point of the form $\tilde{b}\coloneqq  [\tilde{V}\hookrightarrow H]$ by starting from the subsheaf $L^{\oplus n} \hookrightarrow H$ and appropriately twisting the summands by $\CO_C(-p)$ in a balanced way, until the correct degree $\deg V$. Then the claim follows by Lemma \ref{lem:deformation_invariance}. 
\end{proof}

\subsection{Compatibility of virtual classes}
We now record how the virtual fundamental class behaves under two basic operations on the universal subsheaves:
\begin{itemize}
    \item twisting $\CE_i$ by $\CO_C(-p)$ for a point $p \in C$, and
    \item elementary modification of the ambient bundle $V$ at a point of $C$.
\end{itemize}
\subsubsection{Twisting by a point}
Fixed a point $p\in C$, there is an inclusion of Hyperquot schemes
\[
\iota: \Hquot_{\Bd}\coloneqq \Hquot_{\Bd}(C,\Br, V) \to \Hquot_{\Bd + \Boldeta_\ell }\coloneqq\Hquot_{\Bd + \Boldeta_\ell }(C,\Br, V),
\]
where $\Boldeta_\ell=(r_1,\dots,r_\ell,0,\dots,0)$ and the map is defined by the universal property, taking the universal flag of subsheaves $\mathcal{E}_1 \hookrightarrow \dots \hookrightarrow \mathcal{E}_k \hookrightarrow V$ on $\Hquot_{\Bd}\times C$ and constructing from it the different flag
\begin{align*}
    \mathcal{E}_1(-p) \hookrightarrow \dots \hookrightarrow \mathcal{E}_\ell(-p) \hookrightarrow \mathcal{E}_{\ell+1}\hookrightarrow \dots \hookrightarrow \mathcal{E}_k \hookrightarrow V.
\end{align*}
We have the following compatibility result among virtual classes of Hyperquot schemes:
\begin{proposition}\label{prop:compatibility_twisting}
    Let $\mathcal{E}_1 \hookrightarrow \dots \hookrightarrow \mathcal{E}_k \hookrightarrow V$ be the universal flag of subsheaves on $\Hquot_{\Bd+ \Boldeta_\ell}\times C$. Then
    \begin{align*}
        \iota_* [\Hquot_{\Bd}]^{\vir} = e(\CE_{\ell}^\vee\otimes \CE_{\ell+1})_{\vert p}\cap [\Hquot_{\Bd+\Boldeta_\ell}]^{\vir}.
    \end{align*}
    In particular, for any $Q=  
\prod_{j=1}^{k}\prod_{i=1}^{r_j} c_i(\CE_{j|p}^\vee)^{m_{i,j}}$, we have the identity
    	\begin{align*}
    \prod_{i=1}^{\ell}q_i^{r_i}  \sum_{\Bd\in \BZ^k}q_1^{d_1} \cdots q_k^{d_k}\int_{[\Hquot_{\Bd}]^{\vir}} Q =\sum_{\Bd\in \BZ^k}q_1^{d_1} \cdots q_k^{d_k} \int_{[\Hquot_{\mathbf{d}+\delta_l}]^{\vir}} Q\cdot e(\CE_{\ell\mid p}^\vee\otimes \CE_{\ell+1\mid p}).
		\end{align*}
\end{proposition}
\begin{proof}
First of all note that $\iota$ can be seen as the embedding of the zero locus of the section $s$ of the bundle $(\mathcal{E}_l^\vee \otimes \mathcal{E}_{l+1})_{\vert p}$ on $\Hquot_{\Bd+\Boldeta_\ell}$, obtained by restricting to $p$ the universal map $\mathcal{E}_i \rightarrow \mathcal{E}_{i+1}$. Thus we find the fibered diagram
\[
\begin{tikzcd}
    \Hquot_{\Bd} \arrow[r, "\iota"] \arrow[d, "\iota"] \arrow[rd, phantom, "\square"] & \Hquot_{\Bd+\Boldeta_\ell} \arrow[d, "s"]\\
    \Hquot_{\Bd+\Boldeta_\ell} \arrow[r, "0"] & \mathrm{Tot}((\mathcal{E}_l^\vee \otimes \mathcal{E}_{l+1})_{\vert p}).
\end{tikzcd}
\]
The resolutions (\ref{eq:resolution}) used to construct the relative perfect obstruction theory on $g_{\eta}: \Hquot_{\Bd + \Boldeta_\ell} \rightarrow \uBun$ can be obtained from the resolutions (\ref{eq:resolution}) used for $g: \Hquot_\Bd \rightarrow \Bun$ by twisting the summand corresponding to $i=\ell$ by $\CO_C(-p)$. After some computations this implies the commutativity of the diagram of perfect obstruction theories
\[\begin{tikzcd}
        \rL\iota^\ast \BE_{\Bd+\Boldeta_\ell}\arrow[r, "\alpha"]\arrow[d, "\iota^\ast \phi_{\Bd + \Boldeta_\ell}"] & \BE_\Bd \arrow[d, "\phi_\Bd"]\\
        \rL \iota^\ast \BL_{g_{\eta}} \arrow[r, ""] & \BL_{g},
\end{tikzcd}\]
where the top horizontal arrow $\alpha$ is induced by the morphism among the triangles (\ref{monavari_ricolfi_pot}) defined by the inclusion $\CE_\ell(- p) \hookrightarrow \CE_\ell$ 
\[\begin{tikzcd}
    \bigoplus_i \Hom_\pi(\CE_i, \CE_i) \arrow[r, "\iota^\ast \gamma_\delta"]& \bigoplus_{i} \Hom_\pi(\CE_i(-\delta^i_\ell \, p), \CE_{i+1}) \arrow[r] & \rL \iota^\ast \BE_{\Bd +\Boldeta_\ell}^\vee \arrow[r, "+1"] & \,\\
    \bigoplus_i \Hom_\pi(\CE_i, \CE_i) \arrow[r, "\gamma"]\arrow[u, "\sim"] & \bigoplus_{i} \Hom_\pi(\CE_i, \CE_{i+1}) \arrow[r] \arrow[u] & \BE_{\Bd}^\vee \arrow[u, "\alpha^\vee"] \arrow[r, "+1"] & \,
\end{tikzcd}\]
where $\CE_i$ is the $i$-th universal subsheaf on $\Hquot_{\Bd}$. By taking the cones of the vertical arrows in the diagram above we find $\mathrm{Cone}(\alpha^\vee) \simeq \Hom( \mathcal{E}_{\ell\vert p},\mathcal{E}_{\ell+1\vert p})$.
Dualizing we find the exact triangle
\begin{align*}
    \rL\iota^\ast \mathbb{E}_{\Bd + \Boldeta_\ell} \rightarrow \mathbb{E}_{\Bd} \rightarrow \Hom(\mathcal{E}_{\ell\vert p},\mathcal{E}_{\ell+1 \vert p})^\vee[1] \rightarrow +1
\end{align*}
which, after using \cite[\href{https://stacks.math.columbia.edu/tag/08SJ}{Tag 08SJ}]{StacksProj} to show $\Hom(\mathcal{E}_{\ell\vert p},\mathcal{E}_{\ell+1 \vert p})^\vee[1] \simeq \iota^\ast \mathbb{L}_{0}$, implies via \cite[Proposition 5.10]{behrend_fantechi} that $0^! [\Hquot_{\Bd+\Boldeta_\ell}]^{\vir} = [\Hquot_{\Bd}]^{\vir}$, where $0^!$ is the Gysin pullback along the zero section of $(\mathcal{E}_l^\vee \otimes \mathcal{E}_{l+1})_{\vert p}$. The claim then follows from \cite[Theorem 6.2]{fulton:IT}:
\begin{align*}
    \iota_\ast [\Hquot_{\Bd}]^{\vir} &= \iota_\ast 0^! [\Hquot_{\Bd+\Boldeta_\ell}]^{\vir} = 0^\ast s_\ast [\Hquot_{\Bd+\Boldeta_\ell}]^{\vir} = e(\CE_{\ell\vert p}^\vee\otimes \CE_{\ell+1\vert p})\cap [\Hquot_{\Bd+\Boldeta_\ell}]^{\vir},
\end{align*}
since $0^\ast s_\ast$ corresponds to capping with the Euler class of the bundle. 
\end{proof}
\subsubsection{Elementary Modifications}\label{sec:elementary_modification}
Let $V$ be a vector bundle on $C$ and fix a point $p \in C$. Consider a surjection of linear spaces $V_{\vert p} \twoheadrightarrow \BC$, inducing a short exact sequence
\begin{align}\label{modification_ses}
    0 \rightarrow \tilde{V} \xrightarrow{\alpha} V \rightarrow \CO_p \rightarrow 0,
\end{align}
where $\CO_p$ denotes the skyscraper sheaf at $p$. The mapping
\begin{align*}
    [S_1 \xrightarrow{f_1} \dots \xhookrightarrow{f_{k-1}} S_k \xhookrightarrow{f_k} \tilde{V}] \mapsto [S_1 \xhookrightarrow{f_1} \dots \xhookrightarrow{f_{k-1}} S_k \xhookrightarrow{\alpha \circ f_k} V]
\end{align*}
induces a morphism of Hyperquot schemes
\begin{align*}
    j : \Hquot_{\Bd-\bf1}(C, \Br, \tilde{V}) \rightarrow \Hquot_{\Bd}(C, \Br, V)
\end{align*}
where $\bf{1}$ denotes the vector $(1, \dots, 1) \in \BZ^k$.
This is a closed embedding, which can be realized as the zero locus of the section of $\CE_{k\vert p}^\vee$ given by the tautological morphism
\begin{align*}
    \CE_k \rightarrow \pi_C^\ast V \rightarrow \ \pi_C^\ast \CO_p
\end{align*}
restricted to $\Hquot_{\Bd}(C, \Br, V) \times \lbrace p\rbrace$, where $\pi_C$ denotes the projection to the curve $C$.
\begin{proposition}\label{prop:elementary_mod}
    The virtual fundamental classes of the two Hyperquot schemes satisfy the following compatibility condition:
    \begin{align*}
        j_\ast [\Hquot_{\Bd-\bf1}(C, \Br, \tilde{V})]^\vir = e(\CE^\vee_{k\vert p}) \cap [\Hquot_{\Bd}(C, \Br, V)]^\vir.
    \end{align*}
    In particular, for any $Q=  
\prod_{j=1}^{k}\prod_{i=1}^{r_j} c_i(\CE_{j|p}^\vee)^{m_{i,j}}$, we have the identity
    	\begin{align*}
     \sum_{\Bd\in \BZ^k}q_1^{d_1+1} \cdots q_k^{d_k+1}\int_{[\Hquot_{\Bd}(C, \Br, \tilde{V})]^{\vir}} Q =\sum_{\Bd\in \BZ^k}q_1^{d_1} \cdots q_k^{d_k}\int_{[\Hquot_{\Bd}(C, \Br, V)]^{\vir}} Q\cdot c_{r_k}(\CE_{k|p}^\vee).
		\end{align*}
\end{proposition}
\begin{proof}
    Let $\BE$ be the obstruction complex in the perfect obstruction theory of $\Hquot_{\Bd}(C, \Br, V)$ and denote by $\tilde{\BE}$ the one for $\Hquot_{\Bd-\bf1}(C, \Br, \tilde{V})$.
    By the same argument used in the proof of Proposition \ref{prop:compatibility_twisting} we see that the cone
    \begin{align*}
        \frak C\coloneqq  \mathrm{Cone}(\tilde{\BE}^\vee \rightarrow \rL j^\ast \BE^\vee)
    \end{align*}
    fits into the exact triangle
    \begin{align*}
        \Hom_\pi(\CE_k, \tilde{V}) \rightarrow \Hom_\pi(\CE_k, V) \rightarrow \frak C \rightarrow +1,
    \end{align*}
    hence by (\ref{modification_ses})
    \begin{align*}
        \frak C &\simeq \Hom_\pi(\CE_k, \pi_C^\ast \CO_p) \simeq \rR \pi_\ast (\CE_k^\vee \otimes \pi_C^\ast \CO_p) \simeq \rR \pi_\ast \rR j_\ast (j^\ast \CE_{k{\vert p}}^\vee) \simeq j^\ast \CE_{k{\vert p}}^\vee.
    \end{align*}
    After dualization, we find the exact triangle
    \begin{align*}
        \rL j^\ast \BE \rightarrow \tilde{\BE} \rightarrow j^\ast \CE_{k{\vert p}}[1] \rightarrow +1,
    \end{align*}
    which allows us to apply \cite[Proposition 5.10]{behrend_fantechi} to obtain $$0^! [\Hquot_{\Bd}(C, \Br, V)]^\vir = [\Hquot_{\Bd-\bf1}(C, \Br, \tilde{V})]^\vir,$$ where $0$ is the zero section of the vector bundle $(\CE_{k\vert p}^\vee)$ on $\Hquot_{\Bd}(C, \Br, V)$. The rest of the proof follows as in Proposition \ref{prop:compatibility_twisting}.
\end{proof}

\subsection{A vanishing criterion}
Recall that the Hyperquot scheme $\Hquot_\Bd(C,\Br,V)$ is constructed iteratively using relative Quot schemes
\[
\quot_{d_1-d_2}(C,r_1,\CE_2)\xrightarrow{\phi_1}\cdots\xrightarrow{\phi_{k-2}}
\quot_{d_{k-1}-d_k}(C,r_{k-1},\CE_k)\xrightarrow{\phi_{k-1}}
\quot_{d_k}(C,r_k,V)\xrightarrow{\phi_k} \{\pt\},
\]
with relative virtual dimensions
\begin{align*}
    \vdim \phi_j &= \chi(E_j,E_{j+1})-\chi(E_j,E_j)\\ 
    &= \bigr((1-g)\,r_j - e\bigl)(r_{j+1} - r_j) + r_{j+1} d_j - r_j d_{j+1}.
\end{align*}
We note the following vanishing result, not immediate from the formula in Theorem~\ref{thm:intro_VI_formula}, which follows from the properties of virtual pullbacks.
\begin{remark}
When the virtual dimension of an intermediate Quot scheme $\quot_{d_j-d_{j+1}}(C,\\r_{j},\CE_{j+1})$ is negative, i.e., when $\sum_{i=j}^k\vdim \phi_i<0$, the virtual class $[\Hquot_{\Bd}(C,\Br,V)]^{\vir}$ is zero and all virtual integrals on the Hyperquot scheme vanish:
\[
\int_{[\Hquot_{\Bd}(C,\Br,V)]^{\vir}} \prod_{j=1}^{k}\prod_{i=1}^{r_j} c_i(\CE_{j|p}^\vee)^{m_{i,j}} \;=\; 0.
\]
\end{remark}

\begin{remark}\label{rmk:dim_zero_polynomial}
Notice that similar vanishings need not hold when the virtual dimension of a single morphism $\phi_j$ is negative.
For instance, let $C$ be a curve of genus $13$, let $V$ be a rank-$3$ bundle of degree $0$, and take $\Br=(1,2)$. Then Corollary~\ref{cor:intro_r=(1,n-1)} implies
\[
\sum_{d_1,d_2\ge0} q_1^{d_1} q_2^{d_2}
\int_{[\Hquot_{\Bd}]^{\vir}} 1
\;=\;
6^{13}\!\left( q_1^{10} q_2^8 + 20\, q_1^9 q_2^9 + q_1^8 q_2^{10}\right).
\]
When $d_1=d_2=9$, the relative virtual dimensions are $\vdim\phi_2=3$ and $\vdim\phi_{1}=-3$, yet the virtual integral of $1$ equals $20 \cdot 6^{13} \neq 0$.
\end{remark}

\section{Torus action and fixed loci}\label{sec:equiv_localization}
By the deformation invariance property of virtual integrals proved in Proposition~\ref{prop:reduction_split_bundles}, 
the calculation reduces to the case of split vector bundles, so we can assume that
\begin{equation}\label{eq:V_split}
    V=M_1\oplus M_2\oplus\cdots\oplus M_{n}
\end{equation}
where each  $M_i$ is a line bundle of degree zero over $C$. 
\subsection{Equivariant formulation}
In this section we state the equivariant version of our main Theorem \ref{thm:intro_VI_formula}. Let the torus $T = (\mathbb{C}^*)^n$ act on $V$ by scaling the fibers of each line bundle summand $M_j$ by $t_j^{-1}$. This induces a $T$ action on the Hyperquot scheme $\Hquot_{\mathbf{d}}\coloneqq \Hquot_{\mathbf{d}}(C, \mathbf{r}, V)$, by sending a point $[E_1 \xhookrightarrow{} E_2 \xhookrightarrow{} \cdots \xhookrightarrow{} E_k \xhookrightarrow{} V]$ to the composition 
\[
E_1 \xhookrightarrow{} E_2 \xhookrightarrow{} \cdots \xhookrightarrow{} E_k \xhookrightarrow{} V \xrightarrow{t} V
\]
for every element $t \in T$. Let $\varepsilon_1,\varepsilon_2,\dots,\varepsilon_n$ be the weights of the action of $T$ on $V^\vee=M_1^\vee\oplus M_2^\vee\oplus\cdots\oplus M_{n}^\vee$ or, more precisely, set \[\varepsilon_i\coloneqq c_1^T(M_{i\mid \pt}^{\vee})\in H_T^*(\pt),\]
for any point $\pt\in C$.

We now define an equivariant analog of the system of equations \eqref{eq:Bethe_system_intro}, which is a key ingredient of the formula that we want to prove. Consider the tuples of variables $\Bz_1,
\Bz_2,\dots,\Bz_k$ as before with $\Bz_0=\{\}$ and 
$\Bz_{k+1}=(\varepsilon_1, \varepsilon_2, \dots, \varepsilon_n)$, as illustrated below.
\[
		\begin{tikzpicture}
		\draw (-0.5,0) -- (7,0) ;
		\draw (0,0.2) node {$\textbf{z}_0$};
		\draw (1,0.2) node {$\textbf{z}_1$};
		\draw (2,0.2) node {$\textbf{z}_2$};
		\draw (3,0.2) node {$\cdots$};
		\draw (4,0.2) node {$\cdots$};
		\draw (5,0.2) node {$\textbf{z}_k$};
		\draw (6,0.2) node {$\textbf{z}_{k+1}$};
		
		\draw (3,-0.4) node {$\cdots$};
		\draw (4,-0.4) node {$\cdots$};
		
		\draw (3,-2) node {$\ddots$};
		\draw (4,-2.5) node {$\ddots$};
		
		\draw (1,-0.4) node {\small $z_{1,1}$};
		\draw (1,-0.8) node {\small $\vdots$};
		\draw (1,-1.2) node {\small $z_{r_1,1}$};

		\draw (2,-0.4) node {\small $z_{1,2}$};
		\draw (2,-0.8) node {\small $\vdots$};
		\draw (2,-1.2) node {\small $\vdots$};
		\draw (2,-1.7) node {\small $z_{r_2,2}$};
		
		\draw (5,-0.4) node {\small $z_{1,k}$};
		\draw (5,-0.8) node {\small $\vdots$};
		\draw (5,-1.2) node {\small $\vdots$};
		\draw (5,-1.7) node {\small $\vdots$};
		\draw (5,-2.1) node {\small $\vdots$};
		\draw (5,-3) node {\small $z_{r_k,k}$};
		
		\draw (6,-0.4) node {\small $\varepsilon_1$};
		\draw (6,-0.8) node {\small $\vdots$};
		\draw (6,-1.2) node {\small $\vdots$};
		\draw (6,-1.7) node {\small $\vdots$};
		\draw (6,-2.1) node {\small $\vdots$};
		\draw (6,-2.5) node {\small $\vdots$};
		\draw (6,-2.9) node {\small $\vdots$};
		\draw (6,-3.5) node {\small $\varepsilon_n$};
	\end{tikzpicture}
	\]
Let $q_1,q_2 \dots, q_k$ be formal variables and define the following polynomials for each $1\le j\le k$:
      \begin{align}\label{eq:Bethe_equivariant}
		P_j^{T}(X)\coloneqq \prod_{\alpha\in \Bz_{j+1}}(X-\alpha)+(-1)^{r_{j}-r_{j-1}}q_j\prod_{\alpha\in \Bz_{j-1}}(X-\alpha).
        	\end{align}

\begin{definition}
 For fixed values of \((\varepsilon_1,\dots,\varepsilon_n;q_1,\dots,q_k)\in \BC^n\times\parens*{\BC^*}^k\), we define the system of polynomial equations in the variables $z_{s,j}$
 \begin{align}\label{eq:Bethe_system_equiv}
 P_j^T(z_{s,j})=0\quad \text{for each } 1\le j\le k, \text{ and } 1\le s\le r_j.
 \end{align}
 We say that a solution $(\Bz_1,\Bz_2\dots,\Bz_k)=(\boldsymbol{\zeta}_1,\boldsymbol{\zeta}_2\dots,\boldsymbol{\zeta}_k)$ to the above system is \emph{non-degenerate} if each tuple $\boldsymbol{\zeta}_j = (\zeta_{1,j},\zeta_{2,j},\dots,\zeta_{r_j,j})$ has no repetition. 
\end{definition}
The $T$-equivariant version of our main theorem is
\begin{theorem}\label{thm:VI_Hyper_quot_equivariant}
Let $V$ be as in \eqref{eq:V_split}. Then for natural numbers $m_{s,j}$ and values of \((\varepsilon_1,\dots,\varepsilon_n;\\q_1,\dots,q_k)\) in a Zariski open subset of \(\BC^n\times\parens*{\BC^*}^k\) intersecting \(0\times\parens*{\BC^*}^k\), we have 
\begin{align*}\label{eq:Intersection_Hquot(C,V)_equivariant}
    \sum_{\Bd\in \BN^k} q_1^{d_1}q_2^{d_2}\cdots q_k^{d_k}
\int_{[\Hquot_{\Bd}]^{\vir}}  
\prod_{j=1}^{k}\prod_{s=1}^{r_j} c_s^{T}(\CE_{j|p}^\vee)^{m_{s,j}}
    =\sum_{\boldsymbol{\zeta}_1,\dots,\boldsymbol{\zeta}_k}  \prod_{j=1}^{k}\frac{1}{r_j!}\prod_{s=1}^{r_j}e_s(\boldsymbol{\zeta}_j)^{m_{s,j}}\cdot J^{g-1}
	\end{align*}
    where the sum is taken over all non-degenerate solutions $(\boldsymbol{\zeta}_1,\dots,\boldsymbol{\zeta}_k)$ of the equations in \eqref{eq:Bethe_system_equiv} and
    $e_s(\boldsymbol{\zeta}_j)$ denotes the $s$-th elementary symmetric polynomial. 
     The factor $J$ is given by
     \[ J\coloneqq\parens*{\prod_{\ell=1}^{k} \frac{1}{\Delta(\boldsymbol{\zeta}_\ell)}}\cdot
\det\left(\frac{\partial P^T_j(z_{s,j})}{\partial z_{s',j'}}\right)\Bigg|_{(\Bz_1,\Bz_2,\dots,\Bz_k) = (\boldsymbol{\zeta}_1,\boldsymbol{\zeta}_2,\dots,\boldsymbol{\zeta}_k)}
 ,\]
where $ \Delta(X_1,\dots,X_r)\coloneqq\prod_{a\ne b}(X_a-X_b)$ and the second factor is the Jacobian of \eqref{eq:Bethe_equivariant}. 
\end{theorem}
We will first prove a version of this theorem where the solutions are viewed as formal power series in the \(q\)-variables (see Theorem~\ref{thm:VI_Hyper_quot_equivariant_formal}). The version as stated above will be proved in Proposition~\ref{prop:genus0eval} when $C=\BP^1$ and in Section~\ref{Sec:proof_of_thm_1} for arbitrary genus $g$.

\subsection{Fixed loci}\label{subs:fixed_loci}
 We now describe the fixed loci for the torus action introduced in the previous section. A point $[E_1 \xhookrightarrow{} E_2 \xhookrightarrow{} \cdots \xhookrightarrow{} E_k \xhookrightarrow{} V]\in \Hquot_{\Bd}$ is fixed under the torus action if and only if each of the subsheaves splits as a direct sum of line bundles $$E_{j}=K_{1,j}\oplus K_{2,j}\oplus \cdots\oplus K_{r_j,j}\quad \text{for}\ 1\le j\le k;$$  and the inclusion $E_{j}\xhookrightarrow{} E_{j+1}$ respects the splitting. More explicitly, this means that there exists a permutation $\sigma\in S_n$ such that for all $1\le s\le r_j$,
\begin{align}\label{eq:line_bunlde_injection}
    K_{s,j}\xhookrightarrow{} K_{s,j+1}\quad  \forall\ 1\le j\le k-1
\end{align}
and each summand $K_{s,k}$ of $E_k$ maps into $M_{\sigma(s)} \subset V$.

To each torus-fixed point $[E_{\mathbf{r}} \xhookrightarrow{} V] \in \Hquot_{\mathbf{d}}$, we associate two pieces of data:

\begin{itemize}
    \item The \emph{splitting degree}, a tuple recording the degrees of the line bundles $K_{s,j}^\vee$:
    \[
    D = (d_{s,j}) \quad \text{for } 1 \le j \le k,\ 1 \le s \le r_j.
    \]
    Note that $d_j=\sum_{s=1}^{r_j}d_{s,j}$ for all $1\le j\le k$. The injection in \eqref{eq:line_bunlde_injection} implies $d_{s,j}\geq d_{s,j+1}$. 
    
    \item A permutation 
    \begin{equation}\label{eq:permutation_group_AtiyahBott}
        \sigma \in \mathfrak{S}_\Br \coloneqq  S_n / \left(S_{r_1} \times S_{r_2 - r_1} \times \cdots \times S_{r_k - r_{k-1}} \times S_{n - r_k} \right)
    \end{equation}
    that records the choice of summands in $V$ used to split the vector bundle $E_k$. The quotient by the subgroup $S_{r_{j+1} - r_j}$ accounts for the fact that reordering the line bundle summands in $E_{j+1}/E_j$ yields the same splitting data up to permutation of degrees (here, we set $r_0 = 0$ and $r_{k+1}=n$).
\end{itemize}

For every class $\sigma \in \mathfrak{S}_\Br$, fix once and for all a representative in $S_n$, which by an abuse of notation we will still denote by 
\begin{equation}\label{eq:Sn_representative}
    \sigma \in S_n.
\end{equation} 
This allows us to write $\sigma(s)$ for $1\leq s\leq n$.
For a pair $(\sigma,D)$, the corresponding fixed locus $\fix_{\sigma,D}$ is the product of Hilbert schemes of points of the curve
\begin{equation}\label{eq:product_hilb_sch}
    \fix_{\sigma,D} = \prod_{j=1}^k \prod_{s=1}^{r_j} C^{[d_{s,j} - d_{s,j+1}]}  \subset \Hquot_{\mathbf{d}}(C, \mathbf{r}, V)
\end{equation}

which can be described graphically by the diagram:
$$
\adjustbox{scale=0.85,center}{\begin{tikzcd}[cramped,column sep=large,row sep=tiny]
	C^{[d_{1,1}-d_{1,2}]} & {C^{[d_{1,2}-d_{1,3} ]}} & \cdots & C^{[d_{1,k-1}-d_{1,k}]} & {C^{[d_{1,k} ]}} \\
	\begin{array}{c} \times\\ \vdots\\ \times \end{array} & \begin{array}{c} \times\\ \vdots\\ \times \end{array} & \begin{array}{c} \cdots\\ \vdots\\ \cdots \end{array} & \begin{array}{c} \times\\ \vdots\\ \times \end{array} & \begin{array}{c} \times\\ \vdots\\ \times \end{array} \\
	C^{[d_{r_{1},1}-d_{r_{1},2}]} & {C^{[d_{r_1,2}-d_{r_1,3} ]}} & \cdots & C^{[d_{r_1,k-1}-d_{1,k}]} & {C^{[d_{r_1,k} ]}}\\
	& \begin{array}{c} \times\\ \vdots\\ \times \end{array} & \begin{array}{c} \cdots\\ \vdots\\ \cdots \end{array} & \begin{array}{c} \times\\ \vdots\\ \times \end{array} & \begin{array}{c} \times\\ \vdots\\ \times \end{array} \\
	& {C^{[d_{r_2,2}-d_{r_2,3} ]}} & \cdots & C^{[d_{r_2,k-1}-d_{r_2,k}]} & {C^{[d_{r_2,k} ]}} \\
	&& \ddots & \begin{array}{c} \times\\ \vdots\\ \times \end{array} & \begin{array}{c} \times\\ \vdots\\ \times \end{array} \\
	&&& C^{[d_{r_{k-1},k-1}-d_{r_{k-1},k}]} & C^{[d_{r_{k-1}{, k}}]} 
    \\
	&&&& \begin{array}{c} \times\\ \vdots\\ \times \end{array} \\
	&&&& C^{[d_{r_{k}{, k}}]} 
	\arrow[from=1-1, to=1-2, phantom, "\times"]
    \arrow[from=1-2, to=1-3, phantom, "\times"]
    \arrow[from=1-3, to=1-4, phantom, "\times"]
    \arrow[from=1-4, to=1-5, phantom, "\times"]
    \arrow[from=3-1, to=3-2, phantom, "\times"]
    \arrow[from=3-2, to=3-3, phantom, "\times"]
    \arrow[from=3-3, to=3-4, phantom, "\times"]
    \arrow[from=3-4, to=3-5, phantom, "\times"]
    \arrow[from=5-2, to=5-3, phantom, "\times"]
    \arrow[from=5-3, to=5-4, phantom, "\times"]
    \arrow[from=5-4, to=5-5, phantom, "\times"]
    \arrow[from=7-4, to=7-5, phantom, "\times"];
\end{tikzcd}}
$$
here, we use the convention that \(d_{s,k+1}=0\) for all \(s\).
For the $s$-th row, let $j_s$ be the smallest \(j\) such that $E_{j}$ has rank $r_j$ at least $s$.  The scheme 
$$ C^{[d_{s,{j_s}}-d_{s,{j_s+1}}]}\times \cdots \times C^{[d_{s,{k-1}}-d_{s,{k}}]}\times C^{[d_{s,{k}}]}$$
parametrizes sequences of injections of invertible sheaves $$K_{s,j_s}\xhookrightarrow{}K_{s,j_s+1}\xhookrightarrow{}\cdots\xhookrightarrow{}K_{s,k-1}\xhookrightarrow{}K_{s,k}\xhookrightarrow{} K_{s,k+1}\coloneqq M_{\sigma(s)}.$$ Note that the permutation $\sigma$ doesn't play a role in describing the geometry of the fixed locus $X_{\sigma,D}$, but it determines how $X_{\sigma,D}$ embeds in the $\Hquot_{\mathbf{d}}(C,\Br, V)$. It is crucial to keep track of $\sigma$, for example, in the calculation of virtual normal bundles in the next section.

\subsection{Virtual normal bundle}\label{subs:vir_norm_bundle}
Recall the universal subsheaves 
\begin{equation}\label{eq:universal_sequence_virtual_normal}
    \CE_{1}\xhookrightarrow{}\ \cdots\xhookrightarrow{} \ \CE_j\ \xhookrightarrow{}\cdots\xhookrightarrow{} \ \CE_{k-1}\xhookrightarrow{} \ \CE_{k}\ \xhookrightarrow{}\ \CE_{k+1}\coloneqq \pi_C^*V
\end{equation}
on $\Hquot_{\Bd}\times C$, and let $\pi$ and $\pi_C$ denote the projections to $\Hquot_{\Bd}$ and $C$, respectively. Consider the virtual tangent bundle of $\Hquot_\Bd$, namely the $K$-theory class of the dual of the obstruction complex in \eqref{monavari_ricolfi_pot}:
\begin{equation}\label{eq:Tangent_Bundle_QuotFlag}
	\Tan=\bigg[\bigoplus_{j=1}^{k}\Hom_{\pi} (\CE_{j},\CE_{{j+1}}) - \bigoplus_{j=1}^{k} \Hom_{\pi}(\CE_{j},\CE_{j}) \bigg],
\end{equation}
where $\Hom_{\pi}$ is the composition of the functors $R\pi_*$ and $\Hom$. The restriction of the universal subsheaves \eqref{eq:universal_sequence_virtual_normal} to the fixed loci $\fix_{\sigma,D}\times C$ is best seen in the following diagram:
\[
\adjustbox{scale=0.8,center}{\begin{tikzcd}[cramped,column sep=large,row sep=tiny]
    {\mathcal{K}_{1,1}} & {\mathcal{K}_{1,2}} & \cdots & {\mathcal{K}_{1,k-1}} & {\mathcal{K}_{1,k}} & {\mathcal{K}_{1,k+1}} \\
    \begin{array}{c} \oplus\\ \vdots\\ \oplus \end{array} & \begin{array}{c} \oplus\\ \vdots\\ \oplus \end{array} & \begin{array}{c} \cdots\\ \vdots\\ \cdots \end{array} & \begin{array}{c} \oplus\\ \vdots\\ \oplus \end{array} & \begin{array}{c} \oplus\\ \vdots\\ \oplus \end{array} & \begin{array}{c} \oplus\\ \vdots\\ \oplus \end{array} \\
    {\mathcal{K}_{r_1,1}} & {\mathcal{K}_{r_1,2}} & \cdots & {\mathcal{K}_{r_1,k-1}} & {\mathcal{K}_{r_1,k}} & {\mathcal{K}_{r_1,k+1}} \\
    & \begin{array}{c} \oplus\\ \vdots\\ \oplus \end{array} & \begin{array}{c} \cdots\\ \vdots\\ \cdots \end{array} & \begin{array}{c} \oplus\\ \vdots\\ \oplus \end{array} & \begin{array}{c} \oplus\\ \vdots\\ \oplus \end{array} & \begin{array}{c} \oplus\\ \vdots\\ \oplus \end{array} \\
    & {\mathcal{K}_{r_2,2}} & \cdots & {\mathcal{K}_{r_2,k-1}} & {\mathcal{K}_{r_2,k}} & {\mathcal{K}_{r_2,k+1}} \\
    && \ddots & \begin{array}{c} \oplus\\ \vdots\\ \oplus \end{array} & \begin{array}{c} \oplus\\ \vdots\\ \oplus \end{array} & \begin{array}{c} \oplus\\ \vdots\\ \oplus \end{array} \\
    &&& {\mathcal{K}_{r_{k-1},k-1}} & {\mathcal{K}_{r_{k-1},k}} & {\mathcal{K}_{r_{k-1},k+1}} \\
    &&&& \begin{array}{c} \oplus\\ \vdots\\ \oplus \end{array} & \begin{array}{c} \oplus\\ \vdots\\ \oplus \end{array} \\
    &&&& {\mathcal{K}_{r_k,k}} & {\mathcal{K}_{r_{k},k+1}}
    \arrow[from=1-1, to=1-2]
    \arrow[from=1-2, to=1-3]
    \arrow[from=1-3, to=1-4]
    \arrow[from=1-4, to=1-5]
    \arrow[from=1-5, to=1-6]
    \arrow[from=3-1, to=3-2]
    \arrow[from=3-2, to=3-3]
    \arrow[from=3-3, to=3-4]
    \arrow[from=3-4, to=3-5]
    \arrow[from=3-5, to=3-6]
    \arrow[from=5-2, to=5-3]
    \arrow[from=5-3, to=5-4]
    \arrow[from=5-4, to=5-5]
    \arrow[from=5-5, to=5-6]
    \arrow[from=7-4, to=7-5]
    \arrow[from=7-5, to=7-6]
    \arrow[from=9-5, to=9-6]
\end{tikzcd}}
\]
where $\CE_j$ restricts to $j$-th column for all $1\le j\le k$, i.e.,  
\begin{equation}\label{eq:splitting_Ei}
    \CE_{j}\big|_{\fix_{\sigma,D}\times C} = \CK_{1,j}\oplus \cdots \oplus \CK_{r_j,j},
\quad \text{and} \quad
\CK_{s,{k+1}}\coloneqq \pi_C^{*}M_{\sigma(s)}\ \forall s\in\{1,2,\dots,n \}.
\end{equation}
The virtual tangent bundle of $\Hquot_{\Bd}$ restricts to the following $K$-theory class on $\fix_{\sigma,D}$:
\begin{align*}
\Tan\big|_{\fix_{\sigma,D}}&=
\pi_*\sum_{j=1}^{k}\left(\sum_{s\le r_{j}, u\le r_{j+1} }^{} \CK^{\vee}_{s,j}\otimes \CK_{u,{j+1}}
-\sum_{s,u\le r_j}^{} \CK^{\vee}_{s,j}\otimes \CK_{u,j}\right),
\end{align*}
where $\pi_*$ denotes the push forward in $K$-theory for the projection $\pi:\fix_{\sigma,D}\times C\to \fix_{\sigma,D}$. 

 Then the torus weights for each term in the above expression are given by $\varepsilon_{\sigma(u)}-\varepsilon_{\sigma(s)}$. The virtual normal bundle for $ \fix_{\sigma,D}\subset \Hquot_{\Bd}$, denoted $\Nor$, is the $K$-theory class on $\fix_{\sigma,D}$ given by moving part of the virtual tangent bundle, i.e., by imposing the condition $s\ne u$:
\begin{align}\label{eq:Normal_bundle_K-theory}
\Nor=\pi_*\sum_{j=1}^{k}\left(\sum_{\substack{\ s\le r_{j},  u\le r_{j+1}\\  s\ne u }}^{} \CK^{\vee}_{s,j}\otimes \CK_{u,{j+1}} -\sum_{\substack{ s, u\le r_{j}\\ s\ne u }}^{} \CK^{\vee}_{s,j}\otimes \CK_{u,j}  \right).
\end{align}
Note that the stationary part of the virtual tangent bundle, that is, when $s=u$, yields the classical tangent bundle on $\fix_{\sigma,D}$.

\section{Integrals over Hyperquot scheme on $\BP^1$}\label{sec:genus_zero_calculation}
In this section, we will illustrate the proof of Theorem~\ref{thm:intro_VI_formula} when $C=\BP^1$ and $V$ is a trivial bundle of rank $n$. The Hyperquot scheme $\Hquot_{\Bd}=\Hquot_{\Bd}(\BP^1,\Br,\CO_{\BP^1}^{\oplus n})$ is a smooth projective variety of the expected dimension, and the virtual class coincides with the usual fundamental class. We will return to the general case in Section~\ref{sec:higher_genus_case}, after introducing additional tools that are needed.

We apply the Atiyah--Bott localization formula to express the integrals over $\Hquot_\Bd$ in terms of integrals over the fixed loci $\fix_{\sigma, D}$, which in this genus-zero setting are simply products of projective spaces. We can easily express these integrals as residues, and we determine the generating series of these residues by invoking the multivariate Lagrange-B\"urmann formula, thus proving Theorem~\ref{thm:VI_Hyper_quot_equivariant} in this genus-zero setting.

Since we'll compute by localization, we use equivariant cohomology with complex coefficients and introduce the field of fractions of $H_T^\ast(\pt) \simeq \BC[\varepsilon_1, \dots, \varepsilon_n]$, which we denote by
     \begin{equation}\label{eq:Gamma}
        \Gamma\coloneqq \BC(\varepsilon_1,\dots,\varepsilon_n).
     \end{equation}
     Note that $S_n$ acts on $H^\ast_T(\pt)$ (and on $\Gamma$) by shuffling the variables
\begin{align*}
    S_n \curvearrowright H^\ast_T(\pt) \quad : \quad \sigma \cdot \varepsilon_i \coloneqq  \varepsilon_{\sigma(i)},
\end{align*}
and this action extends to the $T$-equivariant cohomology of any $T$-fixed scheme $F$. Via the choice of representatives \eqref{eq:Sn_representative}, for every $\sigma \in \mathfrak{S}_\Br$ and equivariant cohomology class $\alpha \in H^\ast_T(F)$, the expression $\sigma \cdot\alpha$ makes sense as an element of $H^\ast_T(F)$.
\subsection{Equivariant Euler class in genus zero}

 In $K$-theory, the class of the tangent bundle of $\Hquot_d$ equals the virtual tangent bundle given in \eqref{eq:Tangent_Bundle_QuotFlag}. Recall that the fixed loci $\fix_{\sigma,D}$ described in Section~\ref{subs:fixed_loci} are products of Hilbert schemes of points on $\BP^1$, and therefore isomorphic to products of projective spaces
\[
\fix_{\sigma,D}\cong \prod_{j=1}^{k}\prod_{s=1}^{ r_j}\BP^{d_{s,j}-d_{s,j+1}},
\]
with the convention that \(d_{s,k+1}=0\) for all \(s\).
Note that over \(\fix_{\sigma,D}\times \bb{P}^1 \), we have \begin{align}\label{eq:universal_line_bundle_P1}
	\CK_{s,j}
    &=\left(\bigboxtimes_{\ell\geq j}\ca{O}_{\bb{P}^{d_{s,\ell}-d_{s,\ell+1}}}(-1)\right)\boxtimes\CO_{\BP^1}(-d_{s,j}).
	\end{align}
Let $$x_{s,j}\in H^2(\BP^{d_{s,j}-d_{s,j+1}},\BZ) \quad \text{for all } 1\le j\le k,\ 1\le s\le r_j$$
be the hyperplane classes, and use the same notation for their pullbacks in $H^2\left(\fix_{\sigma,D},\BZ\right)$ under the projection map. 
Fixed the identity $\sigma = 1 \in \mathfrak{S}_\Br$, the equivariant first Chern class of the dual of $\CK_{s,j}$ on $\fix_{1,D}$ is given by 
$$c_1^T(\CK_{s,j}^\vee) =z_{s,j}\otimes 1+ 1\otimes d_{s,j}[p] \qquad\quad \text{in } H^2_T(\fix_{1,D}\times \BP^1)$$
under the K\"unneth isomorphism, where $[p]$ is the class of a point $p\in\BP^1$ and 
\begin{align}\label{eq:z_classes_genus0}
    z_{s,j} \coloneqq c_1^T(\CK_{s,j|p}^\vee)=x_{s,j}+x_{s,{j+1}}+\cdots+x_{s,k}+\varepsilon_{s},
\end{align}
with $\CK_{s,j|p}$ denoting the restriction  of $\CK_{s,j}$ to $\fix_{\sigma,D}\times \{p\}$. 
For an arbitrary element $\sigma \in \mathfrak{S}_\Br$, we can describe the K\"unneth decomposition of $c_1^T(\CK_{s,j}^\vee)$ on $\fix_{\sigma,D}\simeq X_{1,D}$ as
$$c_1^T(\CK_{s,j}^\vee) = (\sigma \cdot z_{s,j})\otimes 1+ d_{s,j}\otimes [p] \qquad\quad \text{in } H^2_T(\fix_{\sigma,D}\times \BP^1).$$

In what follows we will use bold symbols as a shorthand notation for tuples, for instance, we write \(\mathbf{x}\coloneqq(x_{s,j})_{j\le k,\ s\le r_j}\).

In our computations, we will use the polynomials     \begin{equation}\label{eq:def_of_R}
        \begin{aligned}
        R_{j}(z) &\coloneqq (z-z_{1,j})(z-z_{2,j})\cdots(z-z_{r_j,j})
         \qquad \text{for }1\le j\le k,\\
        R_{0}(z) &\coloneqq 1, \qquad
        R_{k+1}(z) \coloneqq (z-\varepsilon_1)\cdots(z-\varepsilon_n)
    \end{aligned}
    \end{equation}
    of degree $\deg(R_j)=r_j$, and denote by $R_j^\prime$ their derivative. 
\begin{proposition}\label{prop:Euler_class_normal_genus_0}
    The equivariant Euler class of the normal bundle is given by
    \begin{align}\label{eq:Euler_class_Normal_bundle_genus_0}
        \frac{1}{e_T(\Nor)} = \sigma \cdot (-1)^{\beta}\prod_{j=1}^{k}\prod_{s\le r_j}x_{s,j}^{d_{s,j}-d_{s,j+1}+1}
        \frac{R_j^\prime(z_{s,j})  R_{{j-1}}(z_{s,j})^{d_{s,j}} }{R_{{j+1}}(z_{s,j})^{d_{s,j}+1} },
    \end{align}
    where  the sign is given by 
    $$\beta=\sum_{j\leq k}d_j(r_j-r_{j-1}-1).$$
    The right-hand side should be interpreted as a rational function of the formal variables $\Bx$ and $\boldsymbol{\varepsilon}$, which induces a localized equivariant cohomology class in $H^\ast(\fix_{\sigma,D})\otimes_{H_T^\ast(\pt)}\Gamma$ only after clearing the common factors in the numerators and denominators.
\end{proposition}
\begin{proof}
Without loss of generality we can assume $\sigma=1$. Using the description in \eqref{eq:universal_line_bundle_P1} and the projection formula, we note that the equivariant Euler classes of the pushforwards appearing in \eqref{eq:Normal_bundle_K-theory} are given by
\begin{align*}
e_T\left(\pi_*\left(\CK_{s,j}^\vee\otimes\CK_{u,j}\right)\right)&=\left(z_{s,j}-z_{u,j}\right)^{d_{s,j}-d_{u,j}+1};
        \\ e_T\left(\pi_*\left(\CK_{s,j}^\vee\otimes\CK_{u,{j+1}}\right)\right)&=\left(z_{s,j} -z_{u,{j+1}} \right)^{d_{s,j}-d_{u,{j+1}}+1}.
\end{align*}

We now carefully substitute the equivariant Euler classes computed above into the formula \eqref{eq:Normal_bundle_K-theory} for the normal bundle $\Nor$, noting that the equivariant Euler class is multiplicative. We first consider the terms in the same column (i.e. keeping the subscript $j$ fixed):
\begin{align*}
    \prod_{\substack{s,u\le r_j\\ s\ne u}}^{} e_T\left(\pi_*\left(\CK^{\vee}_{s,j}\otimes \CK_{u,j}\right)\right)
 &=(-1)^{d_j(r_j-1)}
        \prod_{\substack{s,u\le r_{j}\\s\ne u }} \left(z_{s,j} -z_{u,{j}} \right)
        \\
        &=(-1)^{d_j(r_j-1)}\prod_{s\le r_j}R_{j}^\prime(z_{s,j}).
\end{align*}
In the first equality, we used that $\left(z_{s,j}-z_{u,j}\right)^{d_{s,j}-d_{u,j}}$ cancel each other by swapping the role of $s$ and $u$, leaving behind the sign $(-1)^{d_j(r_j-1)}$.

We then consider the terms involving two consecutive columns:
\begin{align*}
    \prod_{\substack{s\le r_{j}\\ u\le  r_{j+1}\\ s\ne u }}e_T\left(\pi_*\left(\CK^{\vee}_{s,j}\otimes \CK_{u,{j+1}}\right)\right)= & \prod_{\substack{s\le r_{j}\\ u\le r_{j+1}\\ s\ne u }} \left(z_{s,j} -z_{u,{j+1}} \right)^{d_{s,j}-d_{u,{j+1}}+1}
    \\
    =&\prod_{\substack{s\le r_j}}\left(\frac{R_{{j+1}}(z_{s,j}) }{x_{s,j}}\right)^{d_{s,j}+1}\cdot \prod_{u\le r_{j+1}}\left(\frac{(-1)^{r_j}x_{u,j}}{R_{j}(z_{u,{j+1}})}\right)^{d_{u,{j+1}}},
    \end{align*}
    where we set \(x_{u,j}\coloneqq 1\) for \(u> r_j\).     
    Substituting these two expressions in the formula for the Euler class of \eqref{eq:Normal_bundle_K-theory} we get the required identity.
    \end{proof}

\subsection{Atiyah--Bott localization}\label{sec:genus_zero_VI}
Recall that $\Hquot_d:=\Hquot_{\Bd}(\BP^1,\Br,V)$ for a split vector bundle $V=M_1\oplus\cdots\oplus M_n$, where $M_{i}^\vee \cong\CO_{\BP^1}$ endowed with equivariant weight $\varepsilon_{i}$ for each $1\le i\le n$. 

Let $Q \in H_T^\ast(\Hquot_\Bd)$ be a polynomial in the equivariant Chern classes of the bundles $\CE^\vee_{1\vert p}, \dots, \CE^\vee_{k\vert p}$, and let $\tilde{Q}(\Bz)$ be the corresponding symmetric polynomial such that $Q\vert_{\fix_{1,D}} = \tilde{Q}(\Bz)$, where $\Bz\coloneqq  (z_{s,j})$ are the Chern roots \eqref{eq:z_classes_genus0}. Note moreover that $Q_{\vert X_{\sigma, D}} = \sigma \cdot \tilde{Q}(\Bz)$.
By the Atiyah-Bott localization formula \cite{AB:moment} we find
\begin{align*}
    \int_{\Hquot_{\Bd}}  Q &= \sum_{\sigma,D}\int_{\fix_{\sigma,D}} \
        \frac{Q|_{\fix_{\sigma,D}}}{e_{T}(\Nor)}.
\end{align*}
Note that the equivariant Euler class in the denominator is invertible in the localized equivariant cohomology $H^\ast_T(X_{\sigma,D}) \otimes_{H_T^\ast(\pt)} \Gamma$.

To evaluate the integrals on the right-hand side, we just have to consider the top-power of the hyperplane class on each projective space $\mathbb{P}^{d_{s,j} - d_{s,j+1}}$, which amounts to extracting the coefficient of $x_{s,j}^{d_{s,j} - d_{s,j+1}}$ for all $1 \le j \le k$ and $1 \le s \le r_j$.
 Therefore, using \eqref{eq:Euler_class_Normal_bundle_genus_0}, we express the integral as the residue
\begin{align}\label{eq:int-genus0}
 \int_{[\Hquot_{\Bd}]}  Q=(-1)^{\beta}\sum_{\sigma,D}\,\sigma \cdot \Res_{\mathbf{x}=\mathbf{0}}\, \Bigl(\tilde{Q}(\Bz)\prod_{\substack{j\le k\\ s\le r_j}}\frac{ R_{j}^\prime(z_{s,j}) R_{{j-1}}(z_{s,j})^{d_{s,j}} }{R_{{j+1}}(z_{s,j})^{d_{s,j}+1}}\Bigr).
\end{align}
The argument of the residue should be understood as a rational function in the formal variables $x_{s,j}$ and $\varepsilon_j$, with $z_{s,j}$ defined as in (\ref{eq:z_classes_genus0}). This can be written uniquely, by expanding each linear factor in the denominator as 
\begin{align}\label{eq:pw_expansion}
    (f(\Bx) + g(\boldsymbol{\varepsilon}))^{-1} = \sum_{k \geq 0} (-1)^k f(\Bx)^{k}g(\boldsymbol{\varepsilon})^{-k-1},
\end{align} 
as a multivariate power series in the variables $x_{s,j}$ having coefficients in the field $\Gamma$. Then the residue $\Res_{\Bx=0}$ extracts the coefficient of $1/\prod_{j,s}x_{s,j}$. 

We now apply the following multivariate Lagrange--B\"urmann formula, see \cite[Theorem 2]{gessel} for instance, to simplify the above residue:
\begin{theorem}[Lagrange--B\"urmann formula]\label{thm:Lagrange_Burmann}
Consider multivariate formal power series $\Phi_1, \dots, \Phi_N$ in variables $x_1, \dots, x_N$, each with non-zero constant term. Then the following change of variables
\[
t_i = \frac{x_i}{\Phi_i} \quad \text{for } 1 \le i \le N,
\]
 has unique inverses $x_i(t_1,t_2,\dots,t_N)$, which are multivariate power series for each $1\le i\le N$. Moreover, for any Laurent series $\Psi$ in $x_1, \dots, x_N$,
\begin{align}
\sum_{m_1, \dots, m_N \ge 0} t_1^{m_1} \cdots t_N^{m_N} 
\operatorname{Res}_{\mathbf{x} = 0} \left( \Psi \cdot \prod_{i=1}^N \left( \frac{\Phi_i}{x_i} \right)^{m_i } \right) 
= \Psi \cdot \det\left(t_{i}^{-1} \frac{\partial t_i}{\partial x_j} \right)_{1 \le i,j \le N}^{-1},
\end{align}
where $\Psi$ on the right-hand side is understood to be expressed in terms of the variables $t_1, \dots, t_N$.
\end{theorem}

 Define the rational functions
\begin{align*}
    \Psi_{s,j}=\frac{x_{s,{j}}R_{{j-1}}(z_{s,j})}{x_{s,{j-1}}R_{{j+1}}(z_{s,j})}
\end{align*}
where we set \(x_{s,j}\coloneqq 1\) for \(s>r_j\). Expanding as in \eqref{eq:pw_expansion}, each $\Psi_{s,j}$ is a formal power
series in the $x$–variables whose constant term is nonzero in $\Gamma$.
For each $j\le k$ and $s\le r_j$ we define
\begin{align}\label{eq:Phi_class}
    \Phi_{s,j}\coloneqq 
    \Psi_{s,j}\Psi_{s,j-1}\cdots \Psi_{s,1},
    \qquad
    \text{with }\;\Psi_{s,j}\coloneqq 1\text{ if }s>r_j,
\end{align}
which is again a multivariate power series with non-zero constant terms. Consider the change of variables
\begin{align}\label{eq:tx_variables}
    t_{s,j}\coloneqq  \frac{x_{s,j}}{\Phi_{s,j}}\quad  \qquad \text{invertible by Theorem \ref{thm:Lagrange_Burmann}}.
\end{align}
\begin{lemma}\label{lem:fundamental_sol}
    Consider the system of equations
    \begin{equation}\label{eq:change_of_variable_q_a,r_j}
    R_{j+1}(z_{s,j}) = q_{s,j}R_{j-1}(z_{s,j}),
    \end{equation}
    in the variables $z_{s,j} \in \Gamma[\![q_{i,\ell}]\!]$ and consider the tuple $\boldsymbol{\zeta}^\fund = (\zeta^\fund_{s,j})$ of power series
    \begin{align}\label{eq:fundamental_solution_qsj}
        \zeta^\fund_{s,j} \coloneqq  \sum_{\ell\geq j} x_{s,\ell}(\Bt) + \varepsilon_s  \qquad \text{evaluated at} \qquad t_{s,j} = q_{s,j}q_{s,j-1}\cdots q_{s,1},
    \end{align}
    using the notation $q_{s,j} \coloneqq 1$ whenever $s>r_j$. Then $\boldsymbol{\zeta}^\fund$ is a solution of the system \eqref{eq:change_of_variable_q_a,r_j}, and we call it the \textit{fundamental solution}. 
\end{lemma}
\begin{proof}
    The following system in the variables $t_{s,j} \in \Gamma[\![q_{i,\ell}]\!]$
    \begin{equation}\label{eq:qt_variables}
        t_{s,j-1} q_{s,j} \coloneqq  t_{s,j},
    \qquad
    \text{can be solved as }\qquad t_{s,j} = q_{s,j}q_{s,j-1}\cdots q_{s,1},
    \end{equation}
    using the notation $t_{s,j}\coloneqq  q_{s,j} \coloneqq 1$ whenever $s>r_j$. Expressed in the $x$-variables, these equations are equivalent to \eqref{eq:change_of_variable_q_a,r_j}. Since \eqref{eq:tx_variables} can be inverted and \eqref{eq:qt_variables} can be solved, we can express $z_{s,j}$ in terms of the $q$-variables, and we obtain the solution $\zeta^\fund_{s,j}$.
\end{proof}

\begin{proposition}\label{prop:Residue_LB_genus_0}
    For any fixed splitting degree $D$, we have 
\begin{align}\label{eq:residue_calculation_g=0}
\Res_{\mathbf{x}=\mathbf{0}}\,\tilde{Q}(\Bz)\prod_{\substack{j\le k\\ s\le r_j}}\frac{ R_{j}^\prime(z_{s,j}) R_{{j-1}}(z_{s,j})^{d_{s,j}} }{R_{{j+1}}(z_{s,j})^{d_{s,j}+1}}= [\mathbf{q}^D]\,\tilde{Q}(\Bz)\prod_{\substack{j\le k\\ s\le r_j}}\frac{ R_{j}^\prime(z_{s,j})  }{R_{j-1}(z_{s,j})} \det\left( \frac{\partial q_{s,j}}{\partial z_{u,i}}\right)^{-1} \big\vert_{\Bz = \boldsymbol{\zeta}^\fund},
\end{align}
where the determinant is the Jacobian of \eqref{eq:change_of_variable_q_a,r_j} and $[\Bq^D]$ denotes the operation of extracting the coefficient of $\mathbf{q}^D\coloneqq  \prod_{j=1}^k \prod_{s=1}^{r_j}q_{s,j}^{d_{s,j}}$ in a power series in the variables $q_{s,j}$. 
\end{proposition}
\begin{proof}
A straightforward computation gives
\begin{align*}
\prod_{j=1}^{k}\prod_{s=1}^{r_j}
\biggl( \frac{R_{j-1}(z_{s,j})}{R_{j+1}(z_{s,j})} \biggr)^{d_{s,j}}
&=
\prod_{j=1}^{k}\prod_{s=1}^{r_j}
\biggl( \frac{\Phi_{s,j}}{\Phi_{s,j-1}}  \frac{x_{s,j-1}}{x_{s,j}} \biggr)^{d_{s,j}} =
\prod_{j=1}^{k}\prod_{s=1}^{r_j}
\biggl( \frac{\Phi_{s,j}}{x_{s,j}} \biggr)^{d_{s,j}-d_{s,j+1}},
\end{align*}
where in the last step we use the telescoping in $j$ and the convention $d_{s,k+1}=0$. Define the rational function
\begin{align*}
    \Psi=\tilde{Q}(\Bz)\prod_{s\le r_j,j\le k}\frac{R_{j}^\prime(z_{s,j})}{R_{{j+1}}(z_{s,j})},
\end{align*} 
which expanded as in \eqref{eq:pw_expansion} becomes a Laurent series in the $x$-variables. The left-hand side of \eqref{eq:residue_calculation_g=0} then becomes
\begin{equation}\label{eq:after_LB}
\Res_{x=0}
\Biggl(
\Psi\;
\prod_{j=1}^{k}\prod_{s=1}^{r_j}
\biggl( \frac{\Phi_{s,j}}{x_{s,j}} \biggr)^{d_{s,j}-d_{s,j+1}}
\Biggr) = \left[\prod_{\substack{j\le k \\s\le r_j}}t_{s,j}^{d_{s,j}-d_{s,{j+1}}}\right]\Psi\det\left(t_{s,j}^{-1}\frac{\partial t_{s,j}}{\partial x_{u,i}}\right)^{-1},
\end{equation}
where the equality follows from Theorem~\ref{thm:Lagrange_Burmann} applied to the change of variables \eqref{eq:tx_variables}. We now introduce the variables $q_{s,j}$ that we want to appear in the final formula. For every $1\le j\le k$ and $1\le s\le r_j$ we can specialize
\begin{equation}\label{eq:specialise_t}
    \qquad t_{s,j} = q_{s,j}q_{s,j-1}\cdots q_{s,1},
\end{equation}
using the notation $t_{s,j}\coloneqq  q_{s,j} \coloneqq 1$ whenever $s>r_j$. Then \eqref{eq:after_LB} equals
\begin{equation}\label{res:interm}
    \left[\mathbf{q}^D\right]\Psi\det\left(t_{i,j}^{-1}\frac{\partial t_{i,j}}{\partial x_{\ell,m}}\right)^{-1},
\end{equation}
once we express the argument of $[\Bq^D]$ as a multivariate Laurent series in the $q$-variables using \eqref{eq:specialise_t}. We can make this more explicit by
\[
\det\left(t_{s,j}^{-1}\frac{\partial t_{s,j}}{\partial q_{b,m}}\right)^{-1} 
=\prod_{ j\le k,\ s\le r_j}q_{s,j}=  \prod_{ j\le k,\ s\le r_j} \frac{x_{s,j}\Phi_{s,j-1}}{x_{s,j-1}\Phi_{s,j}}=\prod_{ j\le k,\ s\le r_j}\frac{R_{{j+1}}(\zeta^\fund_{s,j})}{R_{{j-1}}(\zeta^\fund_{s,{j}})},
\]
thus, by the chain rule, \eqref{res:interm} equals
\begin{align*}
&\left[\mathbf{q}^D\right]\Psi \prod_{ j\le k, s\le r_j}\frac{R_{j+1}(z_{s,j})}{R_{{j-1}}(z_{s,{j}})}\det\left(\frac{\partial q_{s,j}}{\partial z_{b,m}}\right)^{-1}\big\vert_{\Bz = \boldsymbol{\zeta}^\fund}.\qedhere
\end{align*}
 \end{proof}

\begin{definition}
    Consider a system of equations in the variables $z_{s,j}$ with coefficients in $\Gamma[q_1, \dots, q_\ell]$, where $q_j$ are formal variables.
    We call a \textit{formal solution} a tuple $\boldsymbol{\zeta} = (\zeta_{s,j})$
    of formal power series $\zeta_{s,j} \in \Gamma[\![q_1, \dots, q_\ell]\!]$ solving the system.
    We say that a formal solution is \emph{non-degenerate} if, for every $1 \le j \le k$, the power series $\zeta_{1,j},\dots,\zeta_{r_j,j}$ have pairwise distinct images in
    $\Gamma[\![q_1, \dots, q_k ]\!]/(q_1,\dots,q_k) \cong \Gamma$ (equivalently, their constant terms $\zeta_{1,j}(0),\dots,\zeta_{r_j,j}(0) \in \Gamma$ are
    pairwise distinct).
\end{definition}

Note that we can easily characterize the formal non-degenerate solutions of (\ref{eq:change_of_variable_q_a,r_j}):
\begin{lemma}\label{lem:counting_formal_solutions}
    Specialize \eqref{eq:change_of_variable_q_a,r_j} to $\Bq_j \coloneqq  (q_j, \dots, q_j)$, i.e., consider the system
    \begin{equation}\label{eq:specialised_equations}
        R_{j+1}(z_{s,j}) = q_j R_{j-1}(z_{s,j}) \qquad\qquad 1 \leq j \leq k, \quad 1\leq s\leq r_j
    \end{equation}
    in the variables $z_{s,j} \in \Gamma[\![q_1, \dots, q_k]\!]$. The specialization of the fundamental solution \eqref{eq:fundamental_solution_qsj}, which we still denote by $\boldsymbol{\zeta}_\mathrm{fund} = (\zeta^\mathrm{fund}_{s,j})$, is a solution to \eqref{eq:specialised_equations}. Moreover, the set of formal non-degenerate solutions to \eqref{eq:specialised_equations} is in bijection with $\mathfrak{S}_\Br \times \prod_{j=1}^k S_{r_j}$, where $(\sigma, \sigma_1, \dots, \sigma_k)$ corresponds to the solution
    \begin{align}\label{eq:non_deg_solution}
        \boldsymbol{\zeta}_{\sigma, \sigma_1, \dots, \sigma_k} = (\zeta_{s,j}) \qquad \text{with} \qquad \zeta_{s,j}\coloneqq  \sigma \cdot \zeta^\mathrm{fund}_{\sigma_j(s),j}.
    \end{align}
\end{lemma}
    \begin{proof}
    By Lemma \ref{lem:fundamental_sol}, $\zeta^{\mathrm{fund}}_{s,j}$ is a formal solution to \eqref{eq:specialised_equations}. Its constant terms are
$\zeta^{\mathrm{fund}}_{s,j}(0)=\varepsilon_s$, so the solution is
non–degenerate. Clearly $\boldsymbol{\zeta}_{\sigma, \sigma_1, \dots, \sigma_k}$ is still a solution, since the system \eqref{eq:specialised_equations} is invariant under the action of $S_n$ shuffling $\varepsilon_1, \dots, \varepsilon_n$ and under the action of $S_{r_j}$ shuffling the variables $x_{1,j}, \dots, x_{r_j,j}$. All these solutions are distinct and non-degenerate as it can be seen from their constant terms $\boldsymbol{\zeta}_{\sigma, \sigma_1, \dots, \sigma_k}(0) = \varepsilon_{\sigma(\sigma_j(s))}$. It remains to prove that all the non-degenerate solutions are of this form.

To see this, we apply
Hensel’s Lemma \cite[III, §4.5, Corollary~2]{bourbaki} to the ring
$A\coloneqq \Gamma[\![q_1,\dots,q_k]\!]$ with maximal ideal $\mathfrak m\coloneqq (q_1,\dots,q_k)$, so that $A/\mathfrak{m} \simeq \Gamma$. Note that the reduction of \eqref{eq:specialised_equations} mod $\frak{m}$ is the system \begin{equation}\label{eq:reduced_system}
    R_{j+1}(z_{i,j})=0\quad \text{for all } 1\le i\le r_j, \ 1\le j\le k,
\end{equation}
and its only non-degenerate solutions $\Ba = (a_{s,j})$, with $a_{s,j} \in \Gamma$, are the elements of the free orbit of $\mathfrak{S}_\Br \times \prod_{j=1}^k S_{r_j}$ of the fundamental solution 
$a^\fund_{s,j} \coloneqq  \varepsilon_{s}$.

By Hensel's Lemma, for every solution ${\bf{a}}=(a_{s,j})$ to the reduced system \eqref{eq:reduced_system} such that 
\begin{align*}
    \det\left(\frac{R_j(z_{i,j})}{\partial {z_{k,l}}}\right)_{\vert_{\bf{z=a}}} \neq 0 \qquad\qquad \text{in } \Gamma,
\end{align*}
there exists a unique formal solution \(\boldsymbol{\zeta}=(\boldsymbol{\zeta}_1,\boldsymbol{\zeta}_2,\dots, \boldsymbol{\zeta}_k)\) to the system \eqref{eq:specialised_equations} such that $\boldsymbol{\zeta}\equiv\bf{a}$ modulo $\frak{m}$.
Note that \(\parens*{\partial R_{j+1}(z_{i,j})/\partial {z_{k,l}}}\mid_{z_{i,j}=\varepsilon_{i}}\) is triangular with diagonal entries being nonzero elements of \(\BC(\varepsilon_1,\dots,\varepsilon_n)\), thus its determinant is invertible in $\Gamma$. Since reduction modulo $\fr m$ defines a map 
\begin{align*}
    \left\lbrace \text{non-degenerate solutions of \eqref{eq:specialised_equations}}\right\rbrace \xrightarrow{\text{mod } \mathfrak{m}} \left\lbrace \text{non-degenerate solutions of \eqref{eq:reduced_system}}\right\rbrace,
\end{align*}
the uniqueness part of Hensel's lemma proves that all the formal non-degenerate solutions to \eqref{eq:specialised_equations} are of the form \eqref{eq:non_deg_solution}.
\end{proof}

\begin{theorem}\label{thm:VI_Hyper_quot_equivariant_formal}
Let $V$ be as in \eqref{eq:V_split}, then for any tuple $m_{s,j}$, for $1\le j\le k$ and $1\le s\le r_j$,
\begin{align}\label{eq:Intersection_Hquot(C,V)_equivariant_formal}
    \sum_{\Bd\in \BN^k} q_1^{d_1}q_2^{d_2}\cdots q_k^{d_k}
\int_{[\Hquot_{\Bd}]^{\vir}}  
\prod_{j=1}^{k}\prod_{s=1}^{r_j} c_s^{T}(\CE_{j|p}^\vee)^{m_{s,j}}
    =\sum_{\boldsymbol{\zeta}_1,\dots,\boldsymbol{\zeta}_k}  \prod_{j=1}^{k}\frac{1}{r_j!}\prod_{s=1}^{r_j}e_s(\boldsymbol{\zeta}_j)^{m_{s,j}}\cdot J^{g-1}
	\end{align}
    where the sum is taken over all formal non-degenerate solutions $(\boldsymbol{\zeta}_1,\dots,\boldsymbol{\zeta}_k)$ of the equations in \eqref{eq:Bethe_system_equiv} and
    $e_s(\boldsymbol{\zeta}_j)$ denote the $s$-th elementary symmetric polynomial. 
     The factor $J$ equals
     \[ J\coloneqq\parens*{\prod_{\ell=1}^{k} \frac{1}{\Delta(\boldsymbol{\zeta}_\ell)}}\cdot
\det\left(\frac{\partial P^T_j(z_{s,j})}{\partial z_{s',j'}}\right)\Bigg|_{(\Bz_1,\Bz_2,\dots,\Bz_k) = (\boldsymbol{\zeta}_1,\boldsymbol{\zeta}_2,\dots,\boldsymbol{\zeta}_k)}
 ,\]
where $ \Delta(X_1,\dots,X_r)\coloneqq\prod_{a\ne b}(X_a-X_b)$ and the second factor is the Jacobian of \eqref{eq:Bethe_system_equiv}. 
\end{theorem}

\begin{proof}[Proof of Theorem~\ref{thm:VI_Hyper_quot_equivariant_formal} in genus zero]
Define the variables $z_{s,j}$ as in \eqref{eq:z_classes_genus0} and consider the power series in the $q$-variables
\begin{align*}
    \Pi(\Bq_1, \dots, \Bq_k) \coloneqq  \tilde{Q}(\Bz)\prod_{\substack{j\le k\\ s\le r_j}}\frac{ R_{j}^\prime(z_{s,j})  }{R_{j-1}(z_{s,j})} \det\left( \frac{d q_{s,j}}{d x_{u,i}}\right)^{-1}\big\vert_{\Bz = \boldsymbol{\zeta}^\fund}
\end{align*}
with coefficients in $\Gamma$, defined by expressing the $x$-variables on the right-hand side in terms of the variables $\Bq_j = (q_{1,j}, \dots, q_{r_j,j})$ as in \eqref{eq:change_of_variable_q_a,r_j}.

Then the identity \eqref{eq:int-genus0} and Proposition \ref{prop:Residue_LB_genus_0} imply
\begin{equation}\label{eq:genus0-start}
           \int_{[\Hquot_{\Bd}]}  Q = (-1)^{\beta}\sum_{\sigma,D}[\mathbf{q}^D] \Bigl(\, \sigma \cdot \Pi(\Bq_1, \dots, \Bq_k)\,\Bigr),
\end{equation}
where $\sigma$ runs over all elements of $\mathfrak{S}_\Br$ and
$D=(d_{s,j})_{j\le k,\,s\le r_j}$ runs over all splitting degrees satisfying $d_{s,j}\geq d_{s,j+1}$ and
$\sum_{a=1}^{r_j}d_{a,j}=d_j$ for each $s$ and $j$.
In order to simplify \eqref{eq:genus0-start} it is convenient to specialize
this power series by setting the variables at the same level to be equal, thus obtaining a power series in just $k$-variables $q_1, \dots, q_k$ 
\begin{align*}
    \Pi_{\mathrm{spec}}(q_1, \dots, q_k) \coloneqq  \Pi(\Bq_1, \dots, \Bq_k) \qquad \text{where} \qquad \Bq_j \coloneqq  (q_j, \dots, q_j).
\end{align*}
Since $\Pi(\Bq_1, \dots, \Bq_k)$ does not depend on $D$ we can rewrite \eqref{eq:genus0-start} as
\begin{equation}\label{eq:genus0-qj}
\int_{[HQ_\Bd]} Q
=
(-1)^{\beta} [q_1^{d_1}\cdots q_k^{d_k}]
\sum_{\sigma \in \mathfrak{S}_\Br}
\sigma \cdot \Pi_{\mathrm{spec}}(q_1, \dots, q_k).
\end{equation}
By Lemma \ref{lem:counting_formal_solutions} we know that the formal non-degenerate solutions to \eqref{eq:Bethe_system_equiv} are in bijection with the elements of $\mathfrak{S}_\Br \times \prod_{j=1}^k S_{r_j}$, and for every $(\sigma, \sigma_1, \dots, \sigma_k)$ we have
\begin{align*}
    \sigma \cdot \Pi_{\mathrm{spec}}(q_1, \dots, q_k) = \tilde{Q}(\Bz)\prod_{\substack{j\le k\\ s\le r_j}}\frac{ R_{j}^\prime(z_{s,j})  }{R_{j-1}(z_{s,j})} \det\left( \frac{d q_{s,j}}{d x_{u,i}}\right)^{-1}\big\vert_{\Bz = \boldsymbol{\zeta}_{\sigma, \sigma_1, \dots, \sigma_k}}, 
\end{align*}
hence 
\begin{equation}
\int_{[HQ_d]} Q
=
(-1)^{\beta} [q_1^{d_1}\cdots q_k^{d_k}]
\sum_{\boldsymbol{\zeta}}
\tilde{Q}(\Bz)\prod_{\substack{j\le k\\ s\le r_j}}\frac{ R_{j}^\prime(z_{s,j})  }{R_{j-1}(z_{s,j})} \det\left( \frac{d q_{s,j}}{d x_{u,i}}\right)^{-1}\big\vert_{\Bz = \boldsymbol{\zeta}},
\end{equation}
where the sum is over the solutions of the system \eqref{eq:change_of_variable_q_a,r_j}.
Finally, a straightforward change of variables rewrites the Jacobian in the
desired form, and the substitution
\[
q_j \longmapsto (-1)^{r_j-r_{j+1}-1} q_j
\]
cancels the factor $(-1)^\beta$ in~\eqref{eq:genus0-qj}. This yields the
formula in Theorem~\ref{thm:VI_Hyper_quot_equivariant_formal} in genus zero.
\end{proof}

Note that the final expression must be polynomial in both \(\varepsilon\) and \(q\) variables. In the next section, we prove that this polynomial can be evaluated by first specializing the \(\varepsilon\) and \(q\) variables in \eqref{eq:specialised_equations} to values in an open subset of \(\BC^n\times\parens*{\BC^*}^k\) intersecting \(0\times\parens*{\BC^*}^k\), and then summing over its non-degenerate solutions, as stated in Theorem~\ref{thm:VI_Hyper_quot_equivariant}. In the process, we prove that we still obtain the correct number of non-degenerate solutions when we specialize the \(\varepsilon\) and \(q\) variables in this way.

\subsection{Equations}\label{sec:Equations}
Consider a chain of integers $0 = r_0 < r_1 \leq r_2 \leq \dots \leq r_k \leq r_{k+1}=n$ and set ${|\Br|} \coloneqq  \sum_{j=1}^k r_j$. Let $(z_{s,j}:1\le j\le k,1\le s\le r_j)$ be a tuple of ${|\Br|}$ variables and $\Bq \coloneqq  (q_1, \dots, q_k)$ be a fixed tuple of complex numbers, we consider the equations
\begin{align}\label{eq:system}
	\prod_{a=1}^{r_{j+1}}(z_{s, j}-z_{a,j+1}) - q_j \prod_{b=1}^{r_{j-1}}(z_{s,j}-z_{b, j-1}) = 0 \tag{$\mathrm{P}_{s,j}(\Bq)$}
\end{align}
for every $1\le j\le k$ and $1\le s\le r_j$, where we use the notation $z_{s,k+1}=0$ for all $s$. We denote by $X(\Bq)$ the closed subscheme of the affine space $\BA^{|\Br|}$ cut out by the equations (\ref{eq:system}) and by $N(\Bq)$ the cardinality of the set of closed points in $X(\Bq)$, counted without multiplicity. We are interested in proving the following result:
\begin{proposition}\label{number_of_solutions}
	For a generic choice of $\Bq \coloneqq  (q_1, \dots, q_k) \in (\BC^\ast)^k$, the system $($\ref{eq:system}$)$ admits the expected number of distinct solutions computed by B\'ezout's theorem, namely:
	\begin{align*}
		N(\Bq) = \prod_{j=1}^k r_{j+1}^{r_j}.
	\end{align*}
	Moreover, the number of solutions satisfying $z_{s,j} \neq z_{u,j}$ for all $1\le j \le k$ and all $s \neq u$ is $\prod_{j=1}^k \frac{r_{j+1}!}{(r_{j+1}-r_j)!}$.
\end{proposition}

\subsubsection{Idea of the proof}
Before attacking the problem by \textit{Gr\"obner degeneration}, let's explain the idea. We will construct a family of closed subschemes of $\BP^{{|\Br|}}$ 
\[\begin{tikzcd}
	\CX \arrow[r, hook, "i"] \arrow[dr, swap, "\pi"] &\BP^{{|\Br|}} \times (\BC^\ast)^k \times \BA^1 \arrow[d, "\pi_{q{,}t}"]\\
	& (\BC^\ast)^k \times \BA^1
\end{tikzcd}\]
where the generic fiber $\CX_{\vert (\Bq, t)}$, for $t \neq 0$ is isomorphic to $X(\tilde{\Bq})$ for some $\tilde{\Bq}$ depending on $\Bq$ and $t$, while the special fiber above $(\Bq, 0)$ is a nice smooth zero-dimensional scheme having the expected number of points. We will show that this family is smooth at every point of the form $(\Bq,0)$ on the base, and since smoothness of projective morphisms is an open property on the base we can move all the results to nearby fibers of the form $\CX_{\vert (q,t)}$ with $t \neq 0$, thus proving the claimed result.
\subsubsection{Construction of the degeneration}
For every $j \in [k]$ define
\begin{align*}
	\rho_j \coloneqq  r_{j+1}-r_{j-1} \in \bb{N}.
\end{align*}
Consider the homogeneous version of (\ref{eq:system}), which is 
\begin{align}\label{eq:h_system}
	\prod_{a=1}^{r_{j+1}}(z_{s, j}-z_{a,j+1}) \tag{$\mathrm{HP}_{s,j}(\Bq)$} - q_j Z_0^{\rho_j}\prod_{b=1}^{r_{j-1}}(z_{s,j}-z_{b, j-1}) = 0
\end{align}
once we introduce the extra variable $Z_0$.
Consider the $\BC^\ast$-action on $\BP^{{|\Br|}} \times (\BC^\ast)^k$ given by
\begin{align*}
	t\cdot z_{s,j} \coloneqq  t^j z_{s,j}, \quad t\cdot Z_0\coloneqq  Z_0 \quad \text{and} \quad t \cdot q_{j} \coloneqq  t^{j\rho_j + r_{j-1}}q_j.
\end{align*}
We construct the family over $\BA^1 \setminus 0 = \mathrm{Spec}(\BC[t]_t)$ given by the zero locus of the perturbed equations
\begin{align*}
	\prod_{a=1}^{r_{j+1}}t^j(z_{s,j}-t z_{a, j+1}) - t^{j \rho_j + r_{j-1}} q_j Z_0^{\rho_j} \prod_{b=1}^{r_{j-1}}t^{j-1}(t z_{s,j}-z_{b, j-1}) = 0 
\end{align*}
inside $\BP^{{|\Br|}} \times (\BC^\ast)^k \times (\BA^1 \setminus 0)$. 
\begin{remark}
	This is the family over $\BC^\ast = \BA^1\setminus 0$ obtained by moving the set of solutions to (\ref{eq:h_system}) via the $\BC^\ast$-action we just defined.
\end{remark}
There is a classical way to extend this family (which is not flat a priori) to the whole $\BA^1$ via \cite[Theorem 15.17]{eisenbud:CAbook}, and the resulting family is the so-called \textit{Gr\"obner degeneration} of (\ref{eq:h_system}).
This is obtained by factoring out all the extra powers of $t$, so that each equation admits a monomial of degree zero and no term of negative degree. Our equations become
\begin{align}\label{C_system}
	\prod_{a=1}^{r_{j+1}}(z_{s,j}-tz_{a, j+1})- q_j Z_0^{\rho_j}\prod_{b=1}^{r_{j-1}}(t z_{s,j}-z_{b, j-1}) = 0 \tag{$\mathrm{HP}_{s,j}(\Bq,t)$}
\end{align}
and their zero locus $\CX \hookrightarrow \BP^{|\Br|} \times (\BC^\ast)^k \times \BA^1$ defines a scheme over $(\BC^\ast)^k \times \BA^1$ whose generic fiber is isomorphic to the zero locus of (\ref{eq:h_system}) in $\BP^{|\Br|}$, while the special fiber above $(\Bq,0)$ is the zero locus of 
\begin{align}\label{eq:initial_equations}
	z_{s,j}^{r_{j+1}} - (-1)^{r_{j-1}} q_j Z_0^{\rho_j}\prod_{b=1}^{r_{j-1}} z_{b, j-1} = 0 \tag{$\mathrm{HP}_{s,j}(\Bq,0)$}.
\end{align}
\subsubsection{Properties of the fibers}
\begin{lemma}\label{no_infinity_lemma}
	For every $\Bq \in (\BC^\ast)^k$, there is a Zariski open neighborhood $U$ of $0 \in \BA^1$ so that the fiber $\CX_{\vert (\Bq,t)}$ doesn't meet the hyperplane at infinity $Z_0=0$ for every $t \in U$.
\end{lemma}
\begin{proof}
	The statement follows from $\CX_{|(\Bq,0)} \cap \lbrace Z_0=0 \rbrace = \emptyset$, so we'll prove this. Clearly $\rho_1>0$ by definition, so the equation ($\mathrm{HP}_{s, 1}(\Bq,0)$) specialized to $Z_0=0$ implies $z_{s,1}=0$ for all $s$. Now fix $j>1$ and assume $z_{b,j-1}=0$ for every $b$. The equation (\ref{eq:initial_equations}) implies $z_{s,j}=0$ for every $s$, hence by induction we obtain that if a point of $\BP^{|\Br|}$ satisfies both $Z_0=0$ and the equations (\ref{eq:initial_equations}), then $z_{s,j}=0$ for all $s$ and $j$, which is a contradiction. 
\end{proof}
\begin{lemma}\label{comparison_lemma}
	Consider the projection $\CX \rightarrow (\BC^\ast)^k \times \BA^1$. Given any $\Bq$ and a generic $t \neq 0$, the fiber $\CX_{\vert (\Bq,t)}$ is isomorphic to $X(t \cdot \Bq)$.
\end{lemma}
\begin{proof}
	For a generic $t$, by Lemma \ref{no_infinity_lemma} there are no solutions at $Z_0=0$, so we can work in the affine space $\BA^{|\Br|} = \lbrace Z_0\neq0 \rbrace$.
	The global automorphism of $\BA^{|\Br|}$ given by
	\begin{align*}
		z_{s,j} \mapsto t^j z_{s,j}
	\end{align*}
	maps the equations defining $X(t \cdot \Bq)$ into the equations defining $\CX_{\vert (\Bq,t)}$.
\end{proof}
\begin{lemma}\label{structure_special_fiber}
	For every $\Bq \in (\BC^\ast)^k$ the fiber $\CX_{\vert (\Bq,0)}$ is zero-dimensional, smooth, it consists of $\prod_{j=1}^k r_{j+1}^{r_j}$ distinct points and it doesn't intersect any coordinate hyperplane in $\BP^{{|\Br|}}$. Moreover, the number of points in $\CX_{\vert (\Bq,0)}$ that satisfy $z_{s,j} \neq z_{u,j}$ for all $s \neq u$ and $1\le j\le k$ is precisely $\prod_{j=0}^k \frac{r_{j+1}!}{(r_{j+1}-r_j)!}$.
\end{lemma}
\begin{proof}
	By Lemma \ref{no_infinity_lemma} we can work in affine coordinates setting $Z_0=1$. The system (\ref{eq:initial_equations}) can be solved explicitly. For $j=1$ we find 
	\begin{align*}
		z_{s,1} \text{ is a $r_2$-th root of $q_1$}.
	\end{align*}
	Then from $z_{b,1}$ we can recover $z_{s,2}$ via \eqref{eq:initial_equations} and so on:
	\begin{align*}
		z_{s,j} \text{ is a $r_{j+1}$-th root of $(-1)^{r_{j-1}}q_j \prod_{b=1}^{r_{j-1}}z_{b,j-1}$}.
	\end{align*}
	Since $q_j \neq 0$ for every $j$, the number of such choices is precisely $\prod_{j=1}^k r_{j+1}^{r_j}$, and in particular $z_{s,j} \neq 0$ for every point in $\CX^\ast_{\vert (\Bq,0)}$. Finally, counting solutions with $z_{s,j} \neq z_{u,j}$ yields the second part of the claim.
\end{proof}
\subsubsection{Smoothness of the family and proof of Proposition \ref{number_of_solutions}}
\begin{lemma}\label{lem:smooth_family}
	The projection
	\begin{align*}
		(\pi_q, \pi_t) : \CX \rightarrow (\BC^\ast)^k \times \BA^1.
	\end{align*}
	is smooth at every point of $(\BC^\ast)^k \times 0$.
\end{lemma}
\begin{proof}
	Consider the Jacobian matrix $J(\Bz,\Bq, t)$ for the equations (\ref{eq:initial_equations}). The ${|\Br|} \times {|\Br|}$ submatrix $J_{\partial_z}(\Bz,\Bq, 0)$ of $J(\Bz,\Bq, 0)$ corresponding to the $z$-derivatives $\partial/\partial_{z_{i,j}}$ is lower triangular, hence its determinant is the product of the diagonal terms:
	\begin{align*}
		\det J_{\partial_z}(\Bz,\Bq, 0) = \prod_{j=1}^k \prod_{s=1}^{r_j} (-1)^{r_{j+1}} r_{j+1} z_{s,j}^{r_{j+1}-1}
	\end{align*}    
	which, by Lemma \ref{structure_special_fiber}, is non-zero at every point of $\CX_{\vert (\Bq,0)}$. By the Jacobian criterion, this implies the claim.
\end{proof}
We know that the locus of points $(\Bq,t)$ in $(\BC^\ast)^k \times \BA^1$ so that $(\pi_q, \pi_t)$ is smooth at $(\Bq, t)$ is open since the map $(\pi_q, \pi_t)$ is projective\footnote{A projective morphism is closed, hence the image of the critical locus of our map is closed and misses the origin of $\BA^1$, so we take its complement.}. Hence for every $\Bq \in (\BC^\ast)^k$ and a generic $t \in \BC^\ast$ the point $(\Bq, t)$ is a regular value of this map. 
\begin{proof}[Proof of Proposition \ref{number_of_solutions}]
	Given such a regular value $(\Bq, t)$ with $t \neq 0$ we have that $\CX_{\vert (\Bq,t)}$ is smooth of dimension zero, so by B\'ezout's theorem it is the union of $\prod_{j=1}^k r_{j+1}^{r_j}$ distinct closed points, and the proof of the first part of the Proposition follows from Lemma \ref{comparison_lemma}. Consider the symmetric group $S_\Br\coloneqq  \prod_{j=1}^k S_{r_j}$ acting on $\BA^{|\Br|}$ permuting the same level coordinates. The ideal generated by the equations (\ref{C_system}) is $S_\Br$-invariant, hence $\CX$ inherits an induced action. Fixed $(\Bq, 0) \in (\BC^\ast)^k \times \BA^1$, from Lemma \ref{structure_special_fiber} we know that the number of points in $\CX_{\vert (\Bq,0)}$ satisfying $z_{s,j} \neq z_{u,j}$ is $\prod_{j=0}^k \frac{r_{j+1}!}{(r_{j+1}-r_j)!}$. This implies that the fiber $\CX_{\vert (\Bq,0)}$ contains precisely $\prod_{j=0}^k \binom{r_{j+1}}{r_j}$ free $S_\Br$-orbits. By smoothness, in a complex analytic neighborhood $U$ of $(\Bq,0)$ the family $\CX_{\vert U}$ is the disjoint union of copies of $U$ indexed by the points of $\CX_{\vert (\Bq,0)}$. The group $S_\Br$ acts by shuffling these copies, and the action is completely determined by the action on the fiber above $\CX_{\vert(\Bq,0)}$, therefore for every $(\Bq,t) \in U$ with $t \neq 0$ we have that $X(t \cdot \Bq)$ has precisely $\prod_{j=0}^k \binom{r_{j+1}}{r_j}$ free $S_\Br$-orbits, concluding the proof.
\end{proof}
\subsubsection{Equivariant version}
Let $\BA^n$ be the space of some additional variables $\varepsilon_1, \dots, \varepsilon_n$. Consider the family over $(\BC^\ast)^k \times \BA^n$ of closed subschemes of $\BA^{{|\Br|}}$ given by the same equations (\ref{eq:system}), but this time interpreted as if $z_{s,k+1}\coloneqq  \varepsilon_s$ for all $s=1, \dots, n$. If we define
\begin{align*}
	\delta\coloneqq  \bigcup_{j=1}^k \bigcup_{1\leq s \neq u \leq r_j} \big\lbrace z_{s,j} = z_{u,j} \big\rbrace \subset \BA^{|\Br|}
\end{align*}
then we can consider the induced family of closed subschemes of $\BA^{|\Br|} \setminus \delta$
\[\begin{tikzcd}
	\CX^\prime \arrow[r, hook, "i"] \arrow[dr, swap, "p"] &(\BA^{{|\Br|}} \setminus \delta) \times (\BC^\ast)^k \times \BA^n \arrow[d, "p_{q,\varepsilon}"]\\
	& (\BC^\ast)^k \times \BA^n
\end{tikzcd}\]
\begin{lemma}\label{lem:trace}
	There is an open subscheme $\mathrm{Reg}(p) \subset \BA^n \times (\BC^\ast)^k$, which is nonempty and intersects the subspace $\varepsilon_1 = \dots = \varepsilon_n=0$, so that $p$ is \'etale on $\mathrm{Reg}(p)$. Moreover, if $f$ is a regular function on $p^{-1}(\mathrm{Reg}(p))\subset \CX^\prime$, then the function
	\begin{align*}
		\mathrm{Tr}_p(f) : \mathrm{Reg}(p) \rightarrow \BC \quad : \quad (\Bq, \boldsymbol{\varepsilon}) \mapsto \sum_{(\Bx, \Bq, \boldsymbol{\varepsilon}) \in p^{-1}(\Bq, \boldsymbol{\varepsilon})}f(\Bx, \Bq, \boldsymbol{\varepsilon})
	\end{align*}
	is a regular function on $\mathrm{Reg}(p)$.
\end{lemma}
\begin{proof}
	The first part of the claim follows by considering first the homogenized equations cutting a family of subschemes of $\BP^{|\Br|}$, showing that there are no solutions on the hyperplane at infinity, and that this family is smooth at a generic point of the form $(\Bq, 0)$ by the non-equivariant Lemma \ref{lem:smooth_family}. Then projectivity ensures that smoothness is generic on the base. Moreover, we know that for a generic point of the base, the number of solutions in $\BA^{|\Br|}\setminus \delta$ stays equal to $\prod_{j=0}^k \frac{r_{j+1}!}{(r_{j+1}-r_j)!}$, hence $\CX^\prime$ is finite \'etale over those points. The second part follows by applying the trace map $\mathrm{Tr}_p : p_\ast \CO_{\CX^\prime} \rightarrow \CO_{\BA^n \times (\BC^\ast)^k}$ to $f$ \cite[\href{https://stacks.math.columbia.edu/tag/0BVH}{Tag 0BVH}]{StacksProj}.
\end{proof}
\subsubsection{Formal and algebraic equations}
Fix $\varepsilon_1, \dots, \varepsilon_n \in \BC^\ast$ with $\varepsilon_s \neq \varepsilon_u$ for all $s \neq u$ and consider the corresponding equivariant equations.  A simple computation proves the following:
\begin{lemma}
	Fixed $q_1 = \dots = q_k = 0$, the equations can be solved in $\BA^{|\Br|} \setminus \delta$. In this case, the number of solutions is $\prod_{j=0}^k \frac{r_{j+1}!}{(r_{j+1}-r_j)!}$ and the $x$-Jacobian is invertible over the solutions.
\end{lemma}
This implies, by the implicit function theorem, that the algebraic solutions in $\BA^{|\Br|} \setminus \delta$ can be upgraded to formal solutions $\boldsymbol{\zeta}_{\boldsymbol{\varepsilon}}(\Bq) \in \BC[\![q_1, \dots, q_k]\!]^{|\Br|}$. Not only that; there is an analytic neighborhood $U \subseteq (\BC^\ast)^k$ of the origin so that for every $\Bq \in U$ the formal solutions above converge, and $\boldsymbol{\zeta}_{\boldsymbol{\varepsilon}}(\Bq)$ are precisely all the solutions in $\BA^{|\Br|}\setminus \delta$ of the system corresponding to $(\Bq, \boldsymbol{\varepsilon})$. Then we find the following:
\begin{lemma}\label{lem:formal_alg_bij}
	Fix $\boldsymbol{\varepsilon} \in (\BC^\ast)^n$. There is a bijection between the solutions $\boldsymbol{\zeta}_{\boldsymbol{\varepsilon}}$ of the equivariant equations in $\BC[\![q_1, \dots, q_k]\!]^{|\Br|}$ such that $\boldsymbol{\zeta}_{\boldsymbol{\varepsilon}}(0) \in \delta$, and the solutions in $\BA^{|\Br|}\setminus \delta$ for $\Bq$ small enough.
\end{lemma}
\begin{proof}
	By the proof of Theorem~\ref{thm:VI_Hyper_quot_equivariant_formal} in genus zero, the cardinality of the two sets is the same, equal to $\prod_{j=0}^k \frac{r_{j+1}!}{(r_{j+1}-r_j)!}$. Then the discussion above shows that lifting through the implicit function theorem gives an injective map from the set of algebraic solutions to the set of formal ones, and the inverse is given by letting the series converge.
\end{proof}
The equivariant formula of Theorem \ref{thm:VI_Hyper_quot_equivariant} holds true when the right-hand side is evaluated, fixed $\boldsymbol{\varepsilon} \in (\BC^\ast)^n$, at the formal solutions of the equivariant system of equations. By Lemma \ref{lem:formal_alg_bij} we can prove the following
\begin{proposition}\label{prop:genus0eval}
	Let $(\Bq, \boldsymbol{\varepsilon})\in \mathrm{Reg}(p)$.
	Then the generating polynomial of equivariant virtual integrals in \eqref{eq:Intersection_Hquot(C,V)_equivariant_formal} can be evaluated at $(\Bq, \boldsymbol{\varepsilon})$ by computing the right-hand side by summing over the solutions in $\BA^{|\Br|}\setminus\delta$ of the equivariant system of equations with fixed $(\Bq, \boldsymbol{\varepsilon})$. 
\end{proposition}
\begin{proof}
	First of all notice that the equality that we want to prove is
	\begin{align*}
		 \sum_{\Bd\in \BN^k} q_1^{d_1}q_2^{d_2}\cdots q_k^{d_k}
\int_{[\Hquot_{\Bd}]^{\vir}}  
\prod_{j=1}^{k}\prod_{s=1}^{r_j} c_s^{T}(\CE_{j|p}^\vee)^{m_{s,j}}= \left(\prod_{j=1}^k \frac{1}{r_j!}\right)\mathrm{Tr}_p\left(J(\Bz)^{g-1} \prod_{s,j}e_{s}(\Bz_j)^{m_{ij}} \right)
	\end{align*}
	for a generic choice of $(\Bq,\boldsymbol{\varepsilon})$. Note that the left-hand side is a polynomial, while the right-hand side is a rational function by Lemma \ref{lem:trace}. The discussion above shows that there is a small analytic nonempty open subset of $\BA^n \times (\BC^\ast)^k$ where the claim holds, but then the claim holds in general.
\end{proof}

\section{Integrals over products of Hilbert schemes on curves}\label{sec:symmetric_product_of_curves}
We collect notation and lemmas for intersection–theoretic computations on products of Hilbert schemes of points on curves. Although these results are essential for the proof of Theorem~\ref{thm:intro_VI_formula} for Hyperquot schemes on genus-$g$ curves, experts may wish to skip this section on a first reading and return to it as needed.

	 Here $C$ is a smooth projective curve of genus $g$ and we consider, given a non-negative integer $m$ and a degree vector $\Bd \in \mathbb{N}^m$, the product of Hilbert schemes
	\begin{align*}
		X_\Bd \coloneqq  \prod_{i=1}^m C^{[d_i]}.
	\end{align*} 
	This is a smooth projective variety of dimension $\sum_{i=1}^m d_i$ and, on $X_{\Bd} \times C$, we can consider the tautological ideal sheaves $\mathcal{L}_i$ pulled back from $C^{[d_i]}\times C$. We will be interested in some cohomology classes on $X_\Bd$. In this section, we will denote by $H^\ast$ the singular cohomology with real coefficients.
	\subsection{\(\boldsymbol{x}\)-classes and \(\boldsymbol{y}\)-classes}
	For every $i \in \lbrace 1, \dots, m \rbrace$ we can consider the K\"unneth decomposition of the cohomology of $C^{[d_i]} \times C$:
	\begin{align*}
		H^\ast(C^{[d_i]} \times C) \simeq H^\ast(C^{[d_i]}) \otimes_\mathbb{R} H^\ast(C). 
	\end{align*}
	We define the following $x$ and $y$ classes through the K\"unneth decomposition of the first Chern class of a tautological bundle. Having fixed a symplectic basis $1, \delta_1, \dots, \delta_{2g}, \eta$ for $H^\ast(C)$ we define $x_i \in H^2(C^{[d_i]})$ and $y_i^j \in H^1(C^{[d_i]})$ to be the classes satisfying
	\begin{align}\label{kunneth_dec_Li_dual}
		c_1(\mathcal{L}_i^\vee) = x_i \otimes 1 + \sum_{j=1}^{2g} y_i^j \otimes \delta_j + d_i \otimes \eta.
	\end{align}	
	Let's recall a useful fact about intersections of $x$ and $y$-classes, see \cite{thaddeus} or \cite[Section 6]{sinha} for instance:
	\begin{lemma}\label{x_y_intersections}
		On a Hilbert scheme $C^{[d]}$, fix a positive integer $l\leq d$ and consider an integral of the form
		\begin{align*}
			\int_{C^{[d]}} x^{d-l} \prod_{i=1}^{2l} y^{f(i)}
		\end{align*}
		for a function $f:[2l]\rightarrow[2g]$.	It is non-zero if and only if there is an injective function $h:[l]\rightarrow [g]$ so that
		\begin{align*}
			\prod_{i=1}^{2l} y^{f(i)} = \pm \prod_{j=1}^{l} y^{h(j)} \wedge y^{g+h(j)}.
		\end{align*}
		Moreover, in that case
		\begin{align*}
			\int_{C^{[d]}} x^{d-l} \prod_{j=1}^{l} y^{h(j)} \wedge y^{g+h(j)} = \int_{C^{[d]}} x^d = 1.
		\end{align*}
	\end{lemma}
	By a slight abuse of notation we will still denote by $x_i$ and $y_i^j$ the pullbacks of these classes from $C^{[d_i]}$ to $X_\Bd$.
    
    \subsection{Theta classes}\label{sec:theta}
	Consider a vector $\Bw = (w^1, \dots, w^m) \in \mathbb{Z}^m$. The tensor product
	\begin{align*}
		\mathcal{L}_\Bw \coloneqq  \bigotimes_{i = 1}^m \mathcal{L}_i^{\otimes w^i}
	\end{align*}
	defines, after dualisation, a line bundle $\mathcal{L}_\Bw^\vee$ on $X_\Bd \times C$ of relative degree $\langle \Bw, \Bd \rangle \coloneqq  \sum_i w^i d_i$, hence it induces a morphism
	\begin{align*}
		\pi_\Bw : X_\Bd \rightarrow \text{Pic}^{\langle \Bw, \Bd \rangle} (C)
	\end{align*}
	satisfying
	\begin{align}\label{universal_property_pi_w}
		(\pi_\Bw \times \text{Id}_C)^\ast \mathcal{P} = \mathcal{L}_\Bw^\vee
	\end{align}
	for some particular choice of the Poincar\'e line bundle $\mathcal{P}$ on \(\text{Pic}^{\langle \Bw, \Bd \rangle} (C)\times C\), by the universal property of the Jacobian of $C$. We also consider the K\"unneth decomposition of the first Chern class of the dual of this Poincar\'e bundle\footnote{It's immediate to see that the $y$-classes are invariant under tensoring $\mathcal{P}$ with a line bundle coming from the Jacobian, so they are independent of the choice of the Poincar\'e bundle. On the other hand, the class $x$ changes by the first Chern class of the line bundle.}:
		\begin{align}\label{eq:c1P^vee_decomp}
			c_1(\mathcal{P}^\vee) = x \otimes 1 + \sum_{j=1}^{2g} y^j \otimes \delta_j + \langle \Bw, \Bd \rangle \otimes \eta.
		\end{align}
	\begin{definition}\label{def:theta_class}
		The \textit{theta class} on $X_\Bd$ corresponding to $w$ is the pullback of the theta class on the target Jacobian variety: $\theta_\Bw \coloneqq  \pi_\Bw^\ast \theta$.
	\end{definition}
	\begin{remark}\label{vanishing_theta_powers}
		From this definition it's clear that $\theta_\Bw^{j} = 0$ for every $j > g$.
	\end{remark}
	We have a straightforward characterization of this theta class in terms of $y$-classes:
	\begin{lemma}\label{theta_y_classes}
		Given $\Bw \in \mathbb{Z}^m$, the following equality holds true:
		\begin{align*}
			\theta_\Bw = \sum_{j=1}^g \left(\sum_{i=1}^m w^i y_i^j \right) \wedge \left(\sum_{i=1}^m w^i y_i^{j+g} \right).
		\end{align*}
		In particular $\theta_\Bw = \theta_{-\Bw}$.
	\end{lemma}
	\begin{proof}
		It is well-known \cite[page 335]{arbarello_et_al_1} that
		\begin{align*}
			\theta = \sum_{j=1}^g y^j \wedge y^{j+g}.
		\end{align*}
		Then the thesis follows from \eqref{kunneth_dec_Li_dual}, \eqref{universal_property_pi_w}, and \eqref{eq:c1P^vee_decomp}, which indeed imply
		\begin{align*}
			\pi_\Bw^\ast y^j = -\sum_{i=1}^m w^i y^j_i 
		\end{align*}
		for every $j \in \lbrace 1, \dots, 2g \rbrace$.
	\end{proof}
	\subsection{Tautological pushforwards}\label{sec:GRR}
	We consider the projection
	\begin{align*}
		\pi: X_\Bd \times C \rightarrow X_\Bd
	\end{align*}
	and, given a vector $\Bw \in \BZ^m$, we look at the class in $K^0(X_\Bd)$ given by the pushforward $\pi_\ast [\mathcal{L}_\Bw]$. The following result follows from a standard application of the Grothendieck--Riemann--Roch theorem:
	\begin{lemma}\label{lem:Chern_char}
		The Chern character of $\pi_\ast [\mathcal{L}_\Bw]$ is
		\begin{align*}
			\mathrm{ch}(\pi_\ast [\mathcal{L}^\vee_\Bw]) = e^{\langle \Bw, \Bx\rangle}(1-g + \langle \Bw, \Bd \rangle - \theta_\Bw)
		\end{align*}
		where we set $\langle \Bw, \Bx \rangle \coloneqq  \sum_i w^i x_i$.
	\end{lemma}
	\begin{proof}
		Grothendieck--Riemann--Roch tells us that the Chern character of $\pi_\ast [\mathcal{L}^\vee_\Bw]$ is equal to 
		\begin{align}\label{GRR}
			\pi_\ast \left(\text{ch}(\mathcal{L}_\Bw^\vee) \cap \text{Td}(T_\pi)\right),
		\end{align}
        where $T_\pi$ is the (pullback of the) tangent bundle of the curve.
		Notice that, since $\pi$ is a projection, $\pi_\ast$ in cohomology is just integration over $C$.  This means that, in order to compute the pushforward (\ref{GRR}) we can just compute the K\"unneth decomposition of the argument in terms of the symplectic basis of $H^\ast(C)$ and extract the coefficient of $\eta$. We know that 
		\begin{align*}
			\text{Td}(T_C) = 1+ (1-g)\eta,
		\end{align*}
		so we are left to computing $\text{ch}(\mathcal{L}_\Bw^\vee)$ in terms of the $x$ and $y$ classes defined in (\ref{kunneth_dec_Li_dual}). Let $\mathcal{P}'$ be a 
        Poincar\'e bundle on $\text{Pic}^{\langle \Bw,\Bd \rangle}(C)$ such that there is a point $p \in C$ satisfying
		\begin{align*}
			\mathcal{P}'_{\mid_{ \text{Pic}^{\langle \Bw,\Bd \rangle}(C) \times p}} \simeq \mathcal{O}_{\text{Pic}^{\langle \Bw,\Bd \rangle}(C)}.
		\end{align*} 
		This means that $c_1(\mathcal{P}')$ has trivial component coming from $H^2(\text{Pic}^{\langle \Bw,\Bd \rangle}(C))$, and by (\ref{universal_property_pi_w}) 
		\begin{align*}
			c_1(\mathcal{L}_\Bw^\vee) = \langle \Bw, \Bx \rangle + (\pi_\Bw \times \text{Id}_C)^\ast c_1(\mathcal{P}'), 
		\end{align*}
		which, at the level of Chern characters, reads
		\begin{align*}
			\text{ch}(\mathcal{L}_\Bw^\vee) = e^{\langle \Bw,\Bx \rangle} (\pi_\Bw \times \text{Id}_C)^\ast \text{ch}(\mathcal{P}').
		\end{align*}
		The Chern character of this Poincar\'e line bundle is
		\begin{align*}
			\text{ch}(\mathcal{P}') = 1 + \langle \Bw, \Bd \rangle \eta + \gamma - \theta \otimes \eta,
		\end{align*}
		as computed in \cite[page 336]{arbarello_et_al_1}, where $\gamma$ is a class in $H^1(\text{Pic}^{\langle \Bw,\Bd \rangle}(C)) \otimes H^1(C)$. Therefore, we find
		\begin{align*}
			\text{ch}(\mathcal{L}_\Bw^\vee) = e^{\langle \Bw,\Bx \rangle} (1 + \langle \Bw, \Bd \rangle \eta + (\pi_\Bw \times \text{Id}_C)^\ast\gamma - \theta_\Bw \otimes \eta).
		\end{align*}
		Putting all together, the Chern character of $\pi_\ast [\mathcal{L}_\Bw^\vee]$ is the coefficient of $\eta$ in 
		\begin{align*}
			e^{\langle \Bw,\Bx \rangle} (1 + \langle \Bw, \Bd \rangle \eta + (\pi_\Bw \times \text{Id}_C)^\ast\gamma - \theta_\Bw \otimes \eta)(1+(1-g)\eta),
		\end{align*}
		which gives the claimed expression after noticing that $\gamma$ doesn't contribute since its component coming from $C$ is of odd degree.
	\end{proof}
    From this we can obtain an expression for the total Chern class of this pushforward:
    \begin{proposition}\label{chern_class_pushforward}
        The total Chern class of $\pi_\ast [\mathcal{L}_\Bw]$ is
		\begin{align*}
			c(\pi_\ast [\mathcal{L}^\vee_\Bw]) = (1+\langle \Bw, \Bx\rangle)^{1-g + \langle \Bw, \Bd \rangle}\exp\left(- \frac{\theta_\Bw}{1+\langle \Bw, \Bx\rangle}\right).
		\end{align*}
    \end{proposition}
    \begin{proof}
        This follows from Lemma \ref{lem:Chern_char} and the following standard fact: for every vector bundle $V$, if $\mathrm{ch}(V) = (k+\alpha)\mathrm{ch}(L)$ for $k \in \BZ$, $\alpha \in H^2(X)$ and a line bundle $L$, then $c(V) = (1+c_1(L))^k \exp(\frac{\alpha}{1+c_1(L)})$. To see that this holds true, apply the splitting principle and let $\omega_1, \dots, \omega_k$ be the Chern roots of $V\otimes L^{-1}$, so that
        \begin{align*}
            \log(c(V\otimes L^{-1})) = \sum_{i=1}^k \log(1+\omega_i) \quad \text{and} \quad \mathrm{ch}(V\otimes L^{-1}) = \sum_{i=1}^k \exp(\omega_i) = k+\alpha.
        \end{align*}
        From the expression for the Chern character we find that $\sum_{i=1}^k {\omega_i^l}$ is non-zero if and only if $l=0$ or $l=1$, in which case we find $\sum_{i=1}^k \omega_i = \alpha$. Then expanding the logarithms we find that \(\log(c(V\otimes L^{-1})) = \alpha\)
        and therefore
        \begin{align*}
            \prod_{i=1}^k (1 + \omega_i) = c(V\otimes L^{-1}) = \exp(\alpha).
        \end{align*}
        This readily implies that
        \begin{align*}
            c(V) &= \prod_{i=1}^k (1 + \omega_i + c_1(L))
            = (1+c_1(L))^k \prod_{i=1}^k \left(1 + \frac{\omega_i}{1+c_1(L)}\right)\\
            &= (1+c_1(L))^k \exp\left(\frac{\alpha}{1+c_1(L)}\right),
        \end{align*}
        concluding the proof.
    \end{proof}
\subsection{Substitution rules}\label{sec:Substitution_rule} 
    Here we study some integrals of the form
	\begin{align*}
		\int_{X_\Bd} P(x_1, \dots, x_m) Q(\theta_{\Bw_1}, \dots \theta_{\Bw_n})
	\end{align*}
	for specific choices of a polynomial $Q$ in $n$ variables. In this expression, the vectors $\Bw_1,\dots, \Bw_n \in \mathbb{Z}^m$ are possibly not distinct and $P$ is a polynomial in $m$ variables. Our aim in this section is to find a class $R_Q(\Bw_1, \dots, \Bw_n, x_1, \dots, x_m)$, involving the vectors $\Bw_i$ and the $x$-classes, so that the integral above is equal to
	\begin{align*}
		\int_{X_\Bd} P(x_1, \dots, x_m) R_Q(\Bw_1, \dots, \Bw_n, x_1, \dots, x_m).
	\end{align*}
	Notice that we don't aim at an equality at the level of cohomology classes; we just look for classes that integrate to the same number. 
	\begin{definition}
		Let $\alpha_1, \alpha_2$ be two cohomology classes on $X_\Bd$ such that
		\begin{align*}
			\int_{X_\Bd} \alpha_1 \wedge P(x_1, \dots, x_m) = \int_{X_\Bd} \alpha_2 \wedge P(x_1, \dots, x_m)
		\end{align*}
		for every polynomial $P$ in the $x$-classes. We denote this equivalence relation by
		\begin{align*}
			\alpha_1 \rightsquigarrow \alpha_2
		\end{align*}
		and call it a \textit{substitution rule}.
	\end{definition}
	We state here the first substitution rule, which is just the result of a rather lengthy linear algebra computation:
	\begin{lemma}\label{power_1_substitution_rule}
		Let $n$ be a positive integer and consider $n$ vectors $\Bw_1,\dots, \Bw_n \in \mathbb{Z}^m$, from which we construct the cohomology-valued $n \times n$ matrix $M$ whose entries are
		\begin{align*}
			M_a^b \coloneqq  \sum_{c=1}^m w_a^c w_b^c x_c.
		\end{align*}
		Then 
		\begin{align*}
			\prod_{i=1}^n \theta_{\Bw_i} \rightsquigarrow \sum_{A_1 \sqcup \dots \sqcup A_g = [n]} \prod_{j=1}^g \det(M_{A_j,A_j})
		\end{align*}
		is a substitution rule, where the sum is over all ordered partitions of $[n]$ in $g$ (possibly empty) subsets, and $M_{A_i,A_i}$ denotes the submatrix of $M$ obtained by only considering the rows and the columns indexed by elements of $A_i$.
	\end{lemma}
	\begin{proof}
		By Lemma \ref{theta_y_classes} we can express the product of theta classes in terms of the $y$-classes. By considering sums over $n$-tuples as sums over functions from $[n]$ we can write
		\begin{align*}
			\prod_{i=1}^n \theta_{\Bw_i} = \sum_{i: [n]\rightarrow [g]} \sum_{j:[n]\rightarrow [m]} \sum_{k : [n]\rightarrow [m]} \Sigma_{i,j,k},
		\end{align*}
		where the summand $\Sigma_{i,j,k}$ is
		\begin{align*}
			\Sigma_{i,j,k} \coloneqq  \prod_{h=1}^n \left(w_h^{j(h)} w_h^{k(h)}\right) y_{j(h)}^{i(h)} \wedge y_{k(h)}^{i(h) + g}.
		\end{align*}
		By Lemma \ref{x_y_intersections} we know which of these summands will have a chance\footnote{meaning that it can give a non-zero contribution depending on the shape of $P$.} at contributing to the integral. More precisely we know that the integral of $P(x_1,\dots, x_m)\Sigma_{i,j,k}$ vanishes unless there is a permutation $\sigma \in S_n$ such that
		\begin{enumerate}
			\item\label{stabiliser_condition} $i(h) = i(\sigma(h))$ for every $h \in [n]$, and
			\item $k(h) = j(\sigma(h))$ for every $h \in [n]$.
		\end{enumerate}
		If such a $\sigma$ exists, then the corresponding summand is
		\begin{align*}
			\Sigma_{i,j,k} \coloneqq  \Sigma_{i,j,j\circ \sigma} = \prod_{h=1}^n \left(w_h^{j(h)} w_h^{j(\sigma(h))}\right) y_{j(h)}^{i(h)} \wedge y_{j(\sigma(h))}^{i(\sigma(h)) + g}.
		\end{align*}
		By a simple reordering of the variables, we see that
		\begin{align*}
			\prod_{h=1}^n y_{j(h)}^{i(h)} \wedge y_{j(\sigma(h))}^{i(\sigma(h)) + g} = (-1)^{\sigma} \prod_{h=1}^n y_{j(h)}^{i(h)} \wedge y_{j(h)}^{i(h) + g}.
		\end{align*}
		Therefore, we can rewrite the summand $\Sigma_{i,j,\sigma}$ as
		\begin{align*}
			\Sigma_{i,j,\sigma} = \mathrm{sgn}(\sigma) \prod_{h=1}^n \left(w_h^{j(h)} w_h^{j(\sigma(h))}\right) y_{j(h)}^{i(h)} \wedge y_{j(h)}^{i(h) + g}.
		\end{align*}
		By Lemma \ref{x_y_intersections} we find the substitution rule for the summand $\Sigma_{i,j, \sigma}$
		\begin{align*}
			\Sigma_{i,j,\sigma} \rightsquigarrow \mathrm{sgn}(\sigma) \prod_{h=1}^n w_h^{j(h)} w_h^{j(\sigma(h))} x_{j(h)}.
		\end{align*}
		Putting all together, we obtain that 
\begin{align}\label{intermediate_substitution}
			\prod_{i=1}^n \theta_{w_i} \rightsquigarrow \sum_{i: [n]\rightarrow [g]} \sum_{j:[n]\rightarrow [m]} \sum_{\sigma \in (S_n)_i} \mathrm{sgn}(\sigma) \prod_{h=1}^n w_h^{j(h)} w_h^{j(\sigma(h))} x_{j(h)},
		\end{align}
		where $(S_n)_i$ is the stabilizer of $i$ for the action of $S_n$ on $\text{Hom}([n], [m])$, or in other words is the set of $\sigma$ satisfying condition \ref{stabiliser_condition} above. Let's simplify this expression. Keeping $i$ and $\sigma$ fixed, consider the sum
		\begin{align*}
			&\sum_{j: [n] \rightarrow [m]} \prod_{h=1}^n w_h^{j(h)} w_h^{j(\sigma(h))} x_{j(h)} = \sum_{j: [n] \rightarrow [m]} \prod_{h=1}^n w_h^{j(h)} w_{\sigma^{-1}(h)}^{j(h)} x_{j(h)}\\
			=& \sum_{j_1, \dots, j_n=1}^m \prod_{h=1}^n w_h^{j_h} w_{\sigma^{-1}(h)}^{j_h} x_{j_h}
			= \prod_{h=1}^n \sum_{j=1}^m w_h^{j} w_{\sigma^{-1}(h)}^{j} x_j			= \prod_{h=1}^n M_h^{\sigma^{-1}(h)}.
		\end{align*}
		This means that the substitution rule (\ref{intermediate_substitution}) has right-hand side equal to
		\begin{align}\label{intermediate_substitution_2}
			\sum_{i: [n]\rightarrow [g]} \sum_{\sigma \in (S_n)_i} \mathrm{sgn}(\sigma) \prod_{h=1}^n M_h^{\sigma^{-1}(h)}.
		\end{align}
		Now notice that a function $i : [n] \rightarrow [g]$ corresponds to an ordered partition $A_1 \sqcup \dots \sqcup A_g = [n]$ via $A_l \coloneqq  i^{-1}(l)$, and the group $(S_n)_i$ corresponds to the product $\prod_{l=1}^g S_{A_l}$ of automorphisms groups of the subsets forming the partition. Thus, we find that (\ref{intermediate_substitution_2}) is equal to
		\begin{align*}
			&\sum_{A_1 \sqcup \dots \sqcup A_g = [n]} \sum_{\sigma \in \prod_{l=1}^g S_{A_l}} \mathrm{sgn}(\sigma) \prod_{h=1}^n M_h^{\sigma^{-1}(h)}
			= \sum_{A_1 \sqcup \dots \sqcup A_g = [n]} \prod_{l=1}^g \sum_{\sigma_l \in S_{A_l}} (-1)^{\sigma_l} \prod_{h \in A_l} M_h^{\sigma_l^{-1}(h)}\\
			= &\sum_{A_1 \sqcup \dots \sqcup A_g = [n]} \prod_{l=1}^g \det(M_{A_j,A_j}),
		\end{align*}
		where the last equality is the usual expansion of the determinant.
	\end{proof}
	\subsubsection{The monomial substitution rule}
	In Lemma \ref{power_1_substitution_rule} we found a substitution rule once the monomial in the theta classes is expressed in the form $\prod_{i=1}^n \theta_{\Bw_i}$, where each theta class appears with power one but the vectors $\Bw_i$ are allowed to appear repeated multiple times. 
	We now find a more explicit formula for the presentation of monomials of the form $\prod_{i=1}^n \theta_{\Bw_i}^{u_i}$, where powers are allowed:
	\begin{proposition}\label{prop:mono_sub}
		Let $n$ be a positive integer, consider $n$ vectors $\Bw_1, \dots, \Bw_n \in \mathbb{Z}^m$ and a $n$-tuple of non-negative integers $u = (u_1, \dots, u_n)$. Then the substitution rule
		\begin{align*}
			\prod_{i=1}^n \theta^{u_i}_{\Bw_i}\rightsquigarrow \prod_{i=1}^n u_i! \sum_{(A_1,  \dots,A_g)\in \mathcal{P}(u)} \prod_{j=1}^g \det(M_{A_j,A_j})
		\end{align*}
		holds true, where $M$ is the same matrix of Lemma \ref{power_1_substitution_rule} and $\mathcal{P}(u)$ is the set of $g$-tuples $(A_1, \dots, A_g)$ of subsets of $[n]$ satisfying the property
		\begin{align}\label{u_property}
			i \text{ belongs to exactly } u_i \text{ sets among } A_1, \dots, A_g \text{ for every } i \in [n]. 
		\end{align}
	\end{proposition}
	\begin{proof}
		Consider the multiset $U$ whose underlying set is $[n]$ and where $i$ is repeated $u_i$ times. In other words, we consider the set
		\begin{align*}
			U \coloneqq  \left\lbrace 1_1, \dots, 1_{u_1}, 2_1, \dots, 2_{u_2}, \dots, n_1, \dots, n_{u_n} \right\rbrace,
		\end{align*}
		which is endowed with the map $f: U \rightarrow [n]$ given by $f(i_j)\coloneqq  i$.
		Since the monomial we are trying to substitute is 
		\begin{align*}
			\prod_{i=1}^n \theta_{\Bw_i}^{u_i} = \prod_{i \in U} \theta_{\Bw_{f(i)}},
		\end{align*}
		by Lemma \ref{power_1_substitution_rule} we find the substitution rule
		\begin{align}\label{intermediate_monomial_rule}
			\prod_{i=1}^n \theta_{\Bw_i}^{u_i} \rightsquigarrow \sum_{\tilde{A}_1 \sqcup \dots \sqcup \tilde{A}_g = U} \prod_{j=1}^g \det(\tilde{M}_{\tilde{A}_j,\tilde{A}_j})
		\end{align}
		where $\tilde{M}$ is the matrix whose rows and columns are indexed by elements of $U$ and whose entries are $\tilde{M}_i^j = M_{f(i)}^{f(j)}$. This means that the right-hand side of (\ref{intermediate_monomial_rule}) is 
		\begin{align} \label{intermediate_monomial_rule_2}
			\sum_{\tilde{A}_1 \sqcup \dots \sqcup \tilde{A}_g = U} \prod_{j=1}^g \det(M_{f(\tilde{A}_j),f(\tilde{A}_j)}).
		\end{align}
		This immediately shows that the function $f$ must be injective on every $\tilde{A}_j$ in order for $\det(\tilde{M}_{\tilde{A}_j, \tilde{A}_j})$ to be non-zero, hence we can restrict the sum over partitions of $U$ so that each $\tilde{A}_i$ contains at most one copy of each element of $[n]$. Clearly the function
		\begin{align*}
			\left\lbrace{\text{partitions } \tilde{A}_1 \sqcup \dots \sqcup \tilde{A}_g = U \text{ s.t. } f \text{ is injective on each } \tilde{A}_i }\right\rbrace \xrightarrow{f} \mathcal{P}(u)
		\end{align*}
		has fibers of cardinality $\prod_{i=1}^n u_i!$, since given a partition of $U$ one can shuffle all the copies of each element of $[n]$ to produce another partition of $U$ mapping to the same subsets of $[n]$. Therefore (\ref{intermediate_monomial_rule_2}) coincides with the claimed substitution rule once we rename $A_j \coloneqq  f(\tilde{A}_j)$.
	\end{proof}
	\subsubsection{The exponential substitution rule}
	In our applications, we will substitute polynomials in the theta classes that appear as truncated exponentials. For this kind of polynomials, the substitution rule simplifies:
	\begin{proposition}\label{exponential_substitution_rule_n}
		Let $f_1, \dots, f_n \in \mathbb{C}[x_1, \dots, x_m]$ be polynomials in the $x$-classes, from which we form the diagonal $n \times n$ matrix 
		\begin{align*}
			C \coloneqq  \text{diag}\left(f_1, \dots, f_n\right).
		\end{align*}
		Then we have the substitution rule
		\begin{align*}
			\prod_{i=1}^n \exp(f_i \theta_{\Bw_i}) \rightsquigarrow \det(\mathbb{I}_n + CM)^g.
		\end{align*}
	\end{proposition}
	\begin{proof}
		By the definition of the exponential (note that $\theta_{\Bw_i}^{g+1}=0$ as discussed in Remark \ref{vanishing_theta_powers}) and Proposition~\ref{prop:mono_sub} we find that
		\begin{align*}
			\prod_{i=1}^n \exp(f_i \theta_{\Bw_i}) =& \sum_{u_1, \dots, u_n = 0}^g \left(\prod_{i=1}^n \frac{f_i^{u_i}}{u_i!}\right) \prod_{j=1}^n \theta_{\Bw_j}^{u_j}\\
			\rightsquigarrow& \sum_{u_1, \dots, u_n = 0}^g \left(\prod_{i=1}^n f_i^{u_i}\right) \sum_{(A_1,  \dots,A_g)\in \mathcal{P}(u)} \prod_{j=1}^g \det(M_{A_j,A_j})\\
			=& \sum_{u_1, \dots, u_n = 0}^g \sum_{(A_1,  \dots,A_g)\in \mathcal{P}(u)} \prod_{j=1}^g \det((C M)_{A_j,A_j}).
		\end{align*}
        Note that 
		\begin{align*}
			\bigcup_{u_1, \dots, u_n=0}^g \mathcal{P}(u) = \mathcal{P}([n])^g,
		\end{align*}
		where the right-hand side is the Cartesian product of $g$ copies of the power set of $[n]$. Thus we have the substitution rule
		\begin{align*}
			\prod_{i=1}^n \exp(f_i \theta_{\Bw_i}) \rightsquigarrow& \sum_{A_1,  \dots,A_g \in \mathcal{P}([n])} \prod_{j=1}^g \det((C M)_{A_j,A_j})
			\\
			=& \left(\sum_{A \in \mathcal{P}([n])} \det(( CM)_{A,A}) \right)^g= \det(\mathbb{I}_n + CM )^g,
		\end{align*}
		where the last equality is the well-known expansion of the determinant of $\mathbb{I} + Z$ as the sum of the principal minors of $Z$.
	\end{proof}
	Finally, we can reshape this result into the most useful form for our applications:
	\begin{corollary}[Exponential substitution rule]\label{exponential_substitution_rule_m}
		Let $f_1, \dots, f_n \in \mathbb{C}[x_1, \dots, x_m]$ be polynomials in the $x$-classes, and consider the diagonal $m \times m$ matrix 
		\begin{align*}
			X \coloneqq  \text{diag}\left(x_1, \dots, x_m\right).
		\end{align*}
		The substitution rule
		\begin{align*}
			\prod_{i=1}^n \exp(f_i \theta_{\Bw_i}) \rightsquigarrow \det(\mathbb{I}_m + XG(c))^g
		\end{align*}
		holds true, where $G(c)$ is the $m \times m$ matrix
		\begin{align*}
			G(c)\coloneqq  \sum_{i=1}^n f_i \Bw_i^T \cdot \Bw_i,
		\end{align*}
		and the vectors $\Bw_i \in \mathbb{Z}^m$ are considered as row vectors.
	\end{corollary}
	\begin{proof}
		Consider the $n \times m$ matrix $W$ with entries $W_a^b \coloneqq  w_a^b$, or, in other words, the one obtained by treating the vectors $\Bw_i$ as rows of a matrix. Then we immediately find that the matrix $M$ defined in Lemma \ref{power_1_substitution_rule} is
		\begin{align*}
			M = W X W^T,
		\end{align*}
		hence by Proposition \ref{exponential_substitution_rule_n} we find
		\begin{align*}
			\prod_{i=1}^n \exp(f_i \theta_{\Bw_i}) \rightsquigarrow \det(\mathbb{I}_n + CWXW^T)^g.
		\end{align*}
		By Sylvester's determinant theorem, stating that
		\begin{align}
			\det(\mathbb{I}_n + AB) = \det(\mathbb{I}_m + BA)
		\end{align}
		for any $n \times m$ matrix $A$ and $m \times n$ matrix $B$, we obtain that
		\begin{align*}
			\prod_{i=1}^n \exp(f_i \theta_{\Bw_i}) \rightsquigarrow \det(\mathbb{I}_m + XW^TCW)^g
		\end{align*}
		concluding the proof.
	\end{proof}

    \subsubsection{The Jacobian substitution rule}\label{jacobian_substitution_rule}
	Let $\varepsilon_1, \dots, \varepsilon_n \in \mathbb{C}$. Chosen the constants of Corollary \ref{exponential_substitution_rule_m} as 
    $$
    f_i= (\langle \Bx, \Bw_i \rangle + \varepsilon_i)^{-1} \qquad \text{where} \qquad \langle \Bx, \Bw \rangle \coloneqq  \sum_{j=1}^m w^j x_j,
    $$ 
    we can give a nice description of the substitution rule:
	\begin{corollary}\label{jacobian_substitution_rule}
		We have the substitution rule
		\begin{align*}
			\prod_{i=1}^n \exp\left(\frac{ \theta_{\Bw_i}}{\langle \Bx, \Bw_i \rangle + \varepsilon_i}\right) \rightsquigarrow \prod_{j=1}^{m}x_{j}^g\det\left(\frac{1}{t_j}\frac{\partial t_{j}}{\partial x_{k}} \right)^g,
		\end{align*}
		where  $t : \mathbb{C}^m \dashrightarrow \mathbb{C}^m$ is defined as
	\begin{align*}
		t_j(\Bx) \coloneqq  x_j\prod_{i=1}^n (\langle \Bx, \Bw_i \rangle + \varepsilon_i)^{w_i^j} \qquad \forall j \in [m].
	\end{align*}
	\end{corollary} 
	\begin{proof}
		It's immediate to compute that
		\begin{align*}
			\frac{\partial t_j}{\partial x_k} = {t_j}(\Bx) \left(\frac{\delta_{j,k}}{x_{j}}+\sum_{i=1}^n \frac{w_i^j w_i^k}{\langle \Bx, \Bw_i \rangle + \varepsilon_i}\right) = \frac{t_{j}(\Bx)}{x_{j}}\left(\delta_{j,k}+x_{j}G(f)_{j}^k\right) .
		\end{align*}
		where we used the notation of Corollary \ref{exponential_substitution_rule_m}.
	\end{proof}

\section{Virtual integrals in higher genus}\label{sec:higher_genus_case} We now turn our attention to the higher genus case.  In this section, we apply virtual localization \cite{graber-pandharipande} to reduce the integrals in Theorem~\ref{thm:VI_Hyper_quot_equivariant} to a sum of integrals over products of Hilbert schemes of points on $C$. Compared to the genus zero calculation of Section \ref{sec:genus_zero_calculation}, the higher genus case is more complicated, since these localized integrals involve additional theta classes. We developed the tools needed to handle these classes in Section~\ref{sec:symmetric_product_of_curves}. 
\subsection{Cohomology classes on $\boldsymbol{\fix_{\sigma,D}}$}
We continue to use the notation introduced in Sections \ref{sec:genus_zero_calculation} and \ref{sec:symmetric_product_of_curves}. In particular recall that, given $\sigma \in \mathfrak{S}_\Br$ and a splitting degree $D$, there is a corresponding fixed locus $X_{\sigma,D}$  in $\Hquot_\Bd(C,\Br,V)$ isomorphic to the product of Hilbert schemes \eqref{eq:product_hilb_sch}, and that the different choices for $\sigma$ yield fixed loci which are abstractly isomorphic, but are embedded in the Hyperquot scheme in different ways.
\subsubsection{Chern classes}
Whenever \(\CL\) is a vector bundle on a product scheme \(X\times C\) and $p \in C$ is a point, we will denote by \(\CL_{\mid p}\) the restriction of $\CL$ to \(X\times\{p\}\).
Let $\CL_{s,j}$ be the universal ideal subsheaf over the product $C^{[d_{s,j}-d_{s,j+1}]}\times C$. We will denote by the same letter its pullback to $X_{\sigma,D}\times C$, by a slight abuse of notation. Note that 
\begin{equation}\label{eq:K_to_L}
    \CL_{s,j}^\vee=\CK_{s,j}^\vee\otimes \CK_{s,j+1},
\end{equation}
where $\CK_{s,j}$ is the $s$-th line bundle summand of $\CE_j \vert_{X_{\sigma, D}\times C}$ as in \eqref{eq:splitting_Ei}. 
For $\sigma = 1 \in \mathfrak{S}_\Br$ we define the classes
\[
x_{s,j}\coloneqq c_1\left({\CL_{s,j| p}^\vee}\right)\in H^{2}(X_{1,D})\quad\text{for all }1\le j\le k\text{ and }\ 1\le s\le r_j.
\]
From \eqref{eq:K_to_L}, we obtain
\[
z_{s,j}\coloneqq c_{1}^T\left( \CK_{s,j\mid p}^\vee \right) = x_{s,j}+x_{s,{j+1}}+\cdots+x_{s,k}+\varepsilon_{s}  \in H^2(X_{1,D}),
\]
and to be consistent we define $z_{s,k+1}:= c_{1}^T\left( \pi_C^{*}M_{s\mid p}^\vee \right) =\varepsilon_{s}$.
For an arbitrary element $\sigma \in \mathfrak{S}_\Br$ we immediately find that $c_{1}^T\left( \CK_{s,j\mid p}^\vee \right) = \sigma \cdot z_{s,j}$ in $H^2_T(X_{\sigma,D})$.

\subsubsection{Theta classes}
Let $\Bw$ be a vector that assigns an integer to each Hilbert scheme appearing as a factor of $X_{\sigma,D}$, that is, $\Bw = (w^{c,i}) \in \bb{Z}^{r_1}\oplus\dots\oplus\bb{Z}^{r_k}$. We consider the corresponding line bundle
$$
\CL_{\Bw} = \bigotimes_{i=1}^k \bigotimes_{c=1}^{r_i} \CL_{c,i}^{\otimes w^{c,i}},
$$
whose dual \(\CL_\Bw^\vee\) induces a morphism
\[
\pi_{\Bw}:X_{\sigma,D}\to \Pic^{\langle \Bw, \Be\rangle}(C)
\]
where $\Be\in \bb{Z}^{r_1}\oplus\dots\oplus\bb{Z}^{r_k}$ is the dimension vector given by $e^{s,j} = d_{s,j}-d_{s,j+1}$. For any vector $\Bw$ we set,  as in Definition \ref{def:theta_class},
\[
\theta_{\Bw}\coloneqq \pi_\Bw^*\theta\in H^2(\fix_{\sigma,D},\BZ)
\]
where $\theta\in \Pic^{\langle \Bw, \Be\rangle}(C)$ is the theta class. Note that $\theta_\Bw=\theta_{-\Bw}$ and $\theta_\Bw^{g+1}=0$.
\begin{remark}
    We have $ \CK_{s,j}\otimes \CK_{s,k+1}^\vee =\CL_{s,j}\otimes\dots\otimes\CL_{s,k}= \CL_{\Bw_{s,j}}$  for the associated vector 
    $$\Bw_{s,j}=\left(\Bzero^{(1)},\dots,\Bzero^{(j-1)},\Bdelta^{(j)}_s,\Bdelta^{(j+1)}_s,\dots, \Bdelta^{(k)}_s\right)\in \bb{Z}^{r_1}\oplus\dots\oplus\bb{Z}^{r_k},$$
        where \(\Bzero^{(j)}\) is the zero vector in \(\bb{Z}^{r_j}\) and \(\Bdelta^{(j)}_s=(0,\dots,0,1,0\dots,0)\in\bb{Z}^{r_j}\) denotes the $s$-th element of the canonical basis of $\BZ^{r_j}$, with $1$ placed in the $s$-th position.  
\end{remark}

\subsection{Euler class of virtual normal bundle}
 
All the cohomology classes that will appear in the localization calculation are expressions of the $\Bx$ and $\theta$ classes. In particular, Proposition~\ref{chern_class_pushforward} readily implies the following
\begin{proposition}\label{prop:Euler_class_pair}
    Let $1\le j\le k$, $s\le r_j$, $u\le r_{j+1}$ and $s\ne u$. On $X_{\sigma, D}$ we find    \begin{align}
        e_T\left(\pi_*\left(\CK_{s,j}^\vee\otimes\CK_{u,j}\right)\right)&=\sigma \cdot \left(z_{s,j}-z_{u,j}\right)^{d_{s,j}-d_{u,j}+1-g}\exp\left({-\frac{\theta_{\Bw_{s,j}-\Bw_{u,j}}}{z_{s,j}-z_{u,j}}}\right);\label{eq:Euler_class_same_j}
        \\ e_T\left(\pi_*\left(\CK_{s,j}^\vee\otimes\CK_{u,{j+1}}\right)\right)&= \sigma \cdot \left(z_{s,j} -z_{u,{j+1}} \right)^{d_{s,j}-d_{u,{j+1}}+1-g}\exp\left(-\frac{\theta_{\Bw_{s,j}-\Bw_{u,j+1}}}{z_{s,j} -z_{u,{j+1}}}
        \right).\label{eq:Euler_class_j,j+1}
\end{align}
\end{proposition}

By a short explicit computation, these equalities allow us to compute the Euler class of the virtual normal bundle of $X_{\sigma, D} \hookrightarrow \Hquot_\Bd(C,\Br,V)$, as outlined below. Recall the definition \eqref{eq:def_of_R} of the polynomials 
$$R_{j}(z)=(z-z_{1,{j}})(z-z_{2,{j}})\cdots(z-z_{r_j,{j}})\quad \text{for all } 0\le j\le k+1.
$$ 
 \begin{proposition}\label{prop:Euler_class_virtual_normal}
    The equivariant Euler class of the virtual normal bundle is
    \begin{align}\label{eq:Euler_class_Normal_bundle}
        \frac{1}{e_T(\Nor)} = \sigma \cdot (-1)^{\beta}J(\mathbf{x},\boldsymbol{\theta},\boldsymbol{\varepsilon})\cdot\prod_{j=1}^{k}\prod_{s\le r_j}x_{s,j}^{d_{s,j}-d_{s,j+1}+1}
        \frac{R_j'(z_{s,j})  R_{{j-1}}(z_{s,j})^{d_{s,j}} }{R_{{j+1}}(z_{s,j})^{d_{s,j}+1} } ,
    \end{align}
    where the sign is given by
    $$\beta=\sum_{j\leq k}d_j(r_j-r_{j-1}-1)$$
    and the class $J(\mathbf{x},\boldsymbol{\theta},\boldsymbol{\varepsilon})$ is
    $$J(\mathbf{x},\boldsymbol{\theta},\boldsymbol{\varepsilon}) = \prod_{j \leq k}\prod_{s\leq r_j}\parens*{\frac{R_{j+1}(z_{s,j})}{x_{s,j}R_{j}'(z_{s,j})}}^{g}\prod_{\substack{u\leq r_{j+1}\\ u\ne s}}\exp\parens*{\frac{\theta_{\Bw_{s,j}-\Bw_{u,j+1}}}{z_{s,j} -z_{u,{j+1}}}}.$$
    The right-hand side should be interpreted as a rational function of the formal variables $\Bx$, $\boldsymbol{\theta}$ and $\boldsymbol{\varepsilon}$, which induces a localized equivariant cohomology class in $H^\ast(\fix_{\sigma,D})\otimes_{H_T^\ast(\pt)}\Gamma$ only after clearing the common factors in the numerators and denominators.
\end{proposition}
\begin{proof}
    We substitute the equivariant Euler classes computed in Proposition~\ref{prop:Euler_class_pair} in the expression \eqref{eq:Normal_bundle_K-theory} for the virtual normal bundle $\Nor$. The proof is identical to the one of Proposition \ref{prop:Euler_class_normal_genus_0}.
\end{proof}

    \subsection{Proof of Theorem~\ref{thm:VI_Hyper_quot_equivariant}}\label{Sec:proof_of_thm_1}    
    By applying the substitution rules of Section~\ref{sec:Substitution_rule}, we can simplify the integrals involving the factor $J(\mathbf{x},\boldsymbol{\theta},\boldsymbol{\varepsilon})$:
    
\begin{lemma}\label{lem:J(x,theta)-substitution}
For any splitting data $D$ and polynomial $H(\bf{x},\boldsymbol{\varepsilon})\in\Gamma[\Bx]$, we have
    \begin{align}\label{eq:J(x,theta)-substitution}
        \int_{\fix_{1,D}} \
        H(\mathbf{x},\boldsymbol{\varepsilon}) J(\bf{x},\boldsymbol{\theta},\boldsymbol{\varepsilon}) 
        &= \int_{\fix_{1,D}}H(\mathbf{x},\boldsymbol{\varepsilon})J(\mathbf{x},\boldsymbol{\varepsilon})^g, 
    \end{align}
    where 
    \[
    J(\mathbf{x},\boldsymbol{\varepsilon})=\prod_{j \leq k}\prod_{s\le r_j}\frac{R_{{j-1}}(z_{s,j})}{ R_{j}'(z_{s,j})}\cdot \det\left(\frac{\partial q_{s,j}}{\partial x_{u,m}}\right)\quad \text{and}\quad q_{s,j}\coloneqq (-1)^{r_j-r_{j-1}-1}\frac{R_{{j+1}}(z_{s,j})}{R_{{j-1}}(z_{s,j})}.
    \]
\end{lemma}
\begin{proof}
    Define $\varepsilon_{su}:=\varepsilon_u-\varepsilon_s$ and the auxiliary vectors $\Bw_{suj}\coloneqq \Bw_{s,j}-\Bw_{u,j+1}$ in $\bb{Z}^{r_1}\oplus\dots\oplus\bb{Z}^{r_k}$. Then
        \[
        \frac{\theta_{\Bw_{s,j}-\Bw_{u,j+1}}}{z_{s,j} -z_{u,{j+1}}}=\frac{\theta_{\Bw_{suj}}}{\langle \Bx,\Bw_{suj}\rangle+\varepsilon_{su}}
        \]
        and the left-hand side of \eqref{eq:J(x,theta)-substitution} can be written as
        \begin{equation}\label{eq:expression_theta}
            \int_{X_{1,D}} \left(H(\Bx,\boldsymbol{\varepsilon}) \prod_{j\leq k}\prod_{s\le r_j}\parens*{\frac{R_{j+1}(z_{s,j})}{x_{s,j}R_{j}'(z_{s,j})}}^{g}\right)\left(\prod_{j\leq k} \prod_{s\leq r_j} \prod_{\substack{u\leq r_{j+1}\\ u\ne s }}\exp\parens*{\frac{\theta_{\Bw_{suj}}}{\langle \Bx,\Bw_{suj}\rangle+\varepsilon_{su}}}\right).
        \end{equation} 
        We can now apply Corollary~\ref{jacobian_substitution_rule} to convert the $\theta$-classes in this integral into expressions in the $x$-classes. By setting
        \begin{align*}
            t_{c,i}(\textbf{x},\boldsymbol{\varepsilon})&\coloneqq x_{c,i}\prod_{j=1}^k \prod_{s=1}^{r_j} \prod^{r_{j+1}}_{\substack{u=1\\ u\ne s }}\left(\langle \Bx,\Bw_{suj}\rangle+\varepsilon_{su}\right)^{w_{suj}^{c,i}}\quad\text{for }c\le r_i,\  i\le k
        \end{align*}
         and applying Corollary~\ref{jacobian_substitution_rule} we see that \eqref{eq:expression_theta} coincides with
         \[
         \int_{X_{1,D}}H(\Bx,\boldsymbol{\varepsilon}) \prod_{j=1}^{k}\prod_{s\le r_j}\left(\frac{R_{j+1}(z_{s,j})}{x_{s,j}R_{j}'(z_{s,j})}\right)^{g} \prod_{i=1}^k \prod_{c\leq r_j}
        x_{c,i}^g \det \left(\frac{1}{t_{c,i}(\mathbf{x,\boldsymbol{\varepsilon}})}\frac{\partial t_{c,i}(\mathbf{x,\boldsymbol{\varepsilon}})}{\partial x_{c',i'}} \right)^g.
         \]
        Note that 
\begin{align*}
    t_{c,i}(\mathbf{x},\boldsymbol{\varepsilon})=x_{c,i}\frac{\prod_{\substack{j\leq i \\ c\leq r_j}} \prod_{\substack{u\le r_{j+1}\\ u\ne c}}\parens*{z_{c,j}-z_{u,j+1}}}{\prod_{\substack{j \leq i-1 \\ c\leq r_{j+1}}}\prod_{\substack{s\leq r_j \\ s\ne c}}\parens*{z_{s,j}-z_{c,j+1}}}=(-1)^{\alpha(c,i)}\prod_{\substack{j\le i \\ c \leq r_j}}q_{c,j}
\end{align*}
for some sign \(\alpha(c,i)\). This immediately implies that 
        \[
        \det \left(\frac{1}{t_{c,i}(\mathbf{x},\boldsymbol{\varepsilon})}\frac{\partial t_{c,i}(\mathbf{x},\boldsymbol{\varepsilon})}{\partial q_{c',i'}} \right) = \prod_{c\le r_i,\ i\le k}\frac{1}{q_{c,i}},
        \]
        and the lemma follows from the chain rule.
\end{proof}
The procedure of summing the contributions of the fixed loci mirrors the genus–zero analysis in Section~\ref{sec:genus_zero_VI}. We briefly recall that setup and then complete the proof of Theorem~\ref{thm:VI_Hyper_quot_equivariant_formal}, which is the formal version of our main result. Theorems~\ref{thm:VI_Hyper_quot_equivariant} and \ref{thm:intro_VI_formula} follow by the discussion in Section~\ref{sec:Equations}. 

\begin{proof}[Proof of Theorem~\ref{thm:VI_Hyper_quot_equivariant_formal}]
    Let $Q=\prod_{j=1}^{k}\prod_{i=1}^{r_j}c_i(\CE_{j|p}^\vee)^{m_{i,j}}$. The virtual localization formula of \cite{graber-pandharipande} implies
\[
\int_{[\Hquot_{\Bd}(C,\Br, V)]^{\vir}}  Q = \sum_{\sigma,D} \int_{\fix_{\sigma,D}} 
        \frac{Q|_{\fix_{\sigma,D}}}{e_{T}(\Nor)},
\]
where the equivariant Euler class $e_{T}(\Nor)$ of the virtual normal bundle was computed in Proposition~\ref{prop:Euler_class_virtual_normal}. Explicitly, this gives
\[
(-1)^{\beta} \sum_{\sigma,D} \sigma\cdot \int_{\fix_{1,D}} Q\vert_{X_{1,D}}\cdot J(\mathbf{x},\boldsymbol{\theta},\boldsymbol{\varepsilon})\cdot\prod_{j\leq k}\prod_{s\le r_j}x_{s,j}^{d_{s,j}-d_{s,j+1}+1}
        \frac{R_j'(z_{s,j})  R_{{j-1}}(z_{s,j})^{d_{s,j}} }{R_{{j+1}}(z_{s,j})^{d_{s,j}+1} }.
\]
Using Lemma~\ref{lem:J(x,theta)-substitution}, the factor $J(\mathbf{x},\boldsymbol{\theta},\boldsymbol{\varepsilon})$ can be replaced by $J(\mathbf{x},\boldsymbol{\varepsilon})^g$, which depends only on the $x$-variables and the genus $g$. Since the rule for integrating $x$-classes, described in Lemma~\ref{x_y_intersections}, is identical to the one in the case of $C=\mathbb{P}^1$, the above integral reduces to a residue:
\[
\int_{[\Hquot_{\Bd}(C,\Br, V)]^\vir} Q = (-1)^{\beta} \sum_{\sigma,D} \sigma \cdot \Res_{\mathbf{x}=\mathbf{0}} Q\vert_{X_{1,D}} \prod_{j\le k}\prod_{i\le r_j} \frac{ R_{j}'(z_{i,j}) \, R_{{j-1}}(z_{i,j})^{d_{i,j}} }{ R_{{j+1}}(z_{i,j})^{d_{i,j}+1} } \cdot J(\mathbf{x},\boldsymbol{\varepsilon})^g.
\]
Note that the new factor $J(\mathbf{x},\boldsymbol{\varepsilon})^g$ is independent of the splitting data $D$, and consequently the Lagrange--Bürmann argument from the genus-zero case in Section~\ref{sec:genus_zero_VI} applies unchanged. The proof is thus completed by following the proof of Theorem~\ref{thm:VI_Hyper_quot_equivariant} for $C=\BP^1$ verbatim, with the additional observation that $J(\mathbf{x},\boldsymbol{\varepsilon}) = J$ when expressed in terms of the $z$-variables.
\end{proof}

\section{Explicit formulas}
In this section, we specialize Theorem~\ref{thm:intro_VI_formula} to several cases and obtain explicit combinatorial formulas. In two basic instances, the Hyperquot scheme of points on $C$ and the two-step type $\Br=(1,n-1)$, we derive simple combinatorial expressions that manifest the positivity of the virtual intersection numbers. 

\begin{definition}
Fix a genus $g$ curve $C$, ranks $\Br=(r_1,r_2,\dots,r_k)$, and a vector bundle $V$ of degree $e$ and rank $n$.  
For any tuple of natural numbers $$\mathbf{m}=\{m_{i,j}\in\BN : 1\le j\le k,\;1\le i\le r_j\},$$  
we consider the multivariate Laurent polynomial
\begin{equation}\label{eq:B^m_def}
    \mathbf{B}_{g,\Br,n,e}^{\mathbf{m}}(q_1,q_2,\dots,q_k)
= \sum_{\Bd\in \BZ^k} q_1^{d_1}q_2^{d_2}\cdots q_k^{d_k}
\int_{[\Hquot_{\Bd}]^{\vir}}  
\prod_{j=1}^{k}\prod_{i=1}^{r_j} c_i(\CE_{j|p}^\vee)^{m_{i,j}}.
\end{equation}
Note that when $V$ is the trivial bundle, \eqref{eq:B^m_def} is a polynomial.
\end{definition}

\subsection{Explicit formulas for Hyperquot schemes of points on $C$}
Fix a positive integer $k$ and vector bundle $V$ of rank $n$ on a genus $g$ curve $C$. For a tuple of non-negative integers $\Bd=(d_1,d_2,\dots,d_k)$, let $\Hquot_\Bd \coloneqq \Hquot_\Bd(C,\Br,V)$ denote the Hyperquot scheme of points, i.e. $\Br=(n)^k\coloneqq (n,n,\dots,n)$, parametrizing successive zero-dimensional quotients
\[
V\twoheadrightarrow F_1 \twoheadrightarrow F_2 \twoheadrightarrow \cdots \twoheadrightarrow F_k,
\]
where $F_j$ has support of length $d_j$ for all $1\le j\le k$. Note that for \(\Hquot_\Bd\) to be non-empty,  the lengths of the support for subsequent quotients must decrease or stay the same at each step, that is, $d_1\ge d_2\ge \dots\ge d_k$. In this case, \(\Hquot_\Bd\)
is a smooth irreducible scheme of dimension $d_1 n$, see \cite[Theorem 1.4, Proposition 2.1]{mr2}.

\begin{corollary}\label{thm:punctual}
    For any tuple $\Bm =(m_{i,j}: 1\le j\le k, \, 1\le i\le n  )$, 
    \begin{align*}
        \mathbf{B}_{g,(n)^k,n,e}^{\Bm}(q_1,q_2,\dots,q_k) =\begin{cases}
            \prod_{j=1}^k\alpha_j^{m_{n,j}}& \text{if }m_{i,j}=0\quad \forall i<n ,1\le j\le k
            \\
            0& \text{otherwise}.
        \end{cases} 
    \end{align*}
   where 
   \begin{equation}\label{eq:def_alpha_j}
       \alpha_j\coloneqq q_1q_2\cdots q_j(1+q_{j+1}+q_{j+1}q_{j+2}+\cdots +q_{j+1}q_{j+2}\cdots q_{k}) \quad \text{ for }1\le j\le k.
   \end{equation}
   Note that the above formula does not depend on the genus $g$ of $C$ and degree $e$ of $V$.  
\end{corollary}
\begin{proof}
Let us first fix $e=0$ and a generic $(q_1,q_2,\dots,q_k) \in (\BC^\ast)^k$. Consider the system of variables $\Bz_{j}=(z_{1,j},\dots,z_{n,j})$ for $1\le j\le k$, and set by definition $\Bz_0 := \lbrace \rbrace$ and $z_{i,k+1} := 0$ for all $i$. Consider the polynomials
\begin{equation}\label{eq:simple_P_point}
    P_{j}(X)= \prod_{i=1}^{n}(X-z_{i,j+1})+q_j\prod_{i=1}^{n}(X-z_{i,j-1})
    \qquad\text{for}\quad 1\le j\le k.
\end{equation}
In order to apply Theorem \ref{thm:intro_VI_formula}, we need to solve the system of equations 
\begin{align*}
    P_j(z_{i,j}) = 0 \qquad \text{for all } 1\leq j \leq k \quad \text{and}\quad 1\leq i\leq n
\end{align*}
defined in \eqref{eq:Bethe_P(X)}. By Proposition~\ref{number_of_solutions}, there is a unique non-degenerate solution $(\Bz_1,\dots,\Bz_k)=(\boldsymbol{\zeta}_1,\dots,\boldsymbol{\zeta}_k)$ to this system of equations up to permuting the entries in each $\boldsymbol{\zeta}_j$ by elements of $S_n$. In particular, the non-degenerate solutions form a free orbit under the action of $(S_n)^k$. 

Fixed one such solution $(\boldsymbol{\zeta}_1,\dots,\boldsymbol{\zeta}_k)$, since the expression on the right-hand side of the formula in Theorem \ref{thm:intro_VI_formula} is invariant under the action of $(S_n)^k$, we find that
\begin{equation}\label{eq:simplified_formula_points}
    \mathbf{B}_{g,(n)^k,n,e}^{\Bm}(q_1,\dots,q_k) = J(\boldsymbol{\zeta}_1, \dots, \boldsymbol{\zeta}_k)^{g-1} \prod_{j=1}^k \prod_{i=1}^n e_i(\boldsymbol{\zeta}_{j})^{m_{i,j}}.
\end{equation}

In the special case where $\Bm=\boldsymbol{0}$, we can directly compute
\[
 \mathbf{B}_{g,(n)^k,n,0}^{\boldsymbol{0}}(q_1,\dots,q_k)=1 \qquad \forall g \geq 0
\]
since the Hyperquot scheme of points on $C$ is zero-dimensional exactly when $d_1=\cdots=d_k=0$, in which case it consists of a single reduced point. From this and \eqref{eq:simplified_formula_points} we immediately obtain $J(\boldsymbol{\zeta}_1,\dots,\boldsymbol{\zeta}_k)=1$, which in particular implies that the integrals don't depend on the genus $g$. 

In the general case $\Bm \neq 0$, in order to apply \eqref{eq:simplified_formula_points}, we have to evaluate the symmetric polynomials $e_i$ at the chosen solution. To do so, consider the auxiliary polynomials
$$
R_j(X)\coloneqq \prod_{i=0}^n(X-\zeta_{i,j}) = \sum_{i=1}^n (-1)^i e_i(\boldsymbol{\zeta}_j)X^{n-i} \qquad \text{ where } \boldsymbol{\zeta_j} = (\zeta_{1,j}, \dots, \zeta_{n,j})
$$
for each $1\le j\le k$. Since $\boldsymbol{\zeta}_j$ is the complete list of roots of $P_{j}(X)\vert_{\Bz = \boldsymbol{\zeta}}$, we find that $R_j$ divides this polynomial, and since they are both of degree $n$ there must exist a constant $c_j \in \BC$ so that $P_j(X)\vert_{\Bz = \boldsymbol{\zeta}} = c_j R_j$. Equation \eqref{eq:simple_P_point} becomes 
\begin{equation}\label{eq:new_P_points}
    P_{j}(X)\vert_{ \Bz = \boldsymbol{\zeta}} = R_{j+1}(X) +q_j R_{j-1}(X)
    \qquad\text{for}\quad 1\le j\le k,
\end{equation}
and the constant can be easily determined from  by comparing the coefficients of $X^n$, so that $c_1:=1$ and $c_j := (1+q_j)$ for $j \geq 2$. From \eqref{eq:new_P_points} we obtain
\begin{align*}
   (1+q_j) R_j(X)&= R_{j+1}(X)+q_jR_{j-1}(X) \qquad \text{for all } 1 \leq j \leq k,
\end{align*}
where we set $R_0 := 1$ and $R_{k+1}(X):= X^n$.
Solving the above linear relations for the monic polynomials $R_j(X)$, we obtain
$$R_j(X)=X^n+(-1)^{n}\alpha_j,$$
    where $\alpha_j$ is defined in \eqref{eq:def_alpha_j}, and hence the elementary symmetric polynomials satisfy $e_n(\boldsymbol{\zeta}_{j}) =        \alpha_j$ and $e_i(\boldsymbol{\zeta}_{j})=0$ whenever $i<n$.

    For an arbitrary degree $e$, Proposition~\ref{prop:elementary_mod} implies
    \[
     \mathbf{B}_{g,(n)^k,n,e}^{\tilde\Bm}(q_1,\dots,q_k) =  \alpha_k \cdot\mathbf{B}_{g,(n)^k,n,e-1}^{\Bm}(q_1,\dots,q_k)
    \]
    where $\tilde{m}_{k,n} = m_{k,n} + 1$ and $\tilde{m}_{i,j} = m_{i,j}$ for all the other indices. The claim now follows from the result for $e=0$.
\end{proof}
\begin{remark}
 Recall that for a vector bundle $E$ on a projective scheme $X$, the Segre polynomial of $E$ is given by $s_{t}(E) =1/(1+c_1(E)t+\dots +c_{n}(E)t^n)$. Corollary~\ref{cor:Segra_integral} is a direct consequence of Corollary~\ref{thm:punctual}, by expressing the Segre polynomial in terms of the Chern classes of $E$ as above.
\end{remark}

\subsection{Explicit formulas for $\Br=(1,n-1)$}
Fix a smooth curve $C$ of genus $g$, $V$ be a vector bundle of rank $n$ and degree zero. Let $\Br=(1,n-1)$. For any multi-degree $\Bd=(d_1,d_{2})$, the virtual dimension of the Hyperquot scheme $\Hquot_\Bd=\Hquot_\Bd(C,\Br,V)$ is given by 
\[
\vdim\Hquot_{\Bd} = (n-1)(d_1+d_{2}) - (2n-3)\bar{g},
\]
where $\bar{g}:=g-1$.
For any non-negative integers $\ell$ and $m_1,m_2,\dots,m_{n-1}$, we are interested in evaluating the virtual invariants
\[
\mathbf{B}_{g,\Br,n,0}^{\ell,\Bm}(q_1,q_{2}) \coloneqq      \sum_{d_1,d_{2}\in \mathbb{N}}
q_1^{d_1} q_{2}^{d_{2}} \int_{[\Hquot_{\Bd}]^{\vir}}  c_1(\CE_{1|p}^\vee)^{\ell} \prod_{i=1}^{n-1}c_i(\CE_{2|p}^\vee)^{m_{i}},
\]
where we denote $(\ell,\Bm)=(\ell,m_1,\dots,m_{n-1})$, instead of calling it $\Bm$.

To apply Theorem~\ref{thm:intro_VI_formula}, we must first solve system of 
equations in~\eqref{eq:Bethe_system_intro}; we return to the proof of Corollary~\ref{cor:intro_r=(1,n-1)} at the end of this section. For ease of notation, let $x=z_{11}$ and $y_i=z_{i,n-1}$ for $1 \le i \le n-1$. 
The system of equations in \eqref{eq:Bethe_system_intro} specializes to
\begin{align}\label{eq:system_(1,n-1)}
    &&P_{2}(y_1)&\equiv y_1^n+(-1)^{n}q_{2}(y_1-x)=0,\nonumber\\
   P_{1}(x)\equiv\prod_{i=1}^{n-1}(x-y_i)-q_1=0&,\qquad&  \vdots\ \ \ & \ \ \ \ \ \ \ \ \ \ \ \ \ \ \ \ \ \ 
    \vdots\hspace{5em}
    \\
    &&P_{2}(y_{n-2})&\equiv y_{n-2}^n+(-1)^{n}q_{2}(y_{n-2}-x)=0,\nonumber\\
&&P_2(y_{n-1})&\equiv y_{n-1}^n+(-1)^{n}q_{2}(y_{n-1}-x)=0.\nonumber
\end{align}

\begin{lemma}\label{lem:FL(1,n-1,n)-solutions}
   Fix a tuple $(q_1,q_2) $ of generic complex numbers. Every non-degenerate solution of \eqref{eq:system_(1,n-1)},
\[
  (x=\zeta, \; y_1=\eta_1, \dots, y_{n-1}=\eta_{n-1}),
\]
is of the following form:
 let \(w\) be an \(n\)-th root of \(q_1/q_2\), then
    \begin{enumerate}
        \item[(i)] \(\zeta\) is an \((n-1)\)-th root of \(q_1(1+w^{-1})\),
        \item[(ii)] \(\eta_1,\eta_2,\dots,\eta_{n-1}\) are distinct roots
    of 
\[
Y^{n-1}+\left(-\frac{\zeta}{w}\right)Y^{n-2}+\cdots+\left(-\frac{\zeta}{w}\right)^{n-2}Y+\left(-\frac{\zeta}{w}\right)^{n-1} +(-1)^{n}q_{2}=0.
\]
    \end{enumerate}
    
\end{lemma}
\begin{proof}

Suppose  $(\zeta,\eta_1,\dots,\eta_{n-1})$ is a non-degenerate solution to the above system of equations. Using all the equations in \eqref{eq:system_(1,n-1)}, we observe that 
\[
(\eta_1\eta_2\cdots\eta_{n-1})^n =(-1)^{n-1}(\eta_1-\zeta)\cdots(\eta_{n-1}-\zeta)q_2^{n-1}= q_1q_2^{n-1}.
\]
Consider the polynomial $$R_{2}(Y)= Y^n+(-1)^{n}q_{2}(Y-\zeta).$$
Note that $\eta_1,\dots,\eta_{n-1}$ are roots of $R_2$. We denote $\eta_n$ for the missing root of $P_{2}(Y)$, that is $\{\eta_1,\eta_{2},\dots,\eta_n\}$ is the complete list of roots. 

Evaluating the polynomial at $Y=\zeta$, we obtain
\begin{equation}\label{eq:P_n-1(x)}
    \zeta^n= R_{2}(\zeta) = (\zeta-\eta_{1})\cdots(\zeta-\eta_{n-1}) (\zeta-\eta_n) = q_{1}(\zeta-\eta_n).
\end{equation}
On the other hand, evaluating the polynomial at \(Y=0\) and exponentiating to the $n$-th exponent, we have the identity
\begin{align*}
    &\parens*{q_{2}\zeta }^n = R^n_{2}(0)= \parens*{(-1)^n\eta_{1}\eta_2\cdots \eta_{n}}^n = (-1)^nq_1q_2^{n-1}\eta_n^n
\end{align*}
which implies \begin{align}\label{eq:omega}
    \eta_{n}=-\zeta/w\text{ with }w\text{ an }n\text{-th root of }q_1/q_2.
\end{align} Substituting \eqref{eq:omega} back in \eqref{eq:P_n-1(x)}, we obtain the identity \begin{align}\label{eq:x}
    \zeta^{n-1}= {q_1(1+w^{-1})}.
\end{align}

Finally, we express elementary symmetric polynomials in $\eta_{1},\eta_2,\dots,\eta_{n-1}$ in terms of $q_2,w$ and $\zeta$. Analyzing the coefficients of the polynomial $R_{2}(Y)$, we have
\begin{align*}
    e_{n}(\eta_1, \eta_2, \dots, \eta_n) = -\zeta q_{2}, \quad
    e_{n-1}(\eta_1, \eta_2, \dots, \eta_n) = - q_{2}, 
\end{align*}
and $ e_{i}(\eta_1, \eta_2, \dots, \eta_n) = 0$ {for all} $1 \le i \le n-2$.
Using \eqref{eq:omega} and \eqref{eq:x}, we obtain
\begin{align}\label{eq:e_i()FL(1,n-1,n)}
    e_{n-1}(\eta_1,\eta_2,\dots,\eta_{n-1})&=&&q_{2}w&=&\left(\frac{\zeta}{w}\right)^{n-1}-q_{2} \nonumber
    \\
    e_{n-2}(\eta_1,\eta_2,\dots,\eta_{n-1})&=&\frac{w}{\zeta}\cdot &q_{2}(1+w)&=&\left(\frac{\zeta}{w}\right)^{n-2}\nonumber
    \\
    e_{n-3}(\eta_1,\eta_2,\dots,\eta_{n-1})&=&\left(\frac{w}{\zeta}\right)^2\cdot &q_{2}(1+w)&=&\left(\frac{\zeta}{w}\right)^{n-3}\\
    \vdots&&&&&\nonumber\\
    e_{1}(\eta_1,\eta_2,\dots,\eta_{n-1})&=&\left(\frac{w}{\zeta}\right)^{n-2}\cdot &q_{2}(1+w)&=&\left(\frac{\zeta}{w}\right)\nonumber.
\end{align}
We have thus obtained that there are at most \(n\cdot(n-1)\cdot(n-1)!\) tuples \((\zeta,\eta_1,\dots,\eta_{n-1})\) satisfying (i) and (ii) in the statement of Lemma~\ref{lem:FL(1,n-1,n)-solutions}. It follows from Proposition~\ref{number_of_solutions} that they must be all the non-degenerate solutions to \eqref{eq:system_(1,n-1)}.
\end{proof}

\begin{lemma}\label{lem:J(x)FL(1,n-1,n)}
    Continuing with the notation in Lemma~\ref{lem:FL(1,n-1,n)-solutions}, for any choice of $w$ and $\zeta$. Let \(\boldsymbol{\eta}:=\parens*{\eta_1,\dots,\eta_{n-1}}\), we have
    \[
    J(\zeta,\boldsymbol{\eta})
    =n(n-1)q_{2}w\zeta^{n-2}.
    \]
\end{lemma}
\begin{proof}
Suppose  $(\zeta,\eta_1,\dots,\eta_{n-1})$ is a non-degenerate solution. The term $J(\zeta,\boldsymbol{\eta})$ in Theorem~\ref{thm:intro_VI_formula} in our case is explicitly given by
\begin{align*}
     \parens*{\det 
     \begin{bmatrix}
     \frac{\partial P_1(x)}{\partial x}&\frac{\partial P_1(x) }{\partial y_1}&\frac{\partial P_1(x) }{\partial y_2}&\cdots&   \frac{\partial P_1(x) }{\partial y_{n-1}}\\
     (-1)^{n-1}q_{2}&\frac{\partial P_2(y_1)}{\partial y_1}&0&\cdots&0\\
     (-1)^{n-1}q_{2}&0&\frac{\partial P_2(y_2)}{\partial y_2}&\cdots&0\\
     \vdots&\vdots&&\ddots&\vdots\\
     (-1)^{n-1}q_{2}& 0&0&\cdots&\frac{\partial P_2(y_{n-1})}{\partial y_{n-1}}
     \end{bmatrix}
     \cdot\prod_{1\le i\ne j\le n-1}(y_{i}-y_j)^{-1}}_{x=\zeta,\ \mathbf{y}=\boldsymbol{\eta}}
\end{align*}
Recall the polynomial
$$R_2(Y)=Y^n+(-1)^nq_2(Y-\zeta)=  (Y-\eta_1)(Y-\eta_2)\cdots(Y-\eta_n),$$ where $\eta_n$ is the missing root as in the proof of Lemma~\ref{lem:FL(1,n-1,n)-solutions}. Note that 
\[
\prod_{1\le i\ne j\le n-1}(\eta_{i}-\eta_j) = \frac{(-1)^{n-1}}{R_{2}'(\eta_n)}\prod_{i=1}^{n-1}R_{2}'(\eta_i)
\]
and 
\[
\frac{\partial P_2(y_i)}{\partial y_i}\bigg|_{x=\zeta,\ \mathbf{y}=\boldsymbol{\eta}}=R_2'(y_i)|_{\mathbf{y}=\boldsymbol{\eta}}=R_2'(\eta_i).
\]
Using the identity
$$\frac{\partial P_1(x)}{\partial x} = -\sum_{i=1}^{n-1}\frac{\partial P_1(x) }{\partial y_i},$$
we expand the determinant along the first row and obtain
\begin{align*}
    J(\zeta,\boldsymbol{\eta}) &=(-1)^{n-1}R_{2}'(\eta_n)\left[\sum_{i=1}^{n-1}\left( -1 +  (-1)^{n}\frac{q_{2}}{R_{2}'(\eta_i)} \right)\cdot\frac{\partial P_1(x) }{\partial y_i}\bigg|_{x=\zeta,\ \mathbf{y}=\boldsymbol{\eta}}\right]
    \\&=(-1)^{n-1}R_{2}'(\eta_n)\cdot\left[\sum_{i=1}^{n-1} \frac{-R_2'(\eta_i)+(-1)^nq_2}{(\zeta-\eta_i)R_{2}'(\eta_i)} \right]\cdot \prod_{j=1}^{n-1}(\zeta-\eta_j)\\
    &=(-1)^{n-1}R_{2}'(\eta_n)\cdot\left[\sum_{i=1}^{n-1} \frac{n\eta_i^{n-1}}{(\zeta-\eta_i)R_{2}'(\eta_i)} \right]\cdot \prod_{j=1}^{n-1}(\zeta-\eta_j)\\
    &=(-1)^{n-1}R_{2}'(\eta_n)\cdot\left[\sum_{i=1}^{n-1} \frac{n\cdot (-1)^{n}q_{2}}{\eta_i R_{2}'(\eta_i)} \right]\cdot q_1,
\end{align*}
where in the last equality, we used $q_1=\prod_{j=1}^{n-1}(\zeta-\eta_j)$ and $\eta_i^n=(-1)^nq_{2}(\zeta-\eta_i)$. We note that $$\eta_iR'_{2}(\eta_i) =n\eta_i^n+(-1)^nq_{2}\eta_i
=(-1)^nq_{2}(n-1)\left(\frac{n\zeta}{n-1}-\eta_i \right),$$
and hence summing over all roots of $R_{2}$ and using the identity $
\zeta^{n-1}= {q_1(1+w^{-1})}$, we have
\begin{align}\label{eq:n/zeta}
    \sum_{i=1}^{n} \frac{n\cdot (-1)^{n}q_{2}}{\eta_i R_{2}'(\eta_i)}&= \frac{n}{n-1}\cdot \frac{R_{2}'\left(\frac{n\zeta}{n-1}\right)}{R_{2}\left(\frac{n\zeta}{n-1}\right)}
    =\frac{n}{\zeta}.
\end{align}
Recall from \eqref{eq:omega} that $\eta_{n}=-\zeta/w$, and thus
\begin{equation}\label{eq:R_2(eta_n)}
    R_{2}'(\eta_n)= n\eta_n^{n-1}+(-1)^nq_{2}=(-1)^{n-1}q_{2}(n(1+w)-1)
\end{equation}
Subtituting \eqref{eq:n/zeta} and \eqref{eq:R_2(eta_n)} back in the required expression, we have
\begin{align*}
     J(\zeta,\boldsymbol{\eta})&= q_1q_{2} (n(1+w)-1)\left[\frac{n}{\zeta} -\frac{nw}{\zeta (n(1+w)-1)} \right]\\
     &= \frac{n(n-1)q_1 q_{2}(1+w)}{\zeta}.
\end{align*}
We obtain the required expression using the identity $q_1(1+w)=w\zeta^{n-1}$.
 \end{proof}

\begin{proof}[Proof of Corollary~\ref{cor:intro_r=(1,n-1)}]
    Applying Theorem~\ref{thm:intro_VI_formula}, and using 
    Lemmas~\ref{lem:FL(1,n-1,n)-solutions} and \ref{lem:J(x)FL(1,n-1,n)} (in particular, the identities in \eqref{eq:e_i()FL(1,n-1,n)}), we obtain that 
    \begin{align*}
        \mathbf{B}_{g,\Br,n,0}^{\ell,\mathbf{m}}(q_1,q_{2})&=n^{\bar{g}}(n-1)^{\bar{g}}\sum_{w}\left(\sum_{\zeta} \zeta^\ell\cdot \prod_{i=1}^{n-2}\left(\frac{\zeta}{w}\right)^{im_i}\cdot (q_{2}w)^{m_{n-1}}\cdot  (q_{2}w\zeta^{n-2})^{\bar{g}} \right)\\
  & = n^{\bar{g}}(n-1)^{\bar{g}}q_{2}^{m_{n-1}+\bar{g}} \sum_{w} w^{m_{n-1}+\bar{g}-\sum_{i=1}^{n-2}im_i}
    \left(\sum_{\zeta} \zeta^{\ell +\sum_{i=1}^{n-2}im_i+(n-2)\bar{g}} \right).
\end{align*}
The sum is taken over all roots of the equation $w^n=q_1/q_{2}$ and subsequently all roots of $\zeta^{n-1} = (q_1(1+w^{-1}))$.
 Substitute $$\sum_{i=1}^{n-2}im_i = (n-1)(d-m_{n-1})-\ell-(2n-3)\bar{g}$$ to the above expression, we obtain that \(\mathbf{B}_{g,\Br,n,0}^{\ell,\mathbf{m}}(q_1,q_{2})\) equals
\begin{align*}
    &n^{\bar{g}}(n-1)^{\bar{g}}q_{2}^{m_{n-1}+\bar{g}} \sum_{w} w^{nm_{n-1}+\ell-(n-1)d+(2n-2)\bar{g}}
    \left(\sum_{\zeta} \zeta^{(n-1)(d-m_{n-1}-\bar{g})} \right)\\
    =&n^{\bar{g}}(n-1)^{g}q_1^{d-m_{n-1}-\bar{g}}q_{2}^{m_{n-1}+\bar{g}} \sum_{w} w^{nm_{n-1}+\ell-(n-1)d+(2n-2)\bar{g}}
    (1+w^{-1})^{(d-m_{n-1}-\bar{g})}\\
    =&n^{\bar{g}}(n-1)^{g}q_1^{d-\bar{g}}q_{2}^{\bar{g}} \sum_{w} w^{\ell-(n-1)d+(2n-2)\bar{g}}
    (1+w^{-1})^{(d-m_{n-1}-\bar{g})} \\
    =&n^{\bar{g}}(n-1)^{g}q_1^{d-\bar{g}}q_{2}^{\bar{g}} \sum_{w} w^{m_{n-1}+\ell+2n\bar{g}-\bar{g}-nd}
     (1+w)^{(d-m_{n-1}-\bar{g})} \\
    =&n^{\bar{g}}(n-1)^{g}q_1^{\bar{g}}q_{2}^{d-\bar{g}} \sum_{w} w^{m_{n-1}+\ell-\bar{g}}
     (1+w)^{(d-m_{n-1}-\bar{g})}. 
\end{align*}
To simplify the above, we note that 
\begin{align*}
\sum_{w^n=q_1/q_{2}}w^{k}=
\begin{cases}
n(q_1/q_{2})^{k/n}& n\mid k,\\
0,& n\nmid k.
\end{cases}
\end{align*}
By the binomial theorem,
\[
\sum_{w^n={q_1}/{q_{2}}} w^{\ell+m_{n-1}-\bar{g}} (1+w)^{d-\bar{g}-m_{n-1}}
=
\sum_{\substack{j\in\mathbb{Z} \\ 0 \le jn -m_{n-1}- \ell + \bar{g}}}
\binom{d-\bar{g}-m_{n-1}}{jn - \ell-m_{n-1} + \bar{g}} \cdot n
\left( \frac{q_1}{q_{2}} \right)^{j}.
\]
Substituting this identity into the previous expression completes the proof.
 \end{proof}

\subsection{Counting maximal subsheaves}
If $V$ is a general stable bundle $V$ of degree $e$ and $\vdim \quot_{d}(C,r,V) = 0$, the Quot scheme $\quot_{d}(C,r,V)$ is a smooth zero-dimensional scheme consisting of all the rank-$r$ subbundles of $V$ having maximal slope \cite{holla}. Let $$m(n,e,r,g) = \deg[\quot_{d}(C,r,V)]$$ be the number of such subbundles, which is explicitly computed by the Vafa-Intriligator formula for Quot schemes. Recall the iterative description of the Hyperquot scheme $\Hquot_\Bd(C,\Br,V)$ using relative Quot schemes
\[
\quot_{d_1-d_2}(C,r_1,\CE_2)\xrightarrow{\phi_1}\cdots\xrightarrow{\phi_{k-2}}
\quot_{d_{k-1}-d_k}(C,r_{k-1},\CE_k)\xrightarrow{\phi_{k-1}}
\quot_{d_k}(C,r_k,V)\xrightarrow{\phi_k} \{\pt\}.
\]
\begin{corollary}\label{cor:stable}
Let $V$ be a vector bundle on $C$ of rank $n$ and degree $e$, and assume that $g\geq 2$.
Fix $\Br = (r_1, \dots, r_k)$ and $\Bd = (d_1, \dots, d_k)$ so that $\vdim \phi_k = 0$. For any rank-$r_k$ vector bundle $E_k$ of degree $e-d_k$, we have
\[
\int_{[\Hquot_{\Bd}(C,\Br,V)]^{\vir}} \prod_{j=1}^{k-1}\prod_{i=1}^{r_j} c_i(\CE_{j|p}^\vee)^{m_{i,j}}
\;=\;
m(n,e,r_k,g)\cdot
\int_{[\Hquot_{\Bd'}(C,\Br',E_k)]^{\vir}}
\prod_{j=1}^{k-1}\prod_{i=1}^{r_j} c_i(\CE_{j|p}^\vee)^{m_{i,j}},
\]
where $\Bd'=(d_1-d_k,\dots,d_{k-1}-d_k)$ and $\Br'=(r_1,\dots,r_{k-1})$.
\end{corollary}
\begin{proof}
    The formula of Theorem~\ref{thm:intro_VI_formula} shows that, for fixed monomial in the Chern classes, the virtual integral on the Hyperquot scheme depends only on $g,n, e, \Br$ and $\Bd$. If $g\geq 2$, there exists a stable vector bundle of degree $e$ and rank $n$, hence we can assume that $V$ is a general stable bundle and that $\quot_{d_k}(C,r_k,V) = \lbrace V_1, \dots, V_m \rbrace$ with $m\coloneqq  m(n,e,r_k,g)$. In this case, we can compute the pushforward to $\quot_{d_k}(C,r_k,V)$ as    \begin{align}\label{eq:pushforward}
        (\phi_{k-2} \circ \cdots \circ \phi_1)_\ast  \prod_{j=1}^{k-1}\prod_{i=1}^{r_j} c_i(\CE_{j|p}^\vee)^{m_{i,j}} = \sum_{j=1}^m [V_j] \int_{[\Hquot_{\Bd'}(C,\Br',V_j)]^{\vir}}
        \prod_{j=1}^{k-1}\prod_{i=1}^{r_j} c_i(\CE_{j|p}^\vee)^{m_{i,j}}.
    \end{align}
    Since the integrals on the right only depend on $g,\Bd$ and $\Br$, we can assume that $V_j = E_k$ for all $j$, and the formula follows by pushing forward (\ref{eq:pushforward}) to a point, which algebraically consists in setting $[V_j]$ to $1$ for all $j$.
\end{proof}
\begin{remark}
Recall the coefficient $6^{13}$ of $q_1^{10}q_2^8$ in the polynomial computed in Remark \ref{rmk:dim_zero_polynomial} for a curve of genus $13$ and a rank-$3$ bundle $V$ of degree zero. Corollary~\ref{cor:stable} gives a geometric interpretation of this number: When $d_1=10$ and $d_2=8$, the virtual integral
    \begin{align*}
        \int_{[\Hquot_{(10,8)}(C,(1,2),V)]^\vir}1 = m(3,0,2,13) \cdot m(2,-8,1
        ,13) = 3^{13} \cdot 2^{13} = 6^{13},
    \end{align*}
    where the numbers of maximal subbundles are computed by \cite[Corollary 4.5]{holla}.
    Indeed one can easily check that $\vdim\ \phi_1=\vdim\ \phi_2=0$ in this case.
\end{remark}
\begin{remark}
    Note that the hypothesis $g\geq 2$ is needed to ensure the existence of a stable bundle of rank $n$ and degree $e$. Strictly speaking, such stable bundles also exist for $g=1$ in the coprime case $\gcd(n,e)=1$ by work of Atiyah (see \cite[Thm 1 and below]{tu_semistable_bundles}), but in this case the Quot scheme has virtual dimension $\vdim \quot_d(C,r,V) = n d + (r-n)e$, which is never zero since $n$ and $e$ are coprime.
\end{remark}

\bibliographystyle{amsalpha}
	\bibliography{biblio}

\end{document}